\documentclass[10pt]{amsart}

\usepackage[pagebackref,colorlinks=true]{hyperref}  

\usepackage{amsfonts}
\usepackage{amsmath,amsthm,amssymb,amscd,enumerate,eucal,url}

\usepackage{color}

\usepackage{eucal,url,amssymb,verbatim,
enumerate,amscd,
}

\numberwithin{equation}{section}

\newtheorem{thrm}{Theorem}[section]
\newtheorem{lemma}[thrm]{Lemma}
\newtheorem{prop}[thrm]{Proposition}

\usepackage[margin=1in]{geometry}
\newcommand{\R}{\mathbb{R}}
\newcommand{\lc}{\langle}
\newcommand{\rc}{\rangle}
\newcommand{\br}[1]{\overline{#1}}
\newcommand{\dt}{\,.\,}
\newcommand{\spc}{\ \,}

\def\sideremark#1{\ifvmode\leavevmode\fi\vadjust{
\vbox to0pt{\hbox to 0pt{\hskip\hsize\hskip1em
\vbox{\hsize23mm\tiny\raggedright\pretolerance10000
\noindent #1\hfill}\hss}\vbox to8pt{\vfil}\vss}}}

\begin{document}

\begin{abstract}
Following the Cartans's original method of equivalence supported by methods of parabolic geometry, we provide a complete solution for the equivalence problem of  quaternionic contact structures, that is, the problem of finding a complete system of differential invariants for two quaternionic contact manifolds to be locally diffeomorphic. This includes an explicit construction of the corresponding Cartan geometry and detailed information on all curvature components.
\end{abstract}



\title[On the equivalence of quaternionic contact structures]
{On the equivalence of quaternionic contact structures}

\date{\today}

\author{Ivan Minchev}
\address[Ivan Minchev]{University
of Sofia, Faculty of Mathematics and Informatics, blvd. James Bourchier 5, 1164 Sofia, Bulgaria;
Department of Mathematics and Statistics, Masaryk University, Kotlarska 2, 61137 Brno,
Czech Republic}
\email{minchev@fmi.uni-sofia.bg}

\author{Jan Slov\'{a}k}
\address[Jan Slov\'{a}k]{
Department of Mathematics and Statistics, Masaryk University, Kotlarska 2, 61137 Brno,
Czech Republic}
\email{slovak@math.muni.cz}

\maketitle

\setcounter{tocdepth}{2} \tableofcontents

\section{Introduction}

There is the series of important geometries naturally appearing at the generic
hypersurfaces in projective spaces. The Klein models $G\to G/P$ 
for all of them are
spheres, i.e. the conformal Riemannian sphere 
$S^n\subset \mathbb RP^{n+1}$, the CR-sphere $S^{2n+1}\subset \mathbb
CP^{n+1}$, and the quaternionic contact sphere $S^{4n+3}\subset \mathbb
HP^{n+1}$, respectively, or other nice homogeneous spaces in
the cases of other than positive definite signatures.

All these geometries appear as boundaries of domains, carrying a lot of
information -- let us mention the conformal horizons in mathematical physics,
the boundaries of domains in complex analysis and function theory, and the
boundaries of quaternionic-K\"ahler domains.    

The corresponding Lie algebras enjoy very similar algebraic structures with
gradings
$$
\mathfrak g = \mathfrak g_{-2}\oplus \mathfrak g_{-1}\oplus \mathfrak g_0
\oplus \mathfrak g_1 \oplus \mathfrak g_2,
$$ 
where $\mathfrak g_0$ further 
splits as $\mathfrak h\oplus \mathfrak g_0'$, as indicated symbolically 
in the matrix (the $*$ entries mean those computed from the symmetries of
the matrix)
\begin{equation*}\label{basic_grading}
\left(\begin{array}{c|c|c}
\mathfrak h & \mathfrak g_1 & \mathfrak g_2
\\[4mm]
\hline
\mathfrak g_{-1} & \mathfrak g_0' &  *
\\  [4mm]                                        
\hline
\mathfrak g_{-2} & * & *
\end{array}
\right),
\end{equation*}
The corresponding Lie algebras $\mathfrak g$ are  $\mathfrak{so}(p+1,q+1)$,
$\mathfrak{su}(p+1,q+1)$, and $\mathfrak{sp}(p+1,q+1)$. Thus viewing them as
matrix algebras over $\mathbb K=\mathbb R, \mathbb C, \mathbb H$,  they
always have columns and rows of width $1, n, 1$, respectively, and
$\mathfrak h=\mathbb K$, $\mathfrak g_{-1}=\mathbb K^n$, $\mathfrak
g_1=\mathbb K^{n*}$, $\mathfrak g_2$ is the imaginary part of $\Bbb K$ (thus
vanishing in the case $\mathbb K=\mathbb R$) and $\mathfrak g_0'$ is the algebra of the same
type as $\mathfrak g$ of signature $(p,q)$. 

All these geometries fit into the class of Cartan geometries with $G$
semisimple and $P$ parabolic and thus there is the rich general theory
explaining the cohomological character of basic invariants. The
constructions of the relevant normalized Cartan connections and detailed
analysis of its curvature is well known for decades in the first two cases,
but much less is known in the case of quaternionic contact geometries.
This is perhaps due to the much higher complexity of the analysis to be
expected.

Our aim is to fill this gap and provide a full analogy to the construction
of the normal Cartan connection by Chern and Moser in their paper
\cite{ChM}, including detailed information on all curvature components. 
We shall come back to further motivation for this endeavor below. Let us only mention here that given a quaternionic contact structure on a manifold, Ivanov and Vassilev \cite{IV} have provided an explicit expression for a certain tensorial quantity whose vanishing is equivalent to the existence of a local diffeomorphism mapping the starting quaternionic contact structure to that of the standard Sphere (the Klein model). It was shown in \cite{Alt} that the tensorial quantity constructed in \cite{IV} corresponds precisely to the harmonic curvature of the associated Cartan connection. 

We
are going to deliver our construction in terms of the most classical
exterior calculus and it should be completely understandable without direct 
insight
into the general cohomological structure of the curvature.
But of course, it is this knowledge which allows us to know in advance that the
individual steps will work. 

The first step in understanding the difficulty is the definition of the
geometry itself. While the conformal geometry has got the trivial
filtration on the tangent bundle and the geometry itself is one of the most
classical G-structures, 
 the CR geometry is already defined by a contact
distribution 
$
T^{-1}M\subset TM
$
with a further reduction of the
associated graded tangent bundle to a structure group respecting the
additional complex structure on the distribution.

Let us now clarify the CR case carefully from a more abstract point of view.
The book \cite{CS} can be consulted for both the general theory and 
details on CR structures.

The homogeneity zero component of the first Lie algebra cohomology 
$H^1(\mathfrak
g/\mathfrak p,\mathfrak g)_0$ is nontrivial in this case, thus the extra
reduction of the frame bundle. 
Moreover, the second cohomology $H^2(\mathfrak
g/\mathfrak p,\mathfrak g)$ is nontrivial in homogeneities one and two (in
all dimensions $\operatorname{dim}M \ge 5$). In
particular, there is no cohomology in homogeneity zero and thus the
algebraic Lie bracket on $\operatorname{Gr}TM$ induced by the Lie bracket of
vector fields has to coincide with the Lie bracket on $\mathfrak
g_{-2}\oplus \mathfrak g_{-1}$ from the Lie algebra in question. This means
that algebraic bracket on $\operatorname{Gr} TM$ has to be the imaginary
part of a hermitian form on $T^{-1}M$. Further, the cochains generating the 
second cohomology in homogeneity one are of the type $\Lambda^2\mathfrak
g_{-1}^*\otimes \mathfrak g_{-1}$ (obstructing the integrability of the complex
structure $J$ on $T^{-1}M$). This is the torsion of the canonical Cartan
curvature, which automatically vanishes in the case of embedded
hypersurfaces in $\mathbb C^{n+1}$. 

Another important cohomological information is the automatic vanishing of
the cohomology with cochains of the type $\mathfrak
g_{-2}^*\otimes\mathfrak g_{-1}^*\otimes \mathfrak g_{-2}$, also in
homogeneity one. Indeed, any choice of a contact form $\theta$
defining the CR-distribution will split the tangent bundle to
$TM=T^{-1}M\oplus T^{-2}M$, identify $T_{-2}M$ with $M\times \mathbb R$ (via
the Reeb vector field of $\theta$),
the algebraic Lie bracket will get a
symplectic form, and thus together with the complex structure $J$ we also
get the 
Levi-Civita connection for all derivatives in the directions of $T^{-1}M$. 
The latter cohomological information implies that all these objects are
fully in compliance with the canonical Cartan connection for the structure.

Now, we come to the quaternionic contact geometries.  Following the work of Biquard \cite{Biq1}, this is a type of geometric structure describing the Carnot-Carath{\'e}odory geometry of the boundary at infinity of quaternionic K\"ahler manifolds. The quaternionic contact geometry also became a crucial geometric tool in finding the extremals and the best constant in the $L^2$ Folland-Stein Sobolev-type embedding on the quaternionic Heisenberg group \cite{IMV2}, \cite{IMV}, \cite{IMV3}.
Here, the first
cohomology $H^1(\mathfrak
g/\mathfrak p,\mathfrak g)$ appears only in negative homogeneities, thus
the entire geometry is completely defined by the distribution
$T^{-1}M\subset TM$ of codimension three. This means that if there were a
pre-quaternionic vector space structure on $T^{-1}M$ for which an algebraic
bracket $[\ ,\ ]_{\mbox{\tiny alg}}\!:\Lambda^2T^{-1}M\to TM/T^{-1}M$ would be an
imaginary part of a hermitian form, then this structure is unique. 

Again, let us skip the lowest dimension first, i.e. 
$\operatorname{dim}M\ge 11$. Then the second
cohomology $H^2(\mathfrak
g/\mathfrak p,\mathfrak g)$  has two components. One of them appears in
homogeneity zero, with cochains of the
type $\Lambda \mathfrak g_{-1}^*\otimes\mathfrak g_{-2}$. This is a tricky
point, since this means that if this part of the torsion is non-zero, then
the geometric structure is defined by the distribution and a choice of an
algebraic bracket $[\ ,\ ]_{\mbox{\tiny alg}}$ such that the latter bracket
allows for a pre-quaternionic structure such that it gets the imaginary part
of a hermitian form. Of course, the difference of 
this bracket and the standard one (defined by the Lie bracket of vector fields) 
should be normalized (co-closed in the terms of the Lie algebra cohomology, cf. the
appendix).

Dealing with generic hypersurfaces in $\mathbb H^{n+1}$, we should thus
expect three very much different possibilities. First, the distributions
with the inherited pre-quaternionic structure and the Lie bracket 
will satisfy all
the properties (i.e. the bracket will be the imaginary part of suitable
hermitian form on $T_{-1}M$). This is extremely restrictive and, as shown in 
\cite{IMV5},  can happen only if $M$ is locally isomorphic to the homogeneous model---the 3-Sasakian sphere.

The second possibility is to require that the inherited distribution and the
Lie bracket are allowing for a pre-quaternionic structure as above.
This assumption is much less rigid, but we need an explicit construction and
knowledge of the canonical Cartan connection in order to be able to deal
with such examples properly. This has been the initial main motivation for
this paper and
we shall come back to such special class of hypersurfaces in $\mathbb H^{n+1}$
in another future work. In the lowest dimension, seven, this is the most
general case, but there we have also a homogeneity one torsion component.

We shall not deal with the most difficult third option
here at all. Let us just mention a preprint by Stuart Armstrong devoted to 
a general
class of Cartan geometries with this kind of behavior, \cite{Arm}. 

Thus, let us assume we are given an abstract quaternionic contact manifold
$M$, i.e. a distribution equipped with the right quaternionic contact
structure (in any signature). Then, only the other second cohomology
component can give rise to curvature with  cochains of the type
$\Lambda^2\mathfrak g_{-1}^*\otimes \mathfrak g'_0$, except for the lowest 
dimension $\operatorname{dim} M = 7$, where another torsion with cochains of
the type $\mathfrak g_{-2}^*\otimes\mathfrak g_{-1}^*\otimes \mathfrak g_{-2}$ 
may appear. 
In particular, exactly as in the CR-geometry case, 
if there is no curvature of the form $\mathfrak g_{-2}^*\otimes\mathfrak
g_{-1}^*\otimes\mathfrak g_{-2}$, then the general first homogeneity
prolongation procedure will again produce the triples of contact forms and
their corresponding Reeb vector fields corresponding to the reductions of
the structure group to $\mathfrak g_0'$ (as exploited in the
positive definite case in \cite{Biq1}). 

The latter observation will be the starting point in our construction.
Morever, the general knowledge of the total curvature structure deduced in
\cite[Corollary 3.2]{Cap-twistor} reveals that there is no curvature with
values in $\mathfrak h$. Thus, in full analogy with the construction by
Chern and Moser, we may move straight to the appropriate frame bundle
with structure Lie algebra $\mathfrak h=\mathbb R\oplus
\mathfrak{sp}(1)=\mathbb H$, and work out all our exterior
calculus there. 

The main results are spread through the text as follows: Theorem
\ref{Theorem_1} provides the construction of the Cartan connection as the
canonical coframe at the right principal fiber bundle. Then, right in the
beginning of the next section, Proposition \ref{curvature} displays the
complete structure equations of the coframe, 
thus providing all the curvature components of the
canonical coframe. Next the differential consequences of the Bianchi
identities (cf. Proposition \ref{bianchi}) 
are explicitly listed in Proposition \ref{secondary}.

The explanation how the coframe and its curvature are
related to the Lie algebra structure is presented in section \ref{sec5} on the
associated Cartan geometry in terms of the principal fiber bundles and
the algebraic normalization conditions. In particular, in \ref{normal_Cartan}
we verify that the
curvature of the constructed canonical coframe is co-closed and thus
coincides with the normal Cartan connection for the quaternionic contact
manifolds. In the very end, the appendix collects brief information on
the abstract theory of parabolic geometries and provides links of the
general concepts to the the individual objects and formulae in the paper.

{\bf Acknowledgments.} I.M. is partially supported by
Contract DFNI I02/4/12.12.2014 and Contract 80-10-33/2017 with the
Sofia University "St.Kl.Ohridski".  I.M. is also supported by a SoMoPro II Fellowship which is co-funded  by
the European Commission\footnote{This article reflects only the author's views and the EU is not liable for any use that may be made of the information contained therein.} from \lq\lq{}People\rq\rq{} specific program (Marie Curie Actions) within the EU
Seventh Framework Program on the basis of the grant agreement REA No. 291782. It is further co-financed by the South-Moravian Region. J.S. is supported by the grant P201/12/G028 of the Grant Agency of the Czech Republic.

\section{Preliminaries}

\subsection{Conventions concerning the use of complex tensors and indices}\label{prelim}

Throughout this paper, we use without comment the convention of summation over repeating indices; the small Greek indices $\alpha,\beta,\gamma,\dots$ will have the range $1,\dots,2n$, whereas the indices $s,t,k,l,m$ will be running from $1$ to $3$.  

Consider the Euclidean vector space $\R^{4n}$ with its standard inner product $\lc,\rc$  (with or without signature) and a quaternionic structure induced by the identification $\R^{4n}\cong\mathbb H^n$ with the quaternion coordinate space $\mathbb H^n$.  The latter means that we endow $\R^{4n}$ with a fixed triple $J_1,J_2,J_3$ of complex structures which are Hermitian with respect to $\lc,\rc$ and satisfy
$
J_1\,J_2=-J_2\,J_1=J_3.  
$ 
The complex vector space  $\mathbb C^{4n}$, being the complexification of $\R^{4n}$, splits as a direct sum of $+i$ and $-i$ eigenspaces, $\mathbb C^{4n}=\mathcal W\oplus\br {\mathcal W}$, with respect to the complex structure $J_1.$ The complex 2-form $\pi$, 
\begin{equation*}
\pi(u,v)\overset{def}{=}\lc J_2u,v\rc+i\lc J_3u,v\rc,\qquad u,v\in \mathbb C^{4n},
\end{equation*}
has type $(2,0)$ with respect to $J_1$, i.e., it satisfies $\pi(J_1u,v)=\pi(u,J_1v)=i\pi(u,v)$. Let us fix an $\lc,\rc$-orthonormal basis (once and for all)  
\begin{equation}\label{fixed-basis}
\{\mathfrak e_\alpha\in \mathcal W, \mathfrak e_{\bar\alpha}\in \br{\mathcal W}\},\qquad \mathfrak e_{\bar\alpha}=\br{ \mathfrak e_\alpha},
\end{equation}
with dual basis $\{\mathfrak e^\alpha, \mathfrak e^{\bar\alpha}\}$ 
so that $\pi=\mathfrak e^1\wedge \mathfrak e^{n+1}+\mathfrak e^2\wedge \mathfrak e^{n+2}+\dots + \mathfrak e^{n}\wedge \mathfrak e^{2n}$. Then, we have
\begin{equation}
\lc,\rc = g_{\alpha\bar\beta}\, \mathfrak e^\alpha\otimes \mathfrak e^{\bar\beta} + g_{\bar\alpha\beta}\, \mathfrak e^{\bar\alpha}\otimes \mathfrak e^{\beta},\qquad \pi=\pi_{\alpha\beta}\, \mathfrak e^{\alpha}\wedge \mathfrak e^{\beta}, 
\end{equation}
 where,  for the positive definite case, we take
\begin{equation}\label{constants}
 g_{\alpha\bar\beta}=g_{\bar\beta\alpha}=\begin{cases}1, & \mbox{if } \alpha=\beta\\0, & \mbox{if } \alpha\ne\beta 
 \end{cases},\qquad \pi_{\alpha\beta}=-\pi_{\beta\alpha}=\begin{cases} 1,& \mbox{if } \alpha+n=\beta\\-1,& \mbox{if } \alpha=\beta+n\\0, & \mbox{otherwise},
 \end{cases}
\end{equation}
and for the case of signature, we take $-1$ instead of $1$ for the respective coefficients. In fact, the precise values of the constants $g_{\alpha\bar\beta}$ and $\pi_{\alpha\beta}$ are completely irrelevant for the forthcoming developments; the only thing that matters is that $g_{\alpha\bar\beta}$ is non-degenerate and hermitian (i.e. $g_{\alpha\bar\beta}=g_{\bar\beta\alpha}$), $\pi$ is non-degenerate and skew-symmetric (i.e. $\pi_{\alpha\beta}=-\pi_{\beta\alpha}$), and that 
$$g^{\sigma\bar\tau}\pi_{\alpha\sigma}{\pi_{\bar\tau\bar\beta}}=-g_{\alpha\bar\beta},\qquad \pi_{\bar\tau\bar\beta}\overset{def}{=}\overline{\pi_{\tau\beta}},$$ 
where $g^{\alpha\bar\beta}=g^{\bar\beta\alpha}$ denotes the inverse of $g_{\alpha\bar\beta}$, i.e. 
$g^{\alpha\bar\sigma} g_{\bar\sigma\beta}=\delta^{\alpha}_\beta$ ($\delta^{\alpha}_{\beta}$ is the Kronecker delta).

Any array
of complex numbers indexed by lower and upper Greek letters (with and without bars) corresponds to a tensor, e.g., $\{A_{\alpha\dt\dt}^{\spc\beta\bar\gamma}\}$ corresponds to the tensor 
\begin{equation*}
A_{\alpha\dt\dt}^{\spc\beta\bar\gamma}\,\mathfrak e^\alpha\otimes\mathfrak e_{\beta}\otimes\mathfrak e_{\bar\gamma}.
\end{equation*}     
Clearly, the vertical as well as the horizontal position of an index carries information about the tensor. For two-tensors, we take $B^\alpha_\beta$ to mean $B^{\spc\alpha}_{\beta\dt},$ i.e., the lower index is assumed to be first.  We use $g_{\alpha\bar\beta}$ and $g^{\alpha\bar\beta}$ to lower and raise indices in the usual way, e.g., 
$$A^{\spc\beta}_{\alpha\dt\gamma}=g_{\bar\sigma\gamma}\, A_{\alpha\dt\dt}^{\spc\beta\bar\sigma},\qquad A^{\bar\alpha\beta\bar\gamma}=g^{\bar\alpha\sigma}A_{\sigma\dt\dt}^{\spc\beta\bar\gamma}.$$

We  use  also the following convention:  Whenever an array $\{A_{\alpha\dt\dt}^{\spc\beta\bar\gamma}\}$ appears,  the array $\{A_{\bar\alpha\dt\dt}^{\spc\bar\beta\gamma}\}$ will be assumed to be defined, by default, by the complex conjugation
 $$A_{\bar\alpha\dt\dt}^{\spc\bar\beta\gamma}= \overline{A_{\alpha\dt\dt}^{\spc\beta\bar\gamma}}.$$ This means that we interpret $\{A_{\alpha\dt\dt}^{\spc\beta\bar\gamma}\}$ as a representation of a real tensor, defined on $\R^{4n}$, 
 with respect to the fixed complex basis \eqref{fixed-basis}; the corresponding real tensor in this case is
\begin{equation*}
A_{\alpha\dt\dt}^{\spc\beta\bar\gamma}\,\mathfrak e^\alpha\otimes\mathfrak e_{\beta}\otimes\mathfrak e_{\bar\gamma} + A_{\bar\alpha\dt\dt}^{\spc\bar\beta\gamma}\,\mathfrak e^{\bar\alpha}\otimes\mathfrak e_{\bar\beta}\otimes\mathfrak e_{\gamma}.
\end{equation*}     
 
 Notice that $\pi^{\alpha}_{\bar\sigma}\,\pi^{\bar\sigma}_{\beta}=-\ \delta^{\alpha}_{\beta}$. We  introduce a complex antilinear endomorphism $\mathfrak j$ of the tensor algebra of $\R^{4n}$, which takes a tensor with components $T_{\alpha_1\dots\alpha_k\bar\beta_1\dots\bar\beta_l\dots}$ to a tensor of the same type, with components $(\mathfrak jT)_{\alpha_1\dots\alpha_k\bar\beta_1\dots\bar\beta_l\dots}$, by the formula
 \begin{equation}
 (\mathfrak jT)_{\alpha_1\dots\alpha_k\bar\beta_1\dots\bar\beta_l\dots}=\sum_{\bar\sigma_1\dots\bar\sigma_k\tau_1\dots\tau_l\dots}\pi^{\bar\sigma_1}_{\alpha_1}\dots\pi^{\bar\sigma_k}_{\alpha_k}\,\pi^{\tau_1}_{\bar\beta_1}\dots\pi^{\tau_l}_{\bar\beta_l}\dots T_{\bar\sigma_1\dots\bar\sigma_k\tau_1\dots\tau_l\dots}.
 \end{equation}

By definition, the group $Sp(n)$ consists of all endomorphisms of $\R^{4n}$ that preserve the inner product $\lc,\rc$ and commute with the complex structures $J_1,J_2$ and $J_3$.   
With the above notation, we can alternatively describe $Sp(n)$ as the set of all two-tensors $\{U^\alpha_\beta\}$ satisfying
\begin{equation}\label{def-Spn}
 g_{\sigma\bar\tau}U^{\sigma}_{\alpha}U^{\bar\tau}_{\bar\beta}=g_{\alpha\bar\beta},\qquad \pi_{\sigma\tau}U^{\sigma}_{\alpha}U^{\tau}_{\beta}=\pi_{\alpha\beta}.
 \end{equation}
For its Lie algebra, $sp(n)$, we have the following description:
\begin{lemma}\label{sp(n)} For a tensor $\{X_{\alpha\bar\beta}\}$, the following conditions are equivalent:
\par (1) $\{X_{\alpha\bar\beta}\} \in sp(n)$.
\par (2) $X_{\alpha\bar\beta}=-X_{\bar\beta\alpha},$ $(\mathfrak jX)_{\alpha\bar\beta}=X_{\alpha\bar\beta}.$
 \par (3) $X_{\beta}^{\alpha}=\pi^{\alpha\sigma}Y_{\sigma\beta}$ for some tensor $\{Y_{\alpha\beta}\}$ satisfying $Y_{\alpha\beta}=Y_{\beta\alpha}$ and $(\mathfrak jY)_{\alpha\beta}=Y_{\alpha\beta}$.
\end{lemma}
 \begin{proof}
The equivalence between (1) and (2) follows by differentiating \eqref{def-Spn} at the identity.   
To obtain (3), we define the tensor $\{Y_{\sigma\beta}\}$ by $Y_{\sigma\beta}=-\pi_{\sigma\tau}X^{\tau}_{\beta}=-\pi_{\sigma}^{\bar\tau}X_{\beta\bar\tau}.$
 
 \end{proof}

\subsection{Quaternionic contact manifolds}

Let $M$ be a $(4n+3)$-dimensional manifold and $H$ be a smooth distribution on $M$ of codimension three. 
The pair $(M,H)$ is said to be a quaternionic contact (abbr. qc) structure if around each point of $M$ there exist 1-forms $\hat\eta_1,\hat\eta_2,\hat\eta_3$ with common kernel $H,$ a non-degenerate inner product $\hat g$ on $H$ (with or without signature), and endomorphisms $\hat I_1,\hat I_2,\hat I_3$ of $H,$ satisfying
\begin{gather}\label{def-hat-eta}
(\hat I_1)^2=(\hat I_2)^2=(\hat I_3)^2=-\text{id}_H,\qquad \hat I_1\,\hat I_2=-\hat I_2\, \hat I_1=\hat I_3,\\\nonumber
d\hat\eta_s(X,Y)=2\hat g(\hat I_sX,Y) \qquad \ \text{ for all }\ X,Y\in H.
\end{gather}

As we mentioned in the introduction, 
the qc structures may be considered as a quaternion analog of the CR manifolds of hypersurface type known from the complex analysis; one should, however, be  aware of a few differences.  First, for each qc manifold $(M,H)$, the linear span of the pointwise quaternionic structure $\hat I_1,\hat I_2,\hat I_3$ is a 3-dimensional subbundle of $End(H)$ which is uniquely determined by the distribution $H$ and does not need to be prescribed in advance. Secondly, there is an essential part in the definition of a CR manifold, called integrability condition, that requires for the holomorphic CR distribution to satisfy the Frobenius condition. The qc counterpart to the latter is the existence of Reeb vector fields, namely a triple $\hat\xi_1,\hat\xi_2,\hat\xi_3$ of vector fields on $M$ satisfying for all $X\in H$,
\begin{equation}\label{Reeb}
\hat\eta_s(\hat\xi_t)=\delta^s_t,\qquad d\hat\eta_s(\hat\xi_t,X)=-d\hat\eta_t(\hat\xi_s,X)
\end{equation}
($\delta^s_t$ being the Kronecker delta). As  shown in \cite{Biq1}, the Reeb vector fields always exist if $\dim(M)>7$. In the seven dimensional case this is an additional integrability condition on the qc structure (cf. \cite{D}) which we will assume to be satisfied.  

Let us define the 2-forms $\hat\omega_1,\hat\omega_2,\hat\omega_3$  by 
\begin{equation}\label{def-omega}
\hat\xi_s\lrcorner\hat\omega_t=0,\qquad \hat\omega_s(X,Y)=2\hat g(\hat I_sX,Y),\qquad X,Y\in H.
\end{equation}
 Then, as it can be easily verified, the existence of Reeb vector fields allows us to express the exterior derivatives $d\hat\eta_s$ in the form  
\begin{equation}\label{integ}
d\hat\eta_s=-\hat\alpha_{ts}\wedge\hat\eta_t +2\hat\omega_s,\qquad \hat\alpha_{st}=-\hat\alpha_{ts}.
\end{equation}
Explicitly, the one-forms $\hat\alpha_{st}$  are given by
\begin{equation*}
\begin{cases}
\hat \alpha_{12}=\hat\xi_1\lrcorner d\hat\eta_2\ +\ \frac{1}{2}\Big(-d\hat\eta_1(\hat\xi_2,\hat\xi_3)+d\hat\eta_2(\hat\xi_3,\hat\xi_1)+d\hat\eta_3(\hat\xi_1,\hat\xi_2)\Big)\hat\eta_3\ +\ d\hat\eta_1(\hat\xi_1,\hat\xi_2)\hat\eta_1\\
\hat \alpha_{23}=\hat\xi_2\lrcorner d\hat\eta_3\ +\ \frac{1}{2}\Big(d\hat\eta_1(\hat\xi_2,\hat\xi_3)-d\hat\eta_2(\hat\xi_3,\hat\xi_1)+d\hat\eta_3(\hat\xi_1,\hat\xi_2)\Big)\hat\eta_1\ +\ d\hat\eta_2(\hat\xi_2,\hat\xi_3)\hat\eta_2\\
\hat \alpha_{31}=\hat\xi_3\lrcorner d\hat\eta_1\ +\ \frac{1}{2}\Big(d\hat\eta_1(\hat\xi_2,\hat\xi_3)+d\hat\eta_2(\hat\xi_3,\hat\xi_1)-d\hat\eta_3(\hat\xi_1,\hat\xi_2)\Big)\hat\eta_2\ +\ d\hat\eta_3(\hat\xi_3,\hat\xi_1)\hat\eta_3
\end{cases}.
\end{equation*}

\section{Solution to the quaternionic contact equivalence problem}\label{sec_eq_prob}

 It is well known that to each qc manifold $(M,H)$ one can associate a unique, up to a diffeomorphism, regular, normal Cartan geometry, i.e., a certain principle bundle $\mathcal P_1\rightarrow M$ endowed with a Cartan connection that satisfies some natural normalization conditions (see the Appendix and the references therein for more details on the topic). Our goal here is to provide an explicit construction for both the bundle and the connection in terms of geometric data generated entirely by the qc structure of $M$. We are using essentially the original Cartan's method of equivalence that had been  applied later with a great success by Chern and Moser in \cite{ChM} for solving the respective equivalence problem in the CR case. The method is based entirely on classical exterior calculus and does not require any preliminary knowledge concerning the theory of parabolic geometries or the related Lie algebra cohomology. The main result here is Theorem~\ref{Theorem_1}.

Let  $\hat\eta_1,\hat\eta_2,\hat\eta_3, \hat I_1, \hat I_2, \hat I_3, \hat g$ be as in \eqref{def-hat-eta}. If $\tilde\eta_1,\tilde\eta_2,\tilde\eta_3$ are any (other) 1-forms satisfying  \eqref{def-hat-eta} for some  symmetric and positive definite $\tilde g\in H^*\otimes H^*$ and endomorphisms $\tilde I_s \in End(H)$ in place of $\hat g$ and $\hat I_s$ respectively, then it is known (see for example the appendix of \cite{IMV5}) that there exists a positive real-valued function $\mu$ and an $SO(3)$-valued function $\Psi=(a_{st})_{3\times 3}$ so that
\begin{equation*}
\tilde \eta_s=\mu\, a_{ts}\, \hat\eta_t,\qquad \tilde g=\mu \, \hat g,\qquad \tilde I_s=a_{ts}\, \hat I_t.
\end{equation*}

In intrinsic terms, this means that we have a principle bundle $P_o$ over $M$ with structure group $CSO(3)=\R^{+}\times SO(3)$ whose local sections are exactly the triples of
1-forms $(\hat\eta_1,\hat\eta_2,\hat\eta_3)$ satisfying \eqref{def-hat-eta}. The functions $\mu\in \R^{+}$ and $(a_{st})_{3\times 3}\in SO(3)$ may be considered as local fiber coordinates on $P_o$ with respect to a fixed local section $(\hat\eta_1,\hat\eta_2,\hat\eta_3)$. On $P_o$, we have globally defined intrinsic one-forms $\eta_1,\eta_2,\eta_3$ which,  in terms of the local fiber coordinates, have the expression 
\begin{equation}\label{def-eta}
\eta_s=\mu\, a_{ts}\,  \pi_o^*(\hat\eta_t),
\end{equation}
with $\pi_o: P_o\rightarrow M$ being the principle bundle projection. 
We will call a differential forms on $P_o$ semibasic if its contraction with any vector field tangent to the fibers of $\pi_o$ vanishes.

\begin{lemma}\label{Lex}
In a neighborhood of each point of $P_o$, we can find  real one-forms $\varphi_0,\varphi_1,\varphi_2,\varphi_3$ and semibasic complex one-forms $\theta^{\alpha}$ so that  
\begin{equation*}
\begin{cases}
d\eta_1=-\varphi_0\wedge\eta_1-\varphi_2\wedge\eta_3+\varphi_3\wedge\eta_2+2i g_{\alpha\bar\beta}\,\theta^{\alpha}\wedge\theta^{\bar\beta}\\
d\eta_2=-\varphi_0\wedge\eta_2-\varphi_3\wedge\eta_1+\varphi_1\wedge\eta_3+\pi_{\alpha\beta}\,\theta^{\alpha}\wedge\theta^{\beta}+\pi_{\bar\alpha\bar\beta}\,\theta^{\bar\alpha}\wedge\theta^{\bar\beta}\\
d\eta_3=-\varphi_0\wedge\eta_3-\varphi_1\wedge\eta_2 + \varphi_2\wedge\eta_1-i\pi_{\alpha\beta}\,\theta^{\alpha}\wedge\theta^{\beta}+i\pi_{\bar\alpha\bar\beta}\,\theta^{\bar\alpha}\wedge\theta^{\bar\beta},
\end{cases}
\end{equation*}
where $g_{\alpha\bar\beta}=g_{\bar\beta\alpha}$ and $\pi_{\alpha\beta}=-\pi_{\beta\alpha}$ are the same (fixed) constants as in Section~\ref{prelim}. 
\end{lemma}

\begin{proof}

Let us consider the distribution $H\subset TM$ as a vector bundle over $M$ and take $\mathcal H\rightarrow P_o$ to be the corresponding  pull-back bundle via $\pi_o$.
To each point $p\in P_o$, we can associate a natural triple of endomorphisms $I_1,I_2,I_3$ and a symmetric 2-tensor $g$ of the fibers  $\mathcal H_p$ of $\mathcal H$ given by 
\begin{equation}
I_s= a_{ts}\, \hat I_t,\qquad g=\mu\, \hat g
\end{equation}
(with the obvious identification $\mathcal H_p\cong H_{\pi_o(p)}$). Then, 
\begin{equation}
(I_1)^2=(I_2)^2=(I_3)^2=-\text{id}_{\mathcal H},\qquad I_1\,I_2=-I_2\, I_1=I_3.
\end{equation}
The complexification of $\mathcal H_p$ (which we will denote again by $\mathcal H_p$) splits as $\mathcal H_p=W_p\oplus\overline W_p$ with $W_p$ and $\overline W_p$ being the eigenspaces of $+i$ and $-i$ with respect to the endomorphism $I_1$.  We denote by $\pi$ the skew-symmetric 2-tensor on $\mathcal H$ given by 
$$\pi(u,v)=g(I_2u,v)+ig(I_3u,v),\qquad u,v\in \mathcal H,$$ 
which is easily seen to be of type $(2,0)$ with respect to $I_1$, i.e., we have $\pi(I_1u,v)=\pi(u,I_1v)=i\pi(u,v).$

Let us pick a local coframing 
\begin{equation}
\{\theta^{\alpha}\in W^*,\theta^{\bar\alpha}\in \overline{W^*}\},\qquad \theta^{\bar\alpha}=\overline{\theta^{\alpha}}
\end{equation}
for the (complexified) vector bundle $\mathcal H$ so that
\begin{equation}
g=g_{\alpha\bar\beta}\theta^{\alpha}\otimes\theta^{\bar\beta}+g_{\bar\beta\alpha}\theta^{\bar\beta}\otimes\theta^{\alpha} \quad \text{and}\quad \pi=\pi_{\alpha\beta}\theta^{\alpha}\wedge\theta^{\beta},
\end{equation}
where $g_{\alpha\bar\beta}=g_{\bar\beta\alpha}$ and $\pi_{\alpha\beta}=-\pi_{\beta\alpha}$ are given by \eqref{constants} (this is always possible by a standard linear algebra argumentation).

The two-forms $\hat\omega_1,\hat\omega_2,\hat\omega_3$ given by \eqref{def-omega} may be regarded as two-tensors on the fibers of $\mathcal H$. We will need the following  identities: 
\begin{equation}\label{hat-identities}
\begin{cases}
2\mu a_{s1}\,\hat\omega_s=2i g_{\alpha\bar\beta}\theta^{\alpha}\wedge\theta^{\bar\beta}\\
2\mu a_{s2}\,\hat\omega_s=\pi_{\alpha\beta}\,\theta^{\alpha}\wedge\theta^{\beta}+\pi_{\bar\alpha\bar\beta}\,\theta^{\bar\alpha}\wedge\theta^{\bar\beta}\\
2\mu a_{s3}\,\hat\omega_s=-i\pi_{\alpha\beta}\,\theta^{\alpha}\wedge\theta^{\beta}+i\pi_{\bar\alpha\bar\beta}\,\theta^{\bar\alpha}\wedge\theta^{\bar\beta}.
\end{cases}
\end{equation}
Indeed, if $u,v\in \mathcal H_p$, then
\begin{equation*}
2\mu a_{s1}\,\hat\omega_s(u,v)=2g(a_{s1}\,\hat I_s u, v)=2g(I_1u,v)=2ig_{\alpha\bar\beta}\theta^{\alpha}(u)\theta^{\bar\beta}(v)-2ig_{\bar\alpha\beta}\theta^{\bar\alpha}(u)\theta^{\beta}(v)
\end{equation*}
gives the first identity in the list. The second identity follows by the computation
\begin{equation*}
2\mu a_{s2}\,\hat\omega_s(u,v)=2g(a_{s2}\,\hat I_s u, v)=2g(I_2u,v)=\pi(u,v)+\overline{\pi(u,v)}.
\end{equation*}
The third identity we obtain similarly by 
\begin{equation*}
2\mu a_{s3}\,\hat\omega_s(u,v)=2g(a_{s3}\,\hat I_s u, v)=2g(I_3u,v)=-i\pi(u,v)-\overline{i\pi(u,v)}.
\end{equation*}

 Let $\hat\xi_1,\hat\xi_2,\hat\xi_3$ be the Reeb vector fields corresponding to $\hat\eta_1,\hat\eta_2,\hat\eta_3$ (cf. \eqref{Reeb}). For each $p\in P_o$, we have the map 
\begin{equation}\label{P_o-H}
T_pP_o\overset{\pi_o}{\longrightarrow} T_{\pi_o(p)}M\longrightarrow H_{\pi_o(p)}\overset{\cong}{\longrightarrow} \mathcal H_p,
\end{equation}
where the second arrow denotes the projection on the first factor in $T_{\pi_o(p)}M=H_{\pi_o(p)}\oplus span\{\hat\xi_1,\hat\xi_2,\hat\xi_3\}$. By \eqref{P_o-H}, we can consider $\theta^{\alpha}$, $\theta^{\bar\alpha}$ and $\hat\omega_s$ as 1-forms on $P_o$. Then, clearly, the identities \eqref{hat-identities} remain valid. Notice also that,  since $(a_{st})_{3\times 3}\in SO(3)$, we have 
 $a_{sl}\, a_{tl}=\delta_{st}$ and thus the expression $a_{sl}\,da_{tl}$ is skew-symmetric in $s,t$. By differentiating \eqref{def-eta}, we get
\begin{multline}\label{comp-Lex}
d\eta_s{\ =\ }d\mu\wedge a_{ts} \hat\eta_t + \mu\, da_{ts}\wedge \hat\eta_t + \mu a_{ts}\, \, (\underbrace{-\hat\alpha_{lt}\hat\eta_l+2\hat\omega_t}_\text{by \eqref{integ}})\\ 
= \mu a_{tl}\,\Big(\mu^{-1}\delta_{ls}d\mu + a_{kl}da_{ks}-a_{kl}a_{ms}\hat\alpha_{km}\Big)\wedge \hat\eta_t + 2\mu a_{ts}\, \hat\omega_t\\
=\mu^{-1}d\mu\wedge \eta_s+a_{kl}( da_{ks}-a_{ms}\hat\alpha_{km})\wedge \eta_l + 2\mu a_{ts}\, \hat\omega_t.
\end{multline}
Since $a_{kl}( da_{ks}-a_{ms}\hat\alpha_{km})$ is skew symmetric in $l,s$, it can be represented by a triple of one-forms $\varphi_1,\varphi_2,\varphi_3$. Explicitly, we define
\begin{equation*}
\begin{cases}
\varphi_0=-\mu^{-1}d\mu\\
\varphi_1=-a_{k2}( da_{k3}-a_{m3}\hat\alpha_{km})\\
\varphi_2=-a_{k3}( da_{k1}-a_{m1}\hat\alpha_{km})\\
\varphi_3=-a_{k1}( da_{k2}-a_{m2}\hat\alpha_{km}).
\end{cases}
\end{equation*} 
Then, the Lemma follows by \eqref{comp-Lex} and \eqref{hat-identities}. 

\end{proof}

\begin{lemma}\label{P_o-transformations} If $\varphi_0,\varphi_1,\varphi_2,\varphi_3,\theta^{\alpha}$ are any one-forms that satisfy the assertion of Lemma~\ref{Lex}, then $\eta_1,\eta_2,\eta_3,\theta^{\alpha},\theta^{\bar\alpha},\varphi_0,\varphi_1,\varphi_2,\varphi_3$ are pointwise linearly independent.

Any other one-forms $\tilde\varphi_0,\tilde\varphi_1,\tilde\varphi_2,\tilde\varphi_3, \tilde\theta^{\alpha}$ will satisfy the assertion of Lemma~\ref{Lex} if and only if they are given by the formulas
\begin{equation}\label{str_eq_P_o}
\begin{cases}
\tilde\theta^{\alpha}=U^{\alpha}_{\beta}\theta^\beta + ir^\alpha\eta_1+\pi^{\alpha}_{\bar\sigma}r^{\bar\sigma}(\eta_2+i\eta_3)\\
\tilde\varphi_0=\varphi_0+2U_{\beta\bar\sigma}r^{\bar\sigma}\theta^{\beta}+2U_{\bar\beta\sigma}r^{\sigma}\theta^{\bar\beta}+\lambda_1\eta_1+\lambda_2\eta_2+\lambda_3\eta_3\\
\tilde\varphi_1=\varphi_1-2iU_{\beta\bar\sigma}r^{\bar\sigma}\theta^{\beta}+2iU_{\bar\beta\sigma}r^{\sigma}\theta^{\bar\beta}+2r_{\sigma}r^{\sigma}\eta_1-\lambda_3\eta_2+\lambda_2\eta_3,\\
\tilde\varphi_2=\varphi_2-2\pi_{\sigma\tau}U^{\sigma}_{\beta}r^{\tau}\theta^{\beta}-2\pi_{\bar\sigma\bar\tau}U^{\bar\sigma}_{\bar\beta}r^{\bar\tau}\theta^{\bar\beta}+\lambda_3\eta_1+2r_{\sigma}r^{\sigma}\eta_2-\lambda_1\eta_3,\\
\tilde\varphi_3=\varphi_3+2i\pi_{\sigma\tau}U^{\sigma}_{\beta}r^{\tau}\theta^{\beta}-2i\pi_{\bar\sigma\bar\tau}U^{\bar\sigma}_{\bar\beta}r^{\bar\tau}\theta^{\bar\beta}-\lambda_2\eta_1+\lambda_1\eta_2+2r_{\sigma}r^{\sigma}\eta_3,
\end{cases}
\end{equation}
where $U^{\alpha}_{\beta}, r^\alpha, \lambda_s$ are some appropriate functions; $\lambda_1,\lambda_2,\lambda_3$ are real, and $\{U^{\alpha}_{\beta}\}$ satisfy \eqref{def-Spn}, i.e., $\{U^{\alpha}_{\beta}\}\in Sp(n)\subset End(\R^{4n})$.
\end{lemma} 
\begin{proof} Let us begin by taking $\varphi_0,\varphi_1,\varphi_2,\varphi_3,\theta^{\alpha}$ to be the one-forms constructed in the prof of Lemma~\ref{Lex}. For these	one-forms, it is obvious that 
$\{\eta_1,\eta_2,\eta_3,\theta^{\alpha},\theta^{\bar\alpha},\phi_0,\phi_1,\phi_2,\phi_3\}$ is a coframing for $TP_o.$ Let $\tilde\varphi_0,\tilde\varphi_1,\tilde\varphi_2,\tilde\varphi_3,\tilde\theta^{\alpha}$ be any other one-forms satisfying the assertion of Lemma~\ref{Lex}. Then, by subtracting the corresponding equations, we obtain 
\begin{equation}\label{differences-eq}
\begin{split}
0=-&(\tilde\varphi_0-\varphi_0)\wedge\eta_1-(\tilde\varphi_2-\varphi_2)\wedge\eta_3+(\tilde\varphi_3-\varphi_3)\wedge\eta_2+2i g_{\alpha\bar\beta}\,(\tilde\theta^{\alpha}\wedge\tilde\theta^{\bar\beta}-\theta^{\alpha}\wedge\theta^{\bar\beta})\\
0 =-&(\tilde\varphi_0-\varphi_0)\wedge\eta_2-(\tilde\varphi_3-\varphi_3)\wedge\eta_1+(\tilde\varphi_1-\varphi_1)\wedge\eta_3+\pi_{\alpha\beta}\,(\tilde\theta^{\alpha}\wedge\tilde\theta^{\beta}-\theta^{\alpha}\wedge\theta^{\beta})\\ &+\pi_{\bar\alpha\bar\beta}\,(\tilde\theta^{\bar\alpha}\wedge\tilde\theta^{\bar\beta}-\theta^{\bar\alpha}\wedge\theta^{\bar\beta})\\
0=-&(\tilde\varphi_0-\varphi_0)\wedge\eta_3-(\tilde\varphi_1-\varphi_1)\wedge\eta_2 + (\tilde\varphi_2-\varphi_2)\wedge\eta_1-i\pi_{\alpha\beta}\,(\tilde\theta^{\alpha}\wedge\tilde\theta^{\beta}-\theta^{\alpha}\wedge\theta^{\beta})\\&+i\pi_{\bar\alpha\bar\beta}\,(\tilde\theta^{\bar\alpha}\wedge\tilde\theta^{\bar\beta}-\theta^{\bar\alpha}\wedge\theta^{\bar\beta}).
\end{split} 
\end{equation}
Since, by assumption, $\tilde\theta^\alpha$ are semibasic, i.e., their contractions with vector fields tangent to to the fibers of $\pi_o$ vanish, we have that
\begin{equation}\label{differ-theta}
\tilde\theta^\alpha=U^\alpha_\beta\theta^\beta+U^\alpha_{\bar\beta}\theta^{\bar\beta}+A^{\alpha}_s\eta_s
\end{equation}
for some appropriate coefficients $U^\alpha_\beta, U^\alpha_{\bar\beta}, A^{\alpha}_s$.

Wedging the first identity of \eqref{differences-eq} with $\eta_2\wedge\eta_3\wedge\theta^1\wedge\dots\wedge\theta^{2n}\wedge\overline{\theta^1}\wedge\dots\wedge\overline{\theta^{2n}}$ yields
\begin{equation*}
(\tilde\varphi_0-\varphi_0)\wedge\eta_1\wedge\eta_2\wedge\eta_3\wedge\theta^1\wedge\dots\wedge\theta^{2n}\wedge\overline{\theta^1}\wedge\dots\wedge\overline{\theta^{2n}}=0
\end{equation*}
and hence $(\tilde\varphi_0-\varphi_0)\in\text{span}\{\eta_1,\eta_2,\eta_3,\theta^\alpha,\theta^{\bar\alpha}\}$. Proceeding similarly for $\tilde\varphi_1-\varphi_1,\tilde\varphi_2-\varphi_2$ and $\tilde\varphi_3-\varphi_3$, we arrive at the equations
 \begin{equation}\label{differ-phi}
 \begin{cases}
 \tilde\varphi_0=\varphi_0+b_{\alpha}\theta^\alpha+b_{\bar\alpha}\theta^{\bar\alpha}+b_s\eta_s\\
 \tilde\varphi_1=\varphi_1+c_{\alpha}\theta^\alpha+c_{\bar\alpha}\theta^{\bar\alpha}+c_s\eta_s\\
 \tilde\varphi_2=\varphi_2+d_{\alpha}\theta^\alpha+d_{\bar\alpha}\theta^{\bar\alpha}+d_s\eta_s\\
 \tilde\varphi_3=\varphi_3+e_{\alpha}\theta^\alpha+e_{\bar\alpha}\theta^{\bar\alpha}+e_s\eta_s
 \end{cases}
 \end{equation}
 with $b^\alpha, c^\alpha,d^\alpha,e^\alpha, b_s,c_s,d_s,e_s$ being some appropriate functions.
 Substituting \eqref{differ-theta} and \eqref{differ-phi} back into \eqref{differences-eq} gives
\begin{equation}\label{differ-1}
\begin{split}
0\ =\ {}&\Big(b_2+e_1 +2ig_{\alpha\bar\beta}(A_1^{\alpha}A_2^{\bar\beta}-A_2^{\alpha}A_1^{\bar\beta})\Big)\,\eta_1\wedge\eta_2-\Big(d_2+e_3-2ig_{\alpha\bar\beta}(A_2^{\alpha}A_3^{\bar\beta}-A_3^{\alpha}A_2^{\bar\beta})\Big)\,\eta_2\wedge\eta_3\\&
-\Big(b_3-d_1+2ig_{\alpha\bar\beta}(A_3^{\alpha}A_1^{\bar\beta}-A_1^{\alpha}A_3^{\bar\beta})\Big)\,\eta_3\wedge\eta_1\\&+\Big(-b_\alpha-2iU_{\alpha\sigma}A_1^{\sigma}+2iU_{\alpha\bar\sigma} A_1^{\bar\sigma}\Big)\,\theta^\alpha\wedge\eta_1
+\Big(-b_{\bar\alpha}+2iU_{\bar\alpha\bar\sigma}A_1^{\bar\sigma}-2iU_{\bar\alpha\sigma}A_1^{\sigma}\Big)\,\theta^{\bar\alpha}\wedge\eta_1\\&
+\Big(e_\alpha-2iU_{\alpha\sigma}A_2^{\sigma}+2iU_{\alpha\bar\sigma} A_2^{\bar\sigma}\Big)\,\theta^\alpha\wedge\eta_2
+\Big(e_{\bar\alpha}+2iU_{\bar\alpha\bar\sigma}A_2^{\bar\sigma}-2iU_{\bar\alpha\sigma}A_2^{\sigma}\Big)\,\theta^{\bar\alpha}\wedge\eta_2\\&
+\Big(-d_\alpha-2iU_{\alpha\sigma}A_3^{\sigma}+2iU_{\alpha\bar\sigma} A_3^{\bar\sigma}\Big)\,\theta^\alpha\wedge\eta_3
+\Big(-d_{\bar\alpha}+2iU_{\bar\alpha\bar\sigma}A_3^{\bar\sigma}-2iU_{\bar\alpha\sigma}A_3^{\sigma}\Big)\,\theta^{\bar\alpha}\wedge\eta_3
\\
&-2i\Big(g_{\alpha\bar\beta}-g_{\sigma\bar\tau}\, U^{\sigma}_{\alpha}\,U^{\bar\tau}_{\bar\beta}+g_{\sigma\bar\tau}\, U^{\sigma}_{\bar\beta}\,U^{\bar\tau}_{\alpha}\Big)\,\theta^{\alpha}\wedge\theta^{\bar\beta}\\&+ig_{\sigma\bar\tau}\Big(U^\sigma_\alpha\,U^{\bar\tau}_{\beta}-U^\sigma_\beta\,U^{\bar\tau}_{\alpha}\Big)\,\theta^\alpha\wedge\theta^{\beta}
+ig_{\sigma\bar\tau}\Big(U^\sigma_{\bar\alpha}\,U^{\bar\tau}_{\bar\beta}-U^\sigma_{\bar\beta}\,U^{\bar\tau}_{\bar\alpha}\Big)\,\theta^{\bar\alpha}\wedge\theta^{\bar\beta};
\end{split}
\end{equation}
\begin{equation}\label{differ-2}
\begin{split}
0\ =\ {}&-\Big(b_1-e_2 -2\pi_{\alpha\beta}A_1^{\alpha}A_2^{\beta}-2\pi_{\bar\alpha\bar\beta}A_1^{\bar\alpha}A_2^{\bar\beta}\Big)\,\eta_1\wedge\eta_2\\&
+\Big(b_3+c_2+2\pi_{\alpha\beta}A_2^{\alpha}A_3^{\beta}+2\pi_{\bar\alpha\bar\beta}A_2^{\bar\alpha}A_3^{\bar\beta}\Big)\,\eta_2\wedge\eta_3\\&
-\Big(c_1+e_3-2\pi_{\alpha\beta}A_3^{\alpha}A_1^{\beta}-2\pi_{\bar\alpha\bar\beta}A_3^{\bar\alpha}A_1^{\bar\beta}\Big)\,\eta_3\wedge\eta_1
\\&+\Big(-e_\alpha+2\pi_{\sigma\tau}\,U^{\sigma}_{\alpha}A_1^{\tau}+2\pi_{\bar\sigma\bar\tau}\,U^{\bar\sigma}_{\alpha}A_1^{\bar\tau}\Big)\,\theta^\alpha\wedge\eta_1
+\Big(-e_{\bar\alpha}+2\pi_{\bar\sigma\bar\tau}\,U^{\bar\sigma}_{\bar\alpha}A_1^{\bar\tau}+2\pi_{\sigma\tau}\,U^{\sigma}_{\bar\alpha}A_1^{\tau}\Big)\,\theta^{\bar\alpha}\wedge\eta_1
\\&
+\Big(-b_\alpha+2\pi_{\sigma\tau}\,U^{\sigma}_{\alpha}A_2^{\tau}+2\pi_{\bar\sigma\bar\tau}\,U^{\bar\sigma}_{\alpha}A_2^{\bar\tau}\Big)\,\theta^\alpha\wedge\eta_2
+\Big(-b_{\bar\alpha}+2\pi_{\bar\sigma\bar\tau}\,U^{\bar\sigma}_{\bar\alpha}A_2^{\bar\tau}+2\pi_{\sigma\tau}\,U^{\sigma}_{\bar\alpha}A_2^{\tau}\Big)\,\theta^{\bar\alpha}\wedge\eta_2
\\&
+\Big(c_\alpha+2\pi_{\sigma\tau}\,U^{\sigma}_{\alpha}A_3^{\tau}+2\pi_{\bar\sigma\bar\tau}\,U^{\bar\sigma}_{\alpha}A_3^{\bar\tau}\Big)\,\theta^\alpha\wedge\eta_3
+\Big(c_{\bar\alpha}+2\pi_{\bar\sigma\bar\tau}\,U^{\bar\sigma}_{\bar\alpha}A_3^{\bar\tau}+2\pi_{\sigma\tau}\,U^{\sigma}_{\bar\alpha}A_3^{\tau}\Big)\,\theta^{\bar\alpha}\wedge\eta_3
\\
&+2\Big(\pi_{\sigma\tau}\, U^{\sigma}_{\alpha}\,U^{\tau}_{\bar\beta}+\pi_{\bar\sigma\bar\tau}\, U^{\bar\sigma}_{\alpha}\,U^{\bar\tau}_{\bar\beta}\Big)\,\theta^{\alpha}\wedge\theta^{\bar\beta}
\\&-\Big(\pi_{\alpha\beta}-\pi_{\sigma\tau}U^\sigma_\alpha\,U^{\tau}_{\beta}-\pi_{\bar\sigma\bar\tau}U^{\bar\sigma}_\alpha\,U^{\bar\tau}_{\beta}\Big)\,\theta^\alpha\wedge\theta^{\beta}
-\Big(\pi_{\bar\alpha\bar\beta}-\pi_{\sigma\tau}U^\sigma_{\bar\alpha}\,U^{\tau}_{\bar\beta}-\pi_{\bar\sigma\bar\tau}U^{\bar\sigma}_{\bar\alpha}\,U^{\bar\tau}_{\bar\beta}\Big)\,\theta^{\bar\alpha}\wedge\theta^{\bar\beta};
\end{split}
\end{equation}
\begin{equation}\label{differ-3}
\begin{split}
0\ =\ {}&-\Big(c_1+d_2 +2i\pi_{\alpha\beta}A_1^{\alpha}A_2^{\beta}-2i\pi_{\bar\alpha\bar\beta}A_1^{\bar\alpha}A_2^{\bar\beta}\Big)\,\eta_1\wedge\eta_2\\&
+\Big(-b_2+c_3-2i\pi_{\alpha\beta}A_2^{\alpha}A_3^{\beta}+2i\pi_{\bar\alpha\bar\beta}A_2^{\bar\alpha}A_3^{\bar\beta}\Big)\,\eta_2\wedge\eta_3\\&
+\Big(b_1+d_3-2i\pi_{\alpha\beta}A_3^{\alpha}A_1^{\beta}+2i\pi_{\bar\alpha\bar\beta}A_3^{\bar\alpha}A_1^{\bar\beta}\Big)\,\eta_3\wedge\eta_1
\\&+\Big(d_\alpha-2i\pi_{\sigma\tau}\,U^{\sigma}_{\alpha}A_1^{\tau}+2i\pi_{\bar\sigma\bar\tau}\,U^{\bar\sigma}_{\alpha}A_1^{\bar\tau}\Big)\,\theta^\alpha\wedge\eta_1
+\Big(d_{\bar\alpha}+2i\pi_{\bar\sigma\bar\tau}\,U^{\bar\sigma}_{\bar\alpha}A_1^{\bar\tau}-2i\pi_{\sigma\tau}\,U^{\sigma}_{\bar\alpha}A_1^{\tau}\Big)\,\theta^{\bar\alpha}\wedge\eta_1
\\&
-\Big(c_\alpha+2i\pi_{\sigma\tau}\,U^{\sigma}_{\alpha}A_2^{\tau}-2i\pi_{\bar\sigma\bar\tau}\,U^{\bar\sigma}_{\alpha}A_2^{\bar\tau}\Big)\,\theta^\alpha\wedge\eta_2
-\Big(c_{\bar\alpha}-2i\pi_{\bar\sigma\bar\tau}\,U^{\bar\sigma}_{\bar\alpha}A_2^{\bar\tau}+2i\pi_{\sigma\tau}\,U^{\sigma}_{\bar\alpha}A_2^{\tau}\Big)\,\theta^{\bar\alpha}\wedge\eta_2
\\&
-\Big(b_\alpha+2i\pi_{\sigma\tau}\,U^{\sigma}_{\alpha}A_3^{\tau}-2i\pi_{\bar\sigma\bar\tau}\,U^{\bar\sigma}_{\alpha}A_3^{\bar\tau}\Big)\,\theta^\alpha\wedge\eta_3
-\Big(b_{\bar\alpha}-2i\pi_{\bar\sigma\bar\tau}\,U^{\bar\sigma}_{\bar\alpha}A_3^{\bar\tau}+2i\pi_{\sigma\tau}\,U^{\sigma}_{\bar\alpha}A_3^{\tau}\Big)\,\theta^{\bar\alpha}\wedge\eta_3
\\
&+2i\Big(-\pi_{\sigma\tau}\, U^{\sigma}_{\alpha}\,U^{\tau}_{\bar\beta}+\pi_{\bar\sigma\bar\tau}\, U^{\bar\sigma}_{\alpha}\,U^{\bar\tau}_{\bar\beta}\Big)\,\theta^{\alpha}\wedge\theta^{\bar\beta}
\\&+i\Big(\pi_{\alpha\beta}-\pi_{\sigma\tau}U^\sigma_\alpha\,U^{\tau}_{\beta}+\pi_{\bar\sigma\bar\tau}U^{\bar\sigma}_\alpha\,U^{\bar\tau}_{\beta}\Big)\,\theta^\alpha\wedge\theta^{\beta}
-i\Big(\pi_{\bar\alpha\bar\beta}+\pi_{\sigma\tau}U^\sigma_{\bar\alpha}\,U^{\tau}_{\bar\beta}-\pi_{\bar\sigma\bar\tau}U^{\bar\sigma}_{\bar\alpha}\,U^{\bar\tau}_{\bar\beta}\Big)\,\theta^{\bar\alpha}\wedge\theta^{\bar\beta}.
\end{split} 
\end{equation}

Since $\eta_1,\eta_2,\eta_3,\theta^{\alpha},\theta^{\bar\alpha}$ are pointwise linearly independent, all the coefficients in \eqref{differ-1},\eqref{differ-2} and \eqref{differ-3} must vanish.  
The vanishing of the coefficients of $\theta^\alpha\wedge\theta^\beta$ in \eqref{differ-2} and \eqref{differ-3} gives  
\begin{equation}
\pi_{\alpha\beta}=\pi_{\sigma\tau}U^\sigma_\alpha\,U^{\tau}_{\beta}
\end{equation}
This implies that  the array $\{U^\alpha_\beta\}$ corresponds to an invertible endomorphism of $\mathbb R^{4n}$. Furthermore, the vanishing of the coefficients of $\theta^\alpha\wedge\theta^{\bar\beta}$ in \eqref{differ-2} and \eqref{differ-3} yields $\pi_{\sigma\tau}\, U^{\sigma}_{\alpha}\,U^{\tau}_{\bar\beta}=0$ and hence $U^{\alpha}_{\bar\beta}=0$, since both $\pi_{\alpha\beta}$ and $U^\alpha_\beta$ are invertible. Furthermore, the vanishing of the coefficients of  $\theta^\alpha\wedge\theta^{\bar\beta}$  in \eqref{differ-1} implies 
\begin{equation}
g_{\alpha\bar\beta}=g_{\sigma\bar\tau}\, U^{\sigma}_{\alpha}\,U^{\bar\tau}_{\bar\beta}
\end{equation}
and thus $\{U^\alpha_\beta\}\in Sp(n)$. 

The vanishing of the coefficients of $\theta^{\alpha}\wedge\eta_s$ in \eqref{differ-1},\eqref{differ-2} and \eqref{differ-3} gives
\begin{equation}
\begin{cases}
b_\alpha=2iU_{\alpha\bar\sigma}A_1^{\bar\sigma}=2\pi_{\sigma\tau}U^\sigma_\alpha A_2^\tau\\
c_\alpha=-2\pi_{\sigma\tau}U^\sigma_\alpha A_3^\tau=-2i\pi_{\sigma\tau}U^\sigma_\alpha A_2^\tau\\
d_\alpha=2iU_{\alpha\bar\sigma}A_3^{\bar\sigma}=2i\pi_{\sigma\tau}U^\sigma_\alpha A_1^\tau\\
e_\alpha=-2iU_{\alpha\bar\sigma}A_2^{\bar\sigma}=2\pi_{\sigma\tau}U^\sigma_\alpha A_1^\tau
\end{cases}
\end{equation}
from which we deduce that $A_3^\alpha=iA_2^\alpha=-\pi^{\alpha}_{\bar\sigma}A_1^{\bar\sigma}$. Thus, by setting $r^\alpha\overset{def}{\ =\ }-iA_1^\alpha$, we obtain
\begin{equation}\label{r-alpha}
\begin{gathered}
A_1^\alpha=ir^\alpha,\qquad A_2^\alpha = \pi^\alpha_{\bar\sigma}r^{\bar\sigma},\qquad A_3^\alpha=i\pi^{\alpha}_{\bar\sigma}r^{\bar\sigma},\\
b_\alpha=2U_{\alpha\bar\sigma}r^{\bar\sigma},\qquad c_\alpha= -2i U_{\alpha\bar\sigma}r^{\bar\sigma},\qquad d_\alpha=-2\pi_{\sigma\tau}U^\sigma_\alpha r^\tau,\qquad 
e_\alpha=2i\pi_{\sigma\tau}U^\sigma_\alpha r^\tau.
\end{gathered}
\end{equation}

We substitute \eqref{r-alpha} back into \eqref{differ-1}, \eqref{differ-2}, \eqref{differ-3} and consider the coefficients of $\eta_1\wedge\eta_2,\eta_2\wedge\eta_3$ and $\eta_3\wedge\eta_1$  to obtain 
\begin{equation}\label{bcde-ver}
\begin{gathered}
0\ =b_2+e_1\ =\ b_3-d_1\ =\ b_1-e_2\ =\ b_3+c_2\ =\ -b_2+c_3\ =\ b_1+d_3,\\
4r_{\alpha}r^{\alpha}\ =c_1+d_2\ =\ c_1+e_3\ =\ d_2+e_3.
\end{gathered}
\end{equation}
Let us define 
\begin{equation}\label{b-s}
\lambda_1=b_1,\qquad \lambda_2=b_2,\qquad \lambda_3=b_3.
\end{equation}
Then, the equations \eqref{bcde-ver} imply that
\begin{equation}\label{cde-s}
\begin{aligned}
c_1&=2r_{\alpha}r^{\alpha},&\quad c_2&=-\lambda_3,&\quad c_3&=\lambda_2,\\
d_1&=\lambda_3,& d_2&=2r_{\alpha}r^{\alpha},& d_3&=-\lambda_1,\\
e_1&=-\lambda_2,&e_2&=\lambda_1,&e_3&=2r_{\alpha}r^{\alpha}.
\end{aligned}
\end{equation}
Now, the equations \eqref{str_eq_P_o} follow by substituting \eqref{r-alpha},\eqref{b-s} and \eqref{cde-s} into \eqref{differ-theta} and \eqref{differ-phi}. It remains only to show that the one-forms
$\eta_1,\eta_2,\eta_3,\tilde\theta^{\alpha},\tilde\theta^{\bar\alpha},\tilde\varphi_0,\tilde\varphi_1,\tilde\varphi_2,\tilde\varphi_3$ are pointwise linearly independent. Indeed, we have the relation
\begin{equation}\label{trans-P_o}
\begin{pmatrix}
\eta_1\\
\eta_2\\
\eta_3\\
\tilde\theta^\alpha  \\
\tilde\theta^{\bar\alpha}\\
\tilde\varphi_0\\
\tilde\varphi_1\\
\tilde\varphi_2\\
\tilde\varphi_3
\end{pmatrix}
=\begin{pmatrix}
 1&0&0&0&0&0&0&0&0\\
 0&1&0&0&0&0&0&0&0\\
 0&0&1&0&0&0&0&0&0\\
  ir^\alpha&\pi^\alpha_{\bar\sigma} r^{\bar\sigma}&i\pi^\alpha_{\bar\sigma} r^{\bar\sigma}&U^\alpha_\beta&0&0&0&0&0\\
  -ir^{\bar\alpha}&\pi^{\bar\alpha}_{\sigma} r^{\sigma}&-i\pi^{\bar\alpha}_{\sigma} r^{\sigma}&0&U^{\bar\alpha}_{\bar\beta}&0&0&0&0\\
 \lambda_1&\lambda_2&\lambda_3&2U_{\beta\bar\sigma}r^{\bar\sigma}&2U_{\bar\beta\sigma}r^{\sigma}&1&0&0&0\\
 2r_\sigma r^\sigma&-\lambda_3&\lambda_2&-2iU_{\beta\bar\sigma}r^{\bar\sigma}&2iU_{\bar\beta\sigma}r^{\sigma}&0&1&0&0\\
 \lambda_3&2r_\sigma r^\sigma&-\lambda_1&-2\pi_{\sigma\tau}U^{\sigma}_{\beta}r^{\tau}&-2\pi_{\bar\sigma\bar\tau}U^{\bar\sigma}_{\bar\beta}r^{\bar\tau}&0&0&1&0\\
 -\lambda_2&\lambda_1&2r_\sigma r^\sigma&2i\pi_{\sigma\tau}U^{\sigma}_{\beta}r^{\tau}&-2i\pi_{\bar\sigma\bar\tau}U^{\bar\sigma}_{\bar\beta}r^{\bar\tau}&0&0&0&1
\end{pmatrix}
\begin{pmatrix}
\eta_1\\
\eta_2\\
\eta_3\\
\theta^\beta  \\
\theta^{\bar\beta}\\
\varphi_0\\
\varphi_1\\
\varphi_2\\
\varphi_3
\end{pmatrix},
\end{equation}

which is clearly a non-singular transformation, since $\{U^\alpha_\beta\}\in Sp(n)$ is non-singular.

\end{proof}

Let us denote by $G_1$ the set of all matrices (transformations) $A(U^\alpha_\beta,r^\alpha,\lambda_s)$ given by \eqref{trans-P_o} for some real numbers $\lambda_s$, complex numbers $r^\alpha$, and $\{U^\alpha_\beta\}\in Sp(n)$ . Then, it is easy to see that $G_1$ is a group;  for any two matrices $A(U^\alpha_\beta,r^\alpha,\lambda_s), А(\tilde U^\alpha_\beta,\tilde r^\alpha,\tilde\lambda_s)\in G_1$, 
 \begin{equation*}
 A(U^\alpha_\beta,r^\alpha,\lambda_s)\cdot A(\tilde U^\alpha_\beta,\tilde r^\alpha,\tilde\lambda_s)=A(\hat U^\alpha_\beta,\hat r^\alpha,\hat\lambda_s) \in G_1,
 \end{equation*}
 where
 \begin{equation}
 \begin{cases}
 \hat U^\alpha_\beta=U^\alpha_\sigma \tilde U^\sigma_\beta\\
 \hat r^\alpha=U^\alpha_\sigma \tilde r^\sigma + r^\alpha\\
 \hat\lambda_1=\lambda_1+\tilde\lambda_1+2iU_{\alpha\bar\beta}\tilde r^\alpha r^{\bar\beta}-2iU_{\bar\alpha\beta}\tilde r^{\bar\alpha} r^{\beta}\\
 \hat\lambda_2=\lambda_2+\tilde\lambda_2+2\pi^{\bar\sigma}_{\alpha}U_{\bar\sigma\beta}\tilde r^{\alpha} r^{\beta}+2\pi^{\sigma}_{\bar\alpha}U_{\sigma\bar\beta}\tilde r^{\bar\alpha} r^{\bar\beta}\\
 \hat\lambda_3=\lambda_3+\tilde\lambda_3-2i\pi^{\bar\sigma}_{\alpha}U_{\bar\sigma\beta}\tilde r^{\alpha} r^{\beta}+2i\pi^\sigma_{\bar\alpha}U_{\sigma\bar\beta}\tilde r^{\bar\alpha} r^{\bar\beta}.
 \end{cases}
 \end{equation}
 We have also the following formula for the inverse matrix
 \begin{equation}
\Big (A(U_{\beta}^{\alpha},r^\alpha, \lambda_s)\Big)^{-1}=A\Big((U_{\dt\beta}^{\alpha}),(-U_{\dt\sigma}^{\alpha} r^\sigma), (-\lambda_s)\Big).
 \end{equation}
 (Notice that, according to our conventions, $\{U_{\dt\beta}^{\alpha}\}$ is the inverse of $\{U_{\beta}^{\alpha}=U_{\beta\dt}^{\spc\alpha}\}$; this is because  $\{U_{\beta}^{\alpha}\}\in Sp(n)$ is an orthogonal transformation of $\R^{4n}$). 
 
To describe the corresponding representation of the Lie algebra $\mathfrak g_1$, we differentiate the representation \eqref{trans-P_o} of $G_1$ at the identity matrix $Id=A(\delta^\alpha_\beta,0,0)$ and introduce the following parameterization for $\mathfrak g_1$ (being the tangent space of $G_1$ at $Id$):
\begin{equation}
\Gamma_{\alpha\beta}\overset{def}{=} \pi_{\alpha\sigma}(dU^\sigma_\beta )_{|Id},\qquad \phi^\alpha\overset{def}{=}-(dr^\alpha)_{|Id},\qquad \psi_s\overset{def}{=}-(d\lambda_s)_{|Id}.
\end{equation}
 Then, the representation of $\mathfrak g_1$ on $TP_o$ is given by the transformations of the form
 \begin{equation}
\begin{pmatrix}
 0&0&0&0&0&0&0&0&0\\
 0&0&0&0&0&0&0&0&0\\
 0&0&0&0&0&0&0&0&0\\
  -i\phi^\alpha&-\pi^\alpha_{\bar\sigma} \phi^{\bar\sigma}&-i\pi^\alpha_{\bar\sigma} \phi^{\bar\sigma}&-\pi^{\alpha\sigma}\Gamma_{\sigma\beta}&0&0&0&0&0\\
  i\phi^{\bar\alpha}&-\pi^{\bar\alpha}_{\sigma} \phi^{\sigma}&i\pi^{\bar\alpha}_{\sigma} \phi^{\sigma}&0&-\pi^{\bar\alpha\bar\sigma}\Gamma_{\bar\sigma\bar\beta}&0&0&0&0\\
-\psi_1&-\psi_2&-\psi_3&-2\phi_{\beta}&-2\phi_{\bar\beta}&0&0&0&0\\
 0&\psi_3&-\psi_2&2i\phi_{\beta}&-2i\phi_{\bar\beta}&0&0&0&0\\
 -\psi_3&0&\psi_1&-2\pi_{\sigma\beta}\phi^{\sigma}&-2\pi_{\bar\sigma\bar\beta}\phi^{\bar\sigma}&0&0&0&0\\
\psi_2&-\psi_1&0&2i\pi_{\sigma\beta}\phi^{\sigma}&-2i\pi_{\bar\sigma\beta}\phi^{\bar\sigma}&0&0&0&0
\end{pmatrix}
\begin{pmatrix}
\eta_1\\
\eta_2\\
\eta_3\\
\theta^\beta  \\
\theta^{\bar\beta}\\
\varphi_0\\
\varphi_1\\
\varphi_2\\
\varphi_3
\end{pmatrix},
\end{equation}
 where $\Gamma_{\alpha\beta},\phi^\alpha$ are complex, $\psi_s$ are real and the following equations are satisfied:
\begin{equation}
\Gamma_{\alpha\beta}=\Gamma_{\beta\alpha},\qquad (\mathfrak j\Gamma)_{\alpha\beta}=\Gamma_{\alpha\beta},
\end{equation}
i.e., $\{\pi^{\alpha\sigma}\Gamma_{\sigma\beta}\}\in sp(n)$ (cf. Lemma~\ref{sp(n)}).

 By Lemma~\ref{Lex} and Lemma~\ref{P_o-transformations}, it follows that the manifold $P_o$ has an induced $G_1$-structure. Denote by $P_1$ its principle $G_1$-bundle and let $\pi_1:P_1\rightarrow P_o$ be the corresponding principle bundle projection. The local sections of $P_1$ are precisely the local coframings $\{\eta_1,\eta_2,\eta_2,\theta^\alpha,\theta^{\bar\alpha},\varphi_0,\varphi_1,\varphi_2,\varphi_3\}$ for $TP_o$ for which the assertion of Lemma~\ref{Lex} is satisfied. In $P_1$ there are intrinsically (and hence globally) defined one-forms (for which we keep the same notation) $\eta_1,\eta_2,\eta_2,\theta^\alpha,\theta^{\bar\alpha},\varphi_0,\varphi_1,\varphi_2,\varphi_3$ which are everywhere linearly independent and satisfy the structure equations 
 \begin{equation}\label{str-eq-deta}
\begin{cases}
d\eta_1=-\varphi_0\wedge\eta_1-\varphi_2\wedge\eta_3+\varphi_3\wedge\eta_2+2i g_{\alpha\bar\beta}\,\theta^{\alpha}\wedge\theta^{\bar\beta}\\
d\eta_2=-\varphi_0\wedge\eta_2-\varphi_3\wedge\eta_1+\varphi_1\wedge\eta_3+\pi_{\alpha\beta}\,\theta^{\alpha}\wedge\theta^{\beta}+\pi_{\bar\alpha\bar\beta}\,\theta^{\bar\alpha}\wedge\theta^{\bar\beta}\\
d\eta_3=-\varphi_0\wedge\eta_3-\varphi_1\wedge\eta_2 + \varphi_2\wedge\eta_1-i\pi_{\alpha\beta}\,\theta^{\alpha}\wedge\theta^{\beta}+i\pi_{\bar\alpha\bar\beta}\,\theta^{\bar\alpha}\wedge\theta^{\bar\beta}.
\end{cases}
\end{equation}

\begin{thrm}\label{Theorem_1} On $P_1$, there exists a unique set of complex one-forms $\Gamma_{\alpha\beta}, \phi^\alpha$ and real one-forms $\psi_1,\psi_2,\psi_3$  such that 
\begin{equation}\label{Gamma-symmetries}
\Gamma_{\alpha\beta}=\Gamma_{\beta\alpha},\qquad (\mathfrak j\Gamma)_{\alpha\beta}=\Gamma_{\alpha\beta}.
\end{equation}
and the equations
\begin{equation}\label{str-eq-con}
\begin{cases}
d\theta^\alpha=-i\phi^\alpha\wedge\eta_1-\pi^\alpha_{\bar\sigma}\phi^{\bar\sigma}\wedge(\eta_2+i\eta_3)-\pi^{\alpha\sigma}\Gamma_{\sigma\beta}\wedge\theta^\beta-\frac{1}{2}(\varphi_0+i\varphi_1)\wedge\theta^\alpha-\frac{1}{2}\pi^\alpha_{\bar\beta}(\varphi_2+i\varphi_3)\wedge\theta^{\bar\beta}\\
d\varphi_0=-\psi_1\wedge\eta_1-\psi_2\wedge\eta_2-\psi_3\wedge\eta_3-2\phi_\beta\wedge\theta^\beta-2\phi_{\bar\beta}\wedge\theta^{\bar\beta}\\
d\varphi_1=-\varphi_2\wedge\varphi_3-\psi_2\wedge\eta_3+\psi_3\wedge\eta_2+2i\phi_\beta\wedge\theta^\beta-2i\phi_{\bar\beta}\wedge\theta^{\bar\beta}\\
d\varphi_2=-\varphi_3\wedge\varphi_1-\psi_3\wedge\eta_1+\psi_1\wedge\eta_3-2\pi_{\sigma_\beta}\phi^\sigma\wedge\theta^\beta-2\pi_{\bar\sigma\bar\beta}\phi^{\bar\sigma}\wedge\theta^{\bar\beta}\\
d\varphi_3=-\varphi_1\wedge\varphi_2-\psi_1\wedge\eta_2+\psi_2\wedge\eta_1+2i\pi_{\sigma_\beta}\phi^\sigma\wedge\theta^\beta-2i\pi_{\bar\sigma\bar\beta}\phi^{\bar\sigma}\wedge\theta^{\bar\beta},
\end{cases}
\end{equation}
are satisfied.

Furthermore, the set 
\begin{equation}\label{ThrmGlobalCoframing}
\{\eta_1,\eta_2,\eta_2,\theta^\alpha,\theta^{\bar\alpha},\varphi_0,\varphi_1,\varphi_2,\varphi_3\}\cup\{\Gamma_{\alpha\beta} : \alpha\le\beta\}\cup\{\phi^\alpha,\phi^{\bar\alpha},\psi_1,\psi_2,\psi_3\}
\end{equation}
is a global coframing for the (complexified) tangent bundle $TP_1.$ 

In particular, if we suppose that $M'$ is a second qc manifold whose corresponding objects are denoted by dashes, then, in order that there is locally a diffeomorphism which maps  the qc structure of $M$ to this of $M'$, it is necessary and sufficient that there is a diffeomorphism of $P_1$ to $P_1'$ under which the forms in \eqref{ThrmGlobalCoframing} are respectively equal to the forms with dashes.

\end{thrm}
 \begin{proof}
 The exterior differentiation of the structure equations $\eqref{str-eq-deta}$ gives:
  \begin{equation}\label{diff-str-eq}
  \begin{split}
  0={}-d\varphi_0&\wedge\eta_1+\Big(d\varphi_3 + \varphi_1\wedge\varphi_2\Big)\wedge\eta_2-\Big(d\varphi_2+\varphi_3\wedge\varphi_1\Big)\wedge\eta_3
  -ig_{\bar\alpha\beta}\Big( 2d\theta^{\bar\alpha}+\phi_0\wedge\theta^{\bar\alpha}\\ 
  &+\pi^{\bar\alpha}_{\sigma}\,(\varphi_2-i\varphi_3)\wedge\theta^\sigma\Big)\wedge\theta^\beta
  +ig_{\alpha\bar\beta}\Big( 2d\theta^{\alpha}+\varphi_0\wedge\theta^{\alpha}+\pi^{\alpha}_{\bar\sigma}\,(\varphi_2+i\varphi_3)\wedge\theta^{\bar\sigma}\Big)\wedge\theta^{\bar\beta};\\
   0=-\Big(d\varphi_3& + \varphi_1\wedge\varphi_2\Big)\wedge\eta_1-d\varphi_0\wedge\eta_2+\Big(d\varphi_1+\varphi_2\wedge\varphi_3\Big)\wedge\eta_3
  +\pi_{\alpha\beta}\Big( 2d\theta^{\alpha}+\varphi_0\wedge\theta^{\alpha}\\
  &+i\varphi_1\wedge\theta^\alpha +i\pi_{\bar\sigma}^{\alpha}\,\varphi_3\wedge\theta^{\bar\sigma}\Big)\wedge\theta^\beta
  +\pi_{\bar\alpha\bar\beta}\Big( 2d\theta^{\bar\alpha}+\varphi_0\wedge\theta^{\bar\alpha}-i\varphi_1\wedge\theta^{\bar\alpha} -i\pi_{\sigma}^{\bar\alpha}\,\varphi_3\wedge\theta^{\sigma}\Big)\wedge\theta^{\bar\beta};\\
   0=\ {}\ \Big(d\varphi_2 &+ \varphi_3\wedge\varphi_1\Big)\wedge\eta_1-\Big(d\varphi_1+\varphi_2\wedge\varphi_3\Big)\wedge\eta_2-d\varphi_0\wedge\eta_3
  -i\pi_{\alpha\beta}\Big( 2d\theta^{\alpha}+\varphi_0\wedge\theta^{\alpha}\\
  &+i\varphi_1\wedge\theta^\alpha +\pi_{\bar\sigma}^{\alpha}\,\varphi_2\wedge\theta^{\bar\sigma}\Big)\wedge\theta^\beta
  +i\pi_{\bar\alpha\bar\beta}\Big( 2d\theta^{\bar\alpha}+\varphi_0\wedge\theta^{\bar\alpha}-i\varphi_1\wedge\theta^{\bar\alpha} +\pi_{\sigma}^{\bar\alpha}\,\varphi_2\wedge\theta^{\sigma}\Big)\wedge\theta^{\bar\beta}.
  \end{split}
  \end{equation}
 From this, it follows that 
 $
 d\theta^\alpha \equiv 0 
$
modulo $\{\eta_s,\theta^\beta,\theta^{\bar\beta}\}$.
 Letting 
 \begin{equation}\label{d-theta-alpha}
 d\theta^\alpha\equiv X^\alpha_\beta\wedge\theta^\beta + X^\alpha_{\bar\beta}\wedge\theta^{\bar\beta} \mod\ {\eta_s},
 \end{equation}
 for some one forms $X^\alpha_\beta, X^\alpha_{\bar\beta}$,
and substituting back into \eqref{diff-str-eq}, we compute modulo ${\eta_s}$ (i.e., by ignoring the therms involving $\eta_s$):
 \begin{equation}\label{str-modeta-1}
 \begin{split}
0\ &\equiv\  -  \Big(X_{\alpha\beta}-X_{\beta\alpha} + \pi_{\alpha\beta}(\varphi_2-i\varphi_3)\Big)\wedge\theta^{\alpha}\wedge\theta^{\beta}\\
&+ \Big(X_{\bar\alpha\bar\beta}-X_{\bar\beta\bar\alpha} + \pi_{\bar\alpha\bar\beta}(\varphi_2+i\varphi_3)\Big)\wedge\theta^{\bar\alpha}\wedge\theta^{\bar\beta} 
 + 2\Big(X_{\alpha\bar\beta}+X_{\bar\beta\alpha} + g_{\alpha\bar\beta}\varphi_0\Big)\wedge\theta^{\alpha}\wedge\theta^{\bar\beta};\\
0\ &\equiv\ \Big(\pi^{\bar\sigma}_{\alpha}X_{\beta\bar\sigma}-\pi^{\bar\sigma}_{\beta}X_{\alpha\bar\sigma} + \pi_{\alpha\beta}(\varphi_0+i\varphi_1)\Big)\wedge\theta^{\alpha}\wedge\theta^{\beta}\\
&+ \Big(\pi^{\sigma}_{\bar\alpha}X_{\bar\beta\sigma}-\pi^{\sigma}_{\bar\beta}X_{\bar\alpha\sigma} + \pi_{\bar\alpha\bar\beta}(\varphi_0-i\varphi_1)\Big)\wedge\theta^{\bar\alpha}\wedge\theta^{\bar\beta} 
 + 2\Big(\pi^{\bar\sigma}_{\alpha}X_{\bar\beta\bar\sigma}-\pi^{\sigma}_{\bar\beta}X_{\alpha\sigma}+ ig_{\alpha\bar\beta}\varphi_3\Big)\wedge\theta^{\alpha}\wedge\theta^{\bar\beta};\\
0\ &\equiv\  -i\Big(\pi^{\bar\sigma}_{\alpha}X_{\beta\bar\sigma}-\pi^{\bar\sigma}_{\beta}X_{\alpha\bar\sigma} + \pi_{\alpha\beta}(\varphi_0+i\varphi_1)\Big)\wedge\theta^{\alpha}\wedge\theta^{\beta}\\
&+ i\Big(\pi^{\sigma}_{\bar\alpha}X_{\bar\beta\sigma}-\pi^{\sigma}_{\bar\beta}X_{\bar\alpha\sigma} + \pi_{\bar\alpha\bar\beta}(\varphi_0-i\varphi_1)\Big)\wedge\theta^{\bar\alpha}\wedge\theta^{\bar\beta} 
 - 2i\Big(\pi^{\bar\sigma}_{\alpha}X_{\bar\beta\bar\sigma}+\pi^{\sigma}_{\bar\beta}X_{\alpha\sigma}+ g_{\alpha\bar\beta}\varphi_2\Big)\wedge\theta^{\alpha}\wedge\theta^{\bar\beta}.
 \end{split}
 \end{equation}
 Adding the third equation of \eqref{str-modeta-1}, multiplied by $i$, to the second gives
\begin{multline}\label{2+i3}
2\Big(\pi^{\bar\sigma}_{\alpha}X_{\beta\bar\sigma}-\pi^{\bar\sigma}_{\beta}X_{\alpha\bar\sigma} + \pi_{\alpha\beta}(\varphi_0+i\varphi_1)\Big)\wedge\theta^{\alpha}\wedge\theta^{\beta} \\
 + 2\pi^{\bar\sigma}_\alpha\Big(2X_{\bar\beta\bar\sigma}+ \pi_{\bar\beta\bar\sigma}(\varphi_2+i\varphi_3)\Big)\wedge\theta^{\alpha}\wedge\theta^{\bar\beta}\ \equiv\ 0 \qquad \mod{\eta_s}, 
\end{multline}
from which we deduce that both the expressions in the large parentheses vanish modulo ${\{\theta^\alpha,\theta^{\bar\alpha},\eta_s\}}$. Let
\begin{equation}\label{Xbarbar}
2X_{\bar\beta\bar\sigma}+ \pi_{\bar\beta\bar\sigma}(\varphi_2+i\varphi_3)\ \equiv\ Y_{\bar\beta\bar\sigma\gamma}\,\theta^\gamma+Y_{\bar\beta\bar\sigma\bar\gamma}\,\theta^{\bar\gamma}\mod{\eta_s}
\end{equation}
for some functions $Y_{\bar\beta\bar\sigma\gamma}$, $Y_{\bar\beta\bar\sigma\bar\gamma}.$  Then, by substituting back into \eqref{2+i3} and considering only the coefficient of $\theta^\alpha\wedge\theta^{\bar\beta}\wedge\theta^{\bar\gamma}$, we obtain the symmetry $Y_{\bar\beta\bar\sigma\bar\gamma}=Y_{\bar\gamma\bar\sigma\bar\beta}$. Therefore, we have that
\begin{multline}
d\theta^\alpha \equiv X^\alpha_\beta\wedge\theta^\beta+g^{\alpha\bar\sigma}X_{\bar\beta\bar\sigma}\wedge\theta^{\bar\beta} \equiv 
X^\alpha_\beta\wedge\theta^\beta+g^{\alpha\bar\sigma}\Big(X_{\bar\beta\bar\sigma}-\frac{1}{2}Y_{\bar\beta\bar\sigma\bar\gamma}\,\theta^{\bar\gamma}\Big)\wedge\theta^{\bar\beta}\\
\overset{\eqref{Xbarbar}}{\equiv} X^\alpha_\beta\wedge\theta^\beta+g^{\alpha\bar\sigma}\Big(-\frac{1}{2}\pi_{\bar\beta\bar\sigma}(\varphi_2+i\varphi_3)+\frac{1}{2}Y_{\bar\beta\bar\sigma\gamma}\,\theta^{\gamma}\Big)\wedge\theta^{\bar\beta}\\
 \equiv X^\alpha_\beta\wedge\theta^\beta+ \frac{1}{2}Y^{\spc\alpha\spc}_{\bar\beta\dt\gamma}\,\theta^\gamma\wedge\theta^{\bar\beta}-\frac{1}{2}\pi_{\bar\beta}^\alpha(\varphi_2+i\varphi_3)\wedge\theta^{\bar\beta}\\
 \equiv \Big(X^\alpha_\beta-\frac{1}{2}Y^{\spc\alpha\spc}_{\bar\gamma\dt\beta}\,\theta^{\bar\gamma}\Big)\wedge\theta^\beta-\frac{1}{2}\pi_{\bar\beta}^\alpha(\varphi_2+i\varphi_3)\wedge\theta^{\bar\beta}\qquad \mod{\eta_s}.
\end{multline}
This means that among all the one-forms $X^\alpha_\beta,X^\alpha_{\bar\beta}$, for which \eqref{d-theta-alpha} is satisfied, we can find such that
\begin{equation}\label{assume-X-alpha-beta}
 X^\alpha_{\bar\beta} = -\frac{1}{2}\pi_{\bar\beta}^\alpha(\varphi_2+i\varphi_3).
 \end{equation} 
 
 Assuming \eqref{assume-X-alpha-beta}, we have
\begin{equation}\label{d-theta-alpha-2}
 d\theta^\alpha\equiv X^\alpha_\beta\wedge\theta^\beta -\frac{1}{2}\pi_{\bar\beta}^\alpha(\varphi_2+i\varphi_3)\wedge\theta^{\bar\beta} \mod\ {\eta_s}
 \end{equation}
and the equations \eqref{str-modeta-1} (modulo $\eta_s$) become:
 \begin{equation}\label{str-modeta-2}
 \begin{split}
0\ \equiv\   \Big(&X_{\alpha\bar\beta}+X_{\bar\beta\alpha} + g_{\alpha\bar\beta}\varphi_0\Big)\wedge\theta^{\alpha}\wedge\theta^{\bar\beta};\\
0\ \equiv\ \Big(&\pi^{\bar\sigma}_{\alpha}X_{\beta\bar\sigma}-\pi^{\bar\sigma}_{\beta}X_{\alpha\bar\sigma} + \pi_{\alpha\beta}(\varphi_0+i\varphi_1)\Big)\wedge\theta^{\alpha}\wedge\theta^{\beta}\\
&+ \Big(\pi^{\sigma}_{\bar\alpha}X_{\bar\beta\sigma}-\pi^{\sigma}_{\bar\beta}X_{\bar\alpha\sigma} + \pi_{\bar\alpha\bar\beta}(\varphi_0-i\varphi_1)\Big)\wedge\theta^{\bar\alpha}\wedge\theta^{\bar\beta};\\
0\ \equiv\  \Big(&\pi^{\bar\sigma}_{\alpha}X_{\beta\bar\sigma}-\pi^{\bar\sigma}_{\beta}X_{\alpha\bar\sigma} + \pi_{\alpha\beta}(\varphi_0+i\varphi_1)\Big)\wedge\theta^{\alpha}\wedge\theta^{\beta}\\
&-\Big(\pi^{\sigma}_{\bar\alpha}X_{\bar\beta\sigma}-\pi^{\sigma}_{\bar\beta}X_{\bar\alpha\sigma} + \pi_{\bar\alpha\bar\beta}(\varphi_0-i\varphi_1)\Big)\wedge\theta^{\bar\alpha}\wedge\theta^{\bar\beta}.
 \end{split}
 \end{equation}

Let us define
\begin{equation}
\tilde\Gamma_{\alpha\beta}\overset{def}{=}\pi^{\bar\sigma}_\alpha\Big(X_{\beta\bar\sigma}+\frac{1}{2}g_{\beta\bar\sigma}(\varphi_0+i\varphi_1)\Big).
\end{equation}
Then, \eqref{d-theta-alpha-2} and \eqref{str-modeta-2} yield
\begin{equation}\label{tilde-Gamma}
\begin{cases}
d\theta^\alpha\equiv -\pi^{\alpha\sigma}\tilde\Gamma_{\sigma\beta}\wedge\theta^\beta-\frac{1}{2}(\varphi_0+i\varphi_1)\wedge\theta^\alpha -\frac{1}{2}\pi_{\bar\beta}^\alpha(\varphi_2+i\varphi_3)\wedge\theta^{\bar\beta}\\
(\pi^{\bar\sigma}_{\alpha}\tilde\Gamma_{\bar\sigma\bar\beta}+\pi^{\sigma}_{\bar\beta}\tilde\Gamma_{\sigma\alpha})\wedge\theta^\alpha\wedge\theta^{\bar\beta}\equiv 0\\
(\tilde\Gamma_{\alpha\beta}-\tilde\Gamma_{\beta\alpha})\wedge\theta^\alpha\wedge\theta^\beta \equiv 0
\end{cases}\qquad\mod\ {\eta_s}
\end{equation}   
The third equation of \eqref{tilde-Gamma} implies the the existence of (unique) functions $C_{\alpha\beta\gamma}$ satisfying
\begin{equation}\label{prop-C}
\begin{cases}
\tilde\Gamma_{\alpha\beta}-\tilde\Gamma_{\beta\alpha}\equiv C_{\alpha\beta\gamma}\theta^\gamma\mod {\eta_s}\\
C_{\alpha\beta\gamma}=-C_{\beta\alpha\gamma}\\
C_{\alpha\beta\gamma}+C_{\beta\gamma\alpha}+C_{\gamma\alpha\beta}=0.
\end{cases}
\end{equation}

Let $D_{\alpha\beta\gamma}\overset{def}{=}\frac{1}{3}(C_{\alpha\beta\gamma}+C_{\alpha\gamma\beta})$ and 
\begin{equation*}
\hat\Gamma_{\alpha\beta}\overset{def}{=}\tilde\Gamma_{\alpha\beta}-D_{\alpha\beta\gamma}\,\theta^\gamma.
\end{equation*}
Then, modulo $\eta_s$,  
\begin{multline*}
\hat\Gamma_{\alpha\beta}-\hat\Gamma_{\beta\alpha}\equiv \tilde\Gamma_{\alpha\beta}-\tilde\Gamma_{\beta\alpha}-(D_{\alpha\beta\gamma}-D_{\beta\alpha\gamma})\theta^\gamma
\equiv C_{\alpha\beta\gamma}\theta^\gamma
-(D_{\alpha\beta\gamma}-D_{\beta\alpha\gamma})\theta^\gamma
\\
\equiv C_{\alpha\beta\gamma}\theta^\gamma
-\frac{1}{3}(C_{\alpha\beta\gamma}+C_{\alpha\gamma\beta}-C_{\beta\alpha\gamma}-C_{\beta\gamma\alpha})\theta^\gamma
\\
\equiv C_{\alpha\beta\gamma}\theta^\gamma
-\frac{1}{3}(3C_{\alpha\beta\gamma}\underbrace{-C_{\alpha\beta\gamma}-C_{\gamma\alpha\beta}-C_{\beta\gamma\alpha}}_{=\,0\ \text{by \eqref{prop-C}}})\theta^\gamma
 \equiv 0.
\end{multline*}
 And since $D_{\alpha\beta\gamma}=D_{\alpha\gamma\beta}$, the equations \eqref{tilde-Gamma} become
 \begin{equation}\label{hat-Gamma}
\begin{cases}
d\theta^\alpha\equiv -\pi^{\alpha\sigma}\hat\Gamma_{\sigma\beta}\wedge\theta^\beta-\frac{1}{2}(\varphi_0+i\varphi_1)\wedge\theta^\alpha -\frac{1}{2}\pi_{\bar\beta}^\alpha(\varphi_2+i\varphi_3)\wedge\theta^{\bar\beta}\\
(\pi^{\bar\sigma}_{\alpha}\hat\Gamma_{\bar\sigma\bar\beta}+\pi^{\sigma}_{\bar\beta}\hat\Gamma_{\sigma\alpha})\wedge\theta^\alpha\wedge\theta^{\bar\beta}\equiv 0\\
\hat\Gamma_{\alpha\beta}-\hat\Gamma_{\beta\alpha} \equiv 0
\end{cases}\qquad\mod\ {\eta_s}
\end{equation}   
 
 By the second identity of \eqref{hat-Gamma}, $\pi^{\bar\sigma}_{\alpha}\hat\Gamma_{\bar\sigma\bar\beta}+\pi^{\sigma}_{\bar\beta}\hat\Gamma_{\sigma\alpha} \equiv 0$ mod $\{\theta^\alpha,\theta^{\bar\alpha},\eta_s\}$. Hence there exists functions $A_{\alpha\beta\gamma}, B_{\alpha\beta\gamma}$ so that
 \begin{equation}\label{ex-A-B}
 \pi^{\bar\sigma}_{\alpha}\hat\Gamma_{\bar\sigma\bar\beta}+\pi^{\sigma}_{\bar\beta}\hat\Gamma_{\sigma\alpha}\equiv \pi^{\sigma}_{\bar\beta}\,A_{\alpha\sigma\gamma}\,\theta^{\gamma}+\pi^{\bar\sigma}_\alpha\, B_{\bar\sigma\bar\beta\bar\gamma}\theta^{\bar\gamma} \mod {\eta_s}
 \end{equation}
 and $A_{\alpha\beta\gamma}=A_{\gamma\beta\alpha},$ $B_{\alpha\beta\gamma}=B_{\alpha\gamma\beta}$.
 We multiply \eqref{ex-A-B} by $\pi^{\bar\beta}_{\tau}$ and sum over $\bar\beta$ to obtain
 \begin{equation}\label{ex-A-B-2}
 (\mathfrak j\hat\Gamma)_{\alpha\tau}-\hat\Gamma_{\tau\alpha}\equiv-A_{\alpha\tau\gamma}\,\theta^{\gamma}+\pi^{\bar\sigma}_\alpha\,\pi^{\bar\beta}_\tau\, B_{\bar\sigma\bar\beta\bar\gamma}\theta^{\bar\gamma} \mod {\eta_s}.
 \end{equation}
 Since, by the third identity of \eqref{hat-Gamma}, the LHS of the \eqref{ex-A-B-2} is symmetric (modulo $\eta_s$) in the indices $\alpha,\tau$, so is the RHS. This implies that both $A_{\alpha\beta\gamma}$ and $B_{\alpha\beta\gamma}$ are totally symmetric in $\alpha,\beta,\gamma$.   
 
 Applying the map $\mathfrak j$ to both sides of \eqref{ex-A-B-2} gives
\begin{equation*}
 \hat\Gamma_{\alpha\tau}-(\mathfrak j\hat\Gamma)_{\alpha\tau}\equiv \pi^{\bar\mu}_\alpha\,\pi^{\bar\nu}_\tau\, \Big(-A_{\bar\mu\bar\nu\bar\gamma}\,\theta^{\bar\gamma}+\pi^{\sigma}_{\bar\mu}\,\pi^{\beta}_{\bar\nu}\, B_{\sigma\beta\gamma}\theta^{\gamma}\Big)\equiv B_{\alpha\tau\gamma}\theta^{\gamma} - \pi^{\bar\mu}_\alpha\,\pi^{\bar\nu}_\tau\, A_{\bar\mu\bar\nu\bar\gamma}\mod{\eta_s},
 \end{equation*}
 which, by comparison with the initial equation \eqref{ex-A-B-2}, yields $A_{\alpha\beta\gamma}=B_{\alpha\beta\gamma}$.
 
If we set
 \begin{equation*}
 \breve\Gamma_{\alpha\beta}\overset{def}{=}\hat\Gamma_{\alpha\beta}-A_{\alpha\beta\gamma}\,\theta^\gamma,
 \end{equation*}
 then, by \eqref{hat-Gamma}, \eqref{ex-A-B-2}, and the just established properties of $A_{\alpha\beta\gamma}=B_{\alpha\beta\gamma}$, we obtain that 
 \begin{equation}\label{breve-Gamma}
\begin{cases}
(\mathfrak j\breve\Gamma)_{\alpha\beta}\equiv \breve \Gamma_{\alpha\beta}\\
\breve\Gamma_{\alpha\beta}\equiv \breve \Gamma_{\beta\alpha}
\end{cases}\qquad\mod\ {\eta_s}
\end{equation}
and
\begin{equation}\label{dtheta-breve}
d\theta^\alpha =  -\pi^{\alpha\sigma}\breve\Gamma_{\sigma\beta}\wedge\theta^\beta-\frac{1}{2}(\varphi_0+i\varphi_1)\wedge\theta^\alpha -\frac{1}{2}\pi_{\bar\beta}^\alpha(\varphi_2+i\varphi_3)\wedge\theta^{\bar\beta} + (X_s)^\alpha\wedge\eta_s,
\end{equation}
for some one-forms $(X_s)^\alpha$.
 
 Substituting \eqref{dtheta-breve} back into \eqref{diff-str-eq} yields:
  \begin{equation}\label{diff-str-eq-3I}
  \begin{split}
  0={}-\Big(&d\varphi_0-2ig_{\bar\alpha\beta}(X_1)^{\bar\alpha}\wedge\theta^{\beta}+2ig_{\alpha\bar\beta}(X_1)^\alpha\wedge\theta^{\bar\beta}\Big)\wedge\eta_1
  +\Big(d\varphi_3 + \varphi_1\wedge\varphi_2+2ig_{\bar\alpha\beta}(X_2)^{\bar\alpha}\wedge\theta^{\beta}
  \\&-2ig_{\alpha\bar\beta}(X_2)^{\alpha}\wedge\theta^{\bar\beta}\Big)\wedge\eta_2
  -\Big(d\varphi_2+\varphi_3\wedge\varphi_1-2ig_{\bar\alpha\beta}(X_3)^{\bar\alpha}\wedge\theta^{\beta}+2ig_{\alpha\bar\beta}(X_3)^\alpha\wedge\theta^{\bar\beta}\Big)\wedge\eta_3
  \\&-2i\pi^{\sigma}_{\bar\beta} \Big( \breve\Gamma_{\sigma\alpha}-(\mathfrak j \breve\Gamma)_{\alpha\sigma}\Big)\wedge\theta^\alpha\wedge\theta^{\bar\beta};\\
 \end{split}
 \end{equation}
  \begin{equation}\label{diff-str-eq-3II}
  \begin{split}
  0={}-\Big(&d\varphi_3 + \varphi_1\wedge\varphi_2+2\pi_{\alpha\beta}(X_1)^{\alpha}\wedge\theta^{\beta}+2\pi_{\bar\alpha\bar\beta}(X_1)^{\bar\alpha}\wedge\theta^{\bar\beta}\Big)\wedge\eta_1
  -\Big(d\varphi_0+2\pi_{\alpha\beta}(X_2)^{\alpha}\wedge\theta^{\beta}
  \\&+2\pi_{\bar\alpha\bar\beta}(X_2)^{\bar\alpha}\wedge\theta^{\bar\beta}\Big)\wedge\eta_2
  +\Big(d\varphi_1+\varphi_2\wedge\varphi_3-2\pi_{\alpha\beta}(X_3)^{\alpha}\wedge\theta^{\beta}-2\pi_{\bar\alpha\bar\beta}(X_3)^{\bar\alpha}\wedge\theta^{\bar\beta}\Big)\wedge\eta_3
  \\&+2\breve\Gamma_{\alpha\beta}\wedge\theta^\alpha\wedge\theta^\beta+2 \breve\Gamma_{\bar\alpha\bar\beta}\wedge\theta^{\bar\alpha}\wedge\theta^{\bar\beta};\\
 \end{split}
 \end{equation}
  \begin{equation}\label{diff-str-eq-3III}
  \begin{split}
   0\ ={}\ \Big(&d\varphi_2 + \varphi_3\wedge\varphi_1+2i\pi_{\alpha\beta}(X_1)^{\alpha}\wedge\theta^{\beta}-2i\pi_{\bar\alpha\bar\beta}(X_1)^{\bar\alpha}\wedge\theta^{\bar\beta}\Big)\wedge\eta_1
  -\Big(d\varphi_1+\varphi_2\wedge\varphi_3
  \\&-2i\pi_{\alpha\beta}(X_2)^{\alpha}\wedge\theta^{\beta}
  +2i\pi_{\bar\alpha\bar\beta}(X_2)^{\bar\alpha}\wedge\theta^{\bar\beta}\Big)\wedge\eta_2
  -\Big(d\varphi_0-2i\pi_{\alpha\beta}(X_3)^{\alpha}\wedge\theta^{\beta}+2i\pi_{\bar\alpha\bar\beta}(X_3)^{\bar\alpha}\wedge\theta^{\bar\beta}\Big)\wedge\eta_3
  \\&-2i\breve\Gamma_{\alpha\beta}\wedge\theta^\alpha\wedge\theta^\beta+2i \breve\Gamma_{\bar\alpha\bar\beta}\wedge\theta^{\bar\alpha}\wedge\theta^{\bar\beta}.
  \end{split}
  \end{equation}
 
Notice that, by \eqref{breve-Gamma}, the last summands in the above three equations vanish modulo $\eta_s$. Hence if we wedge \eqref{diff-str-eq-3I} with $\eta_2\wedge\eta_3$, \eqref{diff-str-eq-3II} with $\eta_3\wedge\eta_1$, and take the difference, we obtain 
 \begin{equation}
 0= \Big(-2ig_{\bar\alpha\beta}((X_1)^{\bar\alpha}-i\pi^{\bar\alpha}_{\tau}(X_2)^{\tau})\wedge\theta^{\beta}+2ig_{\alpha\bar\beta}((X_1)^\alpha+i\pi^{\alpha}_{\bar\tau}(X_2)^{\bar\tau})\wedge\theta^{\bar\beta}\Big)\wedge\eta_1\wedge\eta_2\wedge\eta_3+\dots, 
 \end{equation}	
 the unwritten terms being of the form $\Big(\bigwedge^2\{\theta^\alpha,\theta^{\bar\alpha}\}\Big)\wedge\eta_1\wedge\eta_2\wedge\eta_3$,
which yields  that
\begin{equation*}
(X_1)^{\alpha}+i\pi^{\alpha}_{\bar\tau}(X_2)^{\bar\tau}\equiv 0\mod {\{\theta^\alpha,\theta^{\bar\alpha},\eta_s\}}.
\end{equation*}

 Similarly, by wedging \eqref{diff-str-eq-3I} with $\eta_2\wedge\eta_3$, \eqref{diff-str-eq-3III} with $\eta_1\wedge\eta_2$, and considering the difference of the resulting two equations, we obtain 
 that
\begin{equation*}
(X_1)^{\alpha}-\pi^{\alpha}_{\bar\tau}(X_3)^{\bar\tau}\equiv 0\mod {\{\theta^\alpha,\theta^{\bar\alpha},\eta_s\}}.
\end{equation*}

Therefore, if we set
\begin{equation*}
\breve\phi^\alpha \ \overset{def}{=} \ i(X_1)^\alpha, 
\end{equation*}
we have
\begin{equation}\label{breve-phi}
\begin{cases}
(X_1)^\alpha = -i\breve\phi^\alpha\\
(X_2)^\alpha = -\pi^{\alpha}_{\bar\sigma}\breve\phi^{\bar\sigma}+Y^{\alpha}_{\beta}\,\theta^{\beta}+Y^{\alpha}_{\bar\beta}\,\theta^{\bar\beta}+\dots\\
(X_3)^\alpha = -i\pi^{\alpha}_{\bar\sigma}\breve\phi^{\bar\sigma}+Z^{\alpha}_{\beta}\,\theta^{\beta}+Z^{\alpha}_{\bar\beta}\,\theta^{\bar\beta}+\dots,
\end{cases}
\end{equation}
where the omitted terms are linear expressions in $\eta_s$ and $Y^{\alpha}_{\beta},Y^{\alpha}_{\bar\beta},Z^{\alpha}_{\beta}, Z^{\alpha}_{\bar\beta}$ are some appropriate coefficients.

Let us wedge \eqref{diff-str-eq-3III} with $\eta_1\wedge\eta_2$, \eqref{diff-str-eq-3II} with $\eta_3\wedge\eta_1$ and take the difference. We obtain
\begin{multline*}
 0= \Big(2\pi_{\alpha\beta}((X_2)^{\alpha}+i(X_3)^{\alpha})\wedge\theta^{\beta}+2\pi_{\bar\alpha\bar\beta}((X_2)^{\bar\alpha}-i(X_3)^{\bar\alpha})\wedge\theta^{\bar\beta}\Big)\wedge\eta_1\wedge\eta_2\wedge\eta_3\\
 +\ast\,\theta^\alpha\wedge\theta^\beta\wedge\eta_1\wedge\eta_2\wedge\eta_3 
 + \ast\,\theta^{\bar\alpha}\wedge\theta^{\bar\beta}\wedge\eta_1\wedge\eta_2\wedge\eta_3\\
 \overset{\eqref{breve-phi}}{=} 
\Big(-2\pi_{\sigma\alpha}(Y^{\sigma}_{\bar\beta}+iZ^{\sigma}_{\bar\beta})+2\pi_{\bar\sigma\bar\beta}(Y^{\bar\sigma}_{\alpha}-iZ^{\bar\sigma}_\alpha)\Big)\,\theta^\alpha\wedge\theta^{\bar\beta}\wedge\eta_1\wedge\eta_2\wedge\eta_3\\
+\ast\,\theta^\alpha\wedge\theta^\beta\wedge\eta_1\wedge\eta_2\wedge\eta_3 
 + \ast\,\theta^{\bar\alpha}\wedge\theta^{\bar\beta}\wedge\eta_1\wedge\eta_2\wedge\eta_3,
 \end{multline*}
where the asterisk represents coefficients which are irrelevant at the moment.   It follows that
\begin{equation*}
\pi_{\sigma\alpha}(Y^{\sigma}_{\bar\beta}+iZ^{\sigma}_{\bar\beta})=\pi_{\bar\sigma\bar\beta}(Y^{\bar\sigma}_{\alpha}-iZ^{\bar\sigma}_\alpha).
\end{equation*}

Similarly, by considering the sum of \eqref{diff-str-eq-3II}, wedged with  $\eta_1\wedge\eta_2$, and \eqref{diff-str-eq-3III}, wedged with  $\eta_3\wedge\eta_1$, we obtain that
\begin{equation*}
\pi_{\sigma\alpha}(Y^{\sigma}_{\bar\beta}+iZ^{\sigma}_{\bar\beta})=-\pi_{\bar\sigma\bar\beta}(Y^{\bar\sigma}_{\alpha}-iZ^{\bar\sigma}_\alpha).
\end{equation*}
Hence 
$
Z^\alpha_{\bar\beta}=iY^\alpha_{\bar\beta}.
$
Setting
\begin{equation*}
\hat\phi^{\alpha}\ \overset{def}{=}\ \breve\phi^{\alpha}+\pi^\alpha_{\bar\sigma}Y^{\bar\sigma}_{\beta}\,\theta^{\beta},
\end{equation*}
the equations \eqref{breve-phi} become
\begin{equation}\label{hat-phi}
\begin{cases}
(X_1)^\alpha = -i\hat\phi^\alpha+i\pi^\alpha_{\bar\sigma}Y^{\bar\sigma}_{\beta}\,\theta^{\beta}\\
(X_2)^\alpha = -\pi^{\alpha}_{\bar\sigma}\hat\phi^{\bar\sigma}+Y^{\alpha}_{\beta}\,\theta^{\beta}+\dots\\
(X_3)^\alpha = -i\pi^{\alpha}_{\bar\sigma}\hat\phi^{\bar\sigma}+Z^{\alpha}_{\beta}\,\theta^{\beta}+\dots
\end{cases}
\end{equation}
(the dots represent terms which are linear in $\eta_s$).

We define
\begin{equation}\label{def-Gamma}
\Gamma_{\alpha\beta}\ \overset{def}{=}\ \breve\Gamma_{\alpha\beta}-ig_{\alpha\bar\sigma}\, Y^{\bar\sigma}_\beta\,\eta_1-\pi_{\alpha\sigma}\, Y^{\sigma}_\beta\,\eta_2-i\pi_{\alpha\sigma}\, Z^{\sigma}_\beta\,\eta_3.
\end{equation}

Substituting \eqref{def-Gamma} and \eqref{hat-phi} into \eqref{dtheta-breve}, we obtain
\begin{multline}\label{dtheta-C}
d\theta^\alpha =  -\pi^{\alpha\sigma}\Gamma_{\sigma\beta}\wedge\theta^\beta-\frac{1}{2}(\varphi_0+i\varphi_1)\wedge\theta^\alpha -\frac{1}{2}\pi_{\bar\beta}^\alpha(\varphi_2+i\varphi_3)\wedge\theta^{\bar\beta} 
-i\hat\phi^\alpha\wedge\eta_1
\\-\pi^\alpha_{\bar\sigma} \hat\phi^{\bar\sigma}\wedge(\eta_2+i\eta_3)+(C_{st})^\alpha\eta_s\wedge\eta_t,
\end{multline}
where $(C_{st})^\alpha=-(C_{st})^{\alpha}$ are some appropriate coefficients.

Finally, if we set
\begin{equation*}
\begin{split}
\phi^\alpha\ \overset{def}{=}\ \hat\phi^{\alpha}&+\Big(\pi^\alpha_{\bar\sigma}\,(C_{12})^{\bar\sigma}-(C_{23})^{\alpha}-i\pi^\alpha_{\bar\sigma}\,(C_{31})^{\bar\sigma}\Big)\,\eta_1
+i\Big(-(C_{12})^{\alpha}+\pi^\alpha_{\bar\sigma}\,(C_{23})^{\bar\sigma}+i(C_{31})^{\alpha}\Big)\,\eta_2\\
&+\Big(-(C_{12})^{\alpha}-\pi^\alpha_{\bar\sigma}\,(C_{23})^{\bar\sigma}+i(C_{31})^{\alpha}\Big)\,\eta_3,
\end{split}
\end{equation*}
then, by \eqref{dtheta-C}, we obtain that
\begin{equation}\label{d-theta}
d\theta^\alpha =  -\pi^{\alpha\sigma}\Gamma_{\sigma\beta}\wedge\theta^\beta-\frac{1}{2}(\varphi_0+i\varphi_1)\wedge\theta^\alpha -\frac{1}{2}\pi_{\bar\beta}^\alpha(\varphi_2+i\varphi_3)\wedge\theta^{\bar\beta} 
-i\phi^\alpha\wedge\eta_1
-\pi^\alpha_{\bar\sigma} \phi^{\bar\sigma}\wedge(\eta_2+i\eta_3).
\end{equation}

So far, we have shown, by \eqref{breve-Gamma} and \eqref{def-Gamma}, that the equations $\Gamma_{\alpha\beta}\equiv \Gamma_{\beta\alpha}$ and $(\mathfrak  j\Gamma)_{\alpha\beta}\equiv \Gamma_{\alpha\beta}$ are satisfied only modulo $\eta_s$. Let us assume
\begin{equation}\label{ver-Gamma}
\begin{cases}
\pi^{\sigma}_{\bar\beta}\,\Big((\mathfrak  j\Gamma)_{\alpha\sigma} - \Gamma_{\sigma\alpha}\Big)= (A_s)_{\alpha\bar\beta}\,\eta_s\\
\Gamma_{\alpha\beta}-\Gamma_{\beta\alpha} = (B_s)_{\alpha\beta}\,\eta_s
\end{cases}
\end{equation}
for some functions $(A_s)_{\alpha\bar\beta}$ and $(B_s)_{\alpha\beta}=-(B_s)_{\beta\alpha}$.

Substituting \eqref{d-theta} and \eqref{ver-Gamma}  back into \eqref{diff-str-eq} gives:

 \begin{equation}\label{diff-str-eq-4I}
  \begin{split}
  0={}-\Big(&d\varphi_0+2\phi_{\beta}\wedge\theta^{\beta}+2\phi_{\bar\beta}\wedge\theta^{\bar\beta}-2i (A_1)_{\alpha\bar\beta}\theta^\alpha\wedge\theta^{\bar\beta}\Big)\wedge\eta_1
  \\&+\Big(d\varphi_3 + \varphi_1\wedge\varphi_2-2i\pi_{\alpha\beta}\phi^{\alpha}\wedge\theta^{\beta}
  +2i\pi_{\bar\alpha\bar\beta}\phi^{\bar\alpha}\wedge\theta^{\bar\beta} +2i(A_2)_{\alpha\bar\beta}\theta^\alpha\wedge\theta^{\bar\beta}\Big)\wedge\eta_2
 \\& -\Big(d\varphi_2+\varphi_3\wedge\varphi_1+2\pi_{\alpha\beta}\phi^{\alpha}\wedge\theta^{\beta}+2\pi_{\bar\alpha\bar\beta}\phi^{\bar\alpha}\wedge\theta^{\bar\beta}
  -2i(A_3)_{\alpha\bar\beta}\theta^\alpha\wedge\theta^{\bar\beta}\Big)\wedge\eta_3;
   \end{split}
 \end{equation}
  \begin{equation}\label{diff-str-eq-4II}
  \begin{split}
  0={}-\Big(&d\varphi_3 + \varphi_1\wedge\varphi_2-2i\pi_{\alpha\beta}\phi^{\alpha}\wedge\theta^{\beta}+2i\pi_{\bar\alpha\bar\beta}\phi^{\bar\alpha}\wedge\theta^{\bar\beta} - (B_1)_{\alpha\beta}\theta^\alpha\wedge\theta^\beta-(B_1)_{\bar\alpha\bar\beta}\theta^{\bar\alpha}\wedge\theta^{\bar\beta}\Big)\wedge\eta_1
 \\& -\Big(d\varphi_0+2\phi_\beta\wedge\theta^{\beta}+2\phi_{\bar\beta}\wedge\theta^{\bar\beta}
  -(B_2)_{\alpha\beta}\theta^\alpha\wedge\theta^\beta-(B_2)_{\bar\alpha\bar\beta}\theta^{\bar\alpha}\wedge\theta^{\bar\beta}\Big)\wedge\eta_2
  \\&+\Big(d\varphi_1+\varphi_2\wedge\varphi_3-2i\phi_{\beta}\wedge\theta^{\beta}+2i\phi_{\bar\beta}\wedge\theta^{\bar\beta}+(B_3)_{\alpha\beta}\theta^{\alpha}\wedge\theta^{\beta}+(B_3)_{\bar\alpha\bar\beta}\theta^{\bar\alpha}\wedge\theta^{\bar\beta}\Big)\wedge\eta_3;
 \end{split}
 \end{equation}
  \begin{equation}\label{diff-str-eq-4III}
  \begin{split}
   0\ ={}\ \Big(&d\varphi_2 + \varphi_3\wedge\varphi_1+2\pi_{\alpha\beta}\phi^{\alpha}\wedge\theta^{\beta}+2\pi_{\bar\alpha\bar\beta}\phi^{\bar\alpha}\wedge\theta^{\bar\beta}
   -i(B_1)_{\alpha\beta}\theta^{\alpha}\wedge\theta^{\beta}+i(B_1)_{\bar\alpha\bar\beta}\theta^{\bar\alpha}\wedge\theta^{\bar\beta}\Big)\wedge\eta_1
  \\&-\Big(d\varphi_1+\varphi_2\wedge\varphi_3
  -2i\phi_{\beta}\wedge\theta^{\beta}
  +2i\phi_{\bar\beta}\wedge\theta^{\bar\beta}+i(B_2)_{\alpha\beta}\theta^{\alpha}\wedge\theta^{\beta}-i(B_2)_{\bar\alpha\bar\beta}\theta^{\bar\alpha}\wedge\theta^{\bar\beta}\Big)\wedge\eta_2
  \\&-\Big(d\varphi_0+2\phi_{\beta}\wedge\theta^{\beta}+2\phi_{\bar\beta}\wedge\theta^{\bar\beta}+i(B_3)_{\alpha\beta}\theta^{\alpha}\wedge\theta^{\beta}
  -i(B_3)_{\bar\alpha\bar\beta}\theta^{\bar\alpha}\wedge\theta^{\bar\beta}\Big)\wedge\eta_3.
  \end{split}
  \end{equation}

By wedging the equation \eqref{diff-str-eq-4I} with $\eta_2\wedge\eta_3$ and subtracting from the result the equation \eqref{diff-str-eq-4I} wedged with $\eta_3\wedge\eta_1$, we see that
\begin{equation}
2i(A_1)_{\alpha\bar\beta}\theta^\alpha\wedge\theta^{\bar\beta}\wedge\eta_1\wedge\eta_2\wedge\eta_3=-(B_2)_{\alpha\beta}\theta^\alpha\wedge\theta^{\beta}\wedge\eta_1\wedge\eta_2\wedge\eta_3-2(B_2)_{\bar\alpha\bar\beta}\theta^{\bar\alpha}\wedge\theta^{\bar\beta}\wedge\eta_1\wedge\eta_2\wedge\eta_3,
\end{equation} 
and hence $(A_1)_{\alpha\bar\beta}=0, \ (B_2)_{\alpha\beta}=0$. Proceeding similarly, we obtain $(A_s)_{\alpha\bar\beta}=0$, $(B_s)_{\alpha\beta}=0$ and thus, by \eqref{ver-Gamma},      
the properties \eqref{Gamma-symmetries}. A substitution back into \eqref{diff-str-eq-4I}, \eqref{diff-str-eq-4II} and  \eqref{diff-str-eq-4III} yields:

\begin{equation}\label{diff-str-eq-5}
  \begin{split}
  0={}-\Big(&d\varphi_0+2\phi_{\beta}\wedge\theta^{\beta}+2\phi_{\bar\beta}\wedge\theta^{\bar\beta}\Big)\wedge\eta_1
  +\Big(d\varphi_3 + \varphi_1\wedge\varphi_2
  -2i\pi_{\alpha\beta}\phi^{\alpha}\wedge\theta^{\beta}
  \\&+2i\pi_{\bar\alpha\bar\beta}\phi^{\bar\alpha}\wedge\theta^{\bar\beta} \Big)\wedge\eta_2
  -\Big(d\varphi_2+\varphi_3\wedge\varphi_1+2\pi_{\alpha\beta}\phi^{\alpha}\wedge\theta^{\beta}+2\pi_{\bar\alpha\bar\beta}\phi^{\bar\alpha}\wedge\theta^{\bar\beta}\Big)\wedge\eta_3;\\
  0={}-\Big(&d\varphi_3 + \varphi_1\wedge\varphi_2-2i\pi_{\alpha\beta}\phi^{\alpha}\wedge\theta^{\beta}+2i\pi_{\bar\alpha\bar\beta}\phi^{\bar\alpha}\wedge\theta^{\bar\beta}\Big)\wedge\eta_1
 -\Big(d\varphi_0+2\phi_\beta\wedge\theta^{\beta}
 \\&+2\phi_{\bar\beta}\wedge\theta^{\bar\beta}\Big)\wedge\eta_2
  +\Big(d\varphi_1+\varphi_2\wedge\varphi_3-2i\phi_{\beta}\wedge\theta^{\beta}+2i\phi_{\bar\beta}\wedge\theta^{\bar\beta}\Big)\wedge\eta_3;\\
   0\ ={}\, \ \ \Big(&d\varphi_2 + \varphi_3\wedge\varphi_1+2\pi_{\alpha\beta}\phi^{\alpha}\wedge\theta^{\beta}+2\pi_{\bar\alpha\bar\beta}\phi^{\bar\alpha}\wedge\theta^{\bar\beta}\Big)\wedge\eta_1
  -\Big(d\varphi_1+\varphi_2\wedge\varphi_3
  -2i\phi_{\beta}\wedge\theta^{\beta}
  \\&+2i\phi_{\bar\beta}\wedge\theta^{\bar\beta}\Big)\wedge\eta_2
  -\Big(d\varphi_0+2\phi_{\beta}\wedge\theta^{\beta}+2\phi_{\bar\beta}\wedge\theta^{\bar\beta}\Big)\wedge\eta_3.
  \end{split}
  \end{equation}

These imply the existence of real one-forms $A_s$, $B_{st}$ so that
\begin{equation}\label{As-Bst}
\begin{cases}
d\varphi_0=-2\phi_\beta\wedge\theta^\beta-2\phi_{\bar\beta}\wedge\theta^{\bar\beta}+A_s\wedge\eta_s\\
d\varphi_1=-\varphi_2\wedge\varphi_3+2i\phi_\beta\wedge\theta^\beta-2i\phi_{\bar\beta}\wedge\theta^{\bar\beta}+B_{s1}\wedge\eta_s\\
d\varphi_2=-\varphi_3\wedge\varphi_1-2\pi_{\sigma_\beta}\phi^\sigma\wedge\theta^\beta-2\pi_{\bar\sigma\bar\beta}\phi^{\bar\sigma}\wedge\theta^{\bar\beta}+B_{s2}\wedge\eta_s\\
d\varphi_3=-\varphi_1\wedge\varphi_2+2i\pi_{\sigma_\beta}\phi^\sigma\wedge\theta^\beta-2i\pi_{\bar\sigma\bar\beta}\phi^{\bar\sigma}\wedge\theta^{\bar\beta}+B_{s3}\wedge\eta_s.
\end{cases}
\end{equation}

Substituting \eqref{As-Bst} into \eqref{diff-str-eq-5}, we obtain the relations
\begin{equation*}
B_{11}=B_{22}=B_{33}=0,\quad B_{23}=-B_{32}=A_1,\quad B_{31}=-B_{13}=A_2,\quad B_{12}=-B_{21}=A_3. 
\end{equation*}
Hence, if setting
\begin{equation*}
\psi_s\ \overset{def}{=}\ -A_s,
\end{equation*} 
the one-forms $\Gamma_{\alpha\beta}$, $\phi^\alpha$, $\psi_s$  satisfy \eqref{str-eq-con} and are as required in the theorem.  
 
 To prove the uniqueness, assume that $\hat\Gamma_{\alpha\beta}$, $\hat\phi^\alpha$, $\hat\psi_s$ are any other one-forms satisfying the requirements of the theorem and let
 \begin{equation}
 L_{\alpha\beta}= \Gamma_{\alpha\beta}-\hat\Gamma_{\alpha\beta},\qquad M^\alpha = \phi^\alpha-\hat\phi^\alpha, \qquad N_s=\psi_s-\hat\psi_s.
 \end{equation}
 Then, $L_{\alpha\beta}=L_{\beta\alpha}$, $(\mathfrak jL)_{\alpha\beta}=L_{\alpha\beta}$ and by subtraction, we obtain the identities 
\begin{equation}\label{str-eq-dif}
\begin{cases}
0=-iM^\alpha\wedge\eta_1-\pi^\alpha_{\bar\sigma}M^{\bar\sigma}\wedge(\eta_2+i\eta_3)-\pi^{\alpha\sigma}L_{\sigma\beta}\wedge\theta^\beta\\
0=-N_1\wedge\eta_1-N_2\wedge\eta_2-N_3\wedge\eta_3-2M_\beta\wedge\theta^\beta-2M_{\bar\beta}\wedge\theta^{\bar\beta}\\
0=-N_2\wedge\eta_3+N_3\wedge\eta_2+2iM_\beta\wedge\theta^\beta-2iM_{\bar\beta}\wedge\theta^{\bar\beta}\\
0=-N_3\wedge\eta_1+N_1\wedge\eta_3-2\pi_{\sigma_\beta}M^\sigma\wedge\theta^\beta-2\pi_{\bar\sigma\bar\beta}M^{\bar\sigma}\wedge\theta^{\bar\beta}\\
0=-N_1\wedge\eta_2+N_2\wedge\eta_1+2i\pi_{\sigma_\beta}M^\sigma\wedge\theta^\beta-2i\pi_{\bar\sigma\bar\beta}M^{\bar\sigma}\wedge\theta^{\bar\beta}.
\end{cases}
\end{equation}
 It follows that
\begin{equation}\label{LMN}
\begin{cases}
L_{\alpha\beta}=L_{\alpha\beta\gamma}\theta^\gamma + L_{\alpha\beta\bar\gamma}\theta^{\bar\gamma} + (L_s)_{\alpha\beta}\eta_s\\
M^{\alpha}= M^{\alpha}_\beta\theta^\beta + M^\alpha_{\bar\beta}\theta^{\bar\beta}+(M_s)^{\alpha}\eta_s\\
N_{s}= (N_{s})_\beta\theta^\beta + (N_s)_{\bar\beta}\theta^{\bar\beta}+N_{ts}\eta_t,
\end{cases}
\end{equation}
where $L_{\alpha\beta\gamma}$,  $L_{\alpha\beta\bar\gamma}$, $(L_s)_{\alpha\beta}$, $M^\alpha_\beta$, $M^\alpha_{\bar\beta}$, $(M_s)^\alpha$, $(N_s)^\alpha$, $N_{st}$ are some appropriate coefficients, and also
\begin{equation}\label{prop-L}
L_{\alpha\beta\gamma}=L_{\beta\alpha\gamma},\qquad 
L_{\alpha\beta\bar\gamma}=\pi^{\bar\mu}_{\alpha}\pi^{\bar\nu}_{\beta}L_{\bar\mu\bar\nu\bar\gamma}, \qquad(L_s)_{\alpha\beta}=(L_s)_{\beta\alpha},\qquad 
(L_s)_{\alpha\beta}=\pi^{\bar\mu}_{\alpha}\pi^{\bar\nu}_{\beta}(L_s)_{\bar\mu\bar\nu}.
\end{equation}  
We substitute \eqref{LMN} into the first equation of \eqref{str-eq-dif}. The vanishing of the coefficient of $\theta^\beta\wedge\theta^{\bar\gamma}$  gives
$
L_{\sigma\beta\bar\gamma}=0
$
which, by the second identity of \eqref{prop-L}, implies that also  $L_{\alpha\beta\gamma}=0$. Proceeding similarly, we easily obtain that  the rest of the coefficients in \eqref{LMN} vanish as well and hence the uniqueness of the one-forms $\Gamma_{\alpha\beta}$, $\phi^\alpha$, $\psi_s$.
 
 Finally, the fact that the one-forms
 \begin{equation}
 \{\eta_1,\eta_2,\eta_2,\theta^\alpha,\theta^{\bar\alpha},\varphi_0,\varphi_1,\varphi_2,\varphi_3\}\cup\{\Gamma_{\alpha\beta} : \alpha\le\beta\}\cup\{\phi^\alpha,\phi^{\bar\alpha},\psi_1,\psi_2,\psi_3\}
 \end{equation}
 are pointwise linearly independent is easily derived form the observation that, by construction, $\eta_1,\eta_2,\eta_2,\theta^\alpha,\theta^{\bar\alpha},\varphi_0,\varphi_1,\varphi_2,\varphi_3$ are pointwise linearly independent and semibasic (w.r.t. the projection $\pi_1: P_1\rightarrow P_o$), whereas $$\{\Gamma_{\alpha\beta} : \alpha\le\beta\}\cup\{\phi^\alpha,\phi^{\bar\alpha},\psi_1,\psi_2,\psi_3\}$$ are independent modulo $\{\eta_1,\eta_2,\eta_2,\theta^\alpha,\theta^{\bar\alpha},\varphi_0,\varphi_1,\varphi_2,\varphi_3\}$. The latter is a consequence of the structure equations $\eqref{str-eq-con}$.
 
 \end{proof}

 \section{The Curvature and the Bianchi identities}

In order to understand the curvature components and their properties, we
shall have to compute the full structure equations, the corresponding Bianchi
identities and some of their differential consequences, all in parallel.
Thus we shall first provide three collections of the 
resulting formulae (in the
next three propositions) and then go through all the computations in one
package. Perhaps reading the computations will enlighten the three
propositions best. Further links are available in the appendix.
   
 \begin{prop}[Curvature components] \label{curvature} On $P_1$, there exist unique,  globally defined,  complex-valued functions
 \begin{equation}
 \mathcal S_{\alpha\beta\gamma\delta},\  \mathcal V_{\alpha\beta\gamma},\ \mathcal  L_{\alpha\beta},\ \mathcal  M_{\alpha\beta}, \ \mathcal C_\alpha, \  \mathcal H_\alpha,\  \mathcal P, \  \mathcal Q,\  \mathcal R
 \end{equation}
 so that:
 
 {\bf (I)} Each of the arrays  $\{\mathcal S_{\alpha\beta\gamma\delta}\},\{ \mathcal V_{\alpha\beta\gamma}\},\{ \mathcal L_{\alpha\beta}\}, \{ \mathcal M_{\alpha\beta}\}$ is totally symmetric in its indices.
 
 {\bf (II)} We have
 \begin{equation}\label{properties-curvature}
 \begin{cases}
 (\mathfrak j{\mathcal S})_{\alpha\beta\gamma\delta}=\mathcal S_{\alpha\beta\gamma\delta}\\
 (\mathfrak j{\mathcal L})_{\alpha\beta}=\mathcal L_{\alpha\beta}\\
 \overline {\mathcal R}=\mathcal R.
 \end{cases}
 \end{equation}
 
{\bf  (III)} The exterior derivatives $d\Gamma_{\alpha\beta}$, $d\phi_\alpha$ and $d\psi_s$ are given by  
 \begin{equation}\label{dGamma_ab}
 \begin{split}
 d\Gamma_{\alpha\beta}\ =\ &-\pi^{\sigma\tau}\Gamma_{\alpha\sigma}\wedge\Gamma_{\tau\beta} + 2\pi^{\bar\sigma}_{\alpha}(\phi_\beta\wedge\theta_{\bar\sigma}-\phi_{\bar\sigma}\wedge\theta_\beta)+2\pi^{\bar\sigma}_{\beta}(\phi_\alpha\wedge\theta_{\bar\sigma}-\phi_{\bar\sigma}\wedge\theta_\alpha)\\
 &+\pi^{\sigma}_{\bar\delta}\,\mathcal S_{\alpha\beta\gamma\sigma}\,\theta^{\gamma}\wedge\theta^{\bar\delta}+\Big(\mathcal V_{\alpha\beta\gamma}\,\theta^\gamma
 +\pi^{\bar\sigma}_{\alpha}\,\pi^{\bar\tau}_{\beta}\,\mathcal V_{\bar\sigma\bar\tau\bar\gamma}\,\theta^{\bar\gamma}\Big)\wedge\eta_1\\
 &-i\pi^{\sigma}_{\bar\gamma}\,\mathcal V_{\alpha\beta\sigma}\,\theta^{\bar\gamma}\wedge(\eta_2+i\eta_3)+i(\mathfrak j\mathcal V)_{\alpha\beta\gamma}\,\theta^{\gamma}\wedge(\eta_2-i\eta_3)\\
 &-i\mathcal L_{\alpha\beta}\,(\eta_2+i\eta_3)\wedge(\eta_2-i\eta_3)+\mathcal M_{\alpha\beta}\,\eta_1\wedge(\eta_2+i\eta_3)+(\mathfrak j M)_{\alpha\beta}\,\eta_1\wedge(\eta_2-i\eta_3),
  \end{split}
 \end{equation}
 
 \begin{equation}\label{dphi_a}
 \begin{split}
 d\phi_{\alpha}\ =\ \  \ &\frac{1}{2}(\varphi_0+i\varphi_1)\wedge\phi_\alpha 
 +\frac{1}{2}\pi_{\alpha\gamma}(\varphi_2-i\varphi_3)\wedge\phi^{\gamma}
 -\pi^{\bar\sigma}_{\alpha}\, \Gamma_{\bar\sigma\bar\gamma}\wedge\phi^{\bar\gamma}
 -\frac{i}{2}\,\psi_1\wedge\theta_\alpha\\
  -\ &\frac{1}{2}\,\pi_{\alpha\gamma}(\psi_2-i\psi_3)\wedge\theta^{\gamma}
 -i\pi_{\bar\delta}^{\sigma}\,\mathcal V_{\alpha\gamma\sigma}\,\theta^{\gamma}\wedge\theta^{\bar\delta}+\mathcal M_{\alpha\gamma}\,\theta^\gamma\wedge\eta_1+\pi^{\bar\sigma}_{\alpha}\,\mathcal L_{\bar\sigma\bar\gamma}\,\theta^{\bar\gamma}\wedge\eta_1\\
 +\ &i\mathcal L_{\alpha\gamma}\,\theta^{\gamma}\wedge(\eta_2-i\eta_3)-i\pi^{\sigma}_{\bar\gamma}\,\mathcal M_{\alpha\sigma}\,\theta^{\bar\gamma}\wedge(\eta_2+i\eta_3)
 -\mathcal C_{\alpha}(\eta_2+i\eta_3)\wedge(\eta_2-i\eta_3)\\
 +\ &\mathcal H_\alpha\,\eta_1\wedge(\eta_2+i\eta_3)+i\pi_{\alpha\sigma}\,\mathcal C^\sigma\,\eta_1\wedge(\eta_2-i\eta_3),
  \end{split}
 \end{equation}
 
 \begin{equation}\label{dpsi_1}
 \begin{split}
 d\psi_1\ =\ \  \ &\varphi_0\wedge\psi_1-\varphi_2\wedge\psi_3+\varphi_3\wedge\psi_2-4i \phi_\gamma\wedge\phi^\gamma+4\pi^{\sigma}_{\bar\delta}\,\mathcal L_{\gamma\sigma}\,\theta^\gamma\wedge\theta^{\bar\delta}+4\mathcal C_\gamma\,\theta^\gamma\wedge\eta_1\\
 +\ &4\mathcal C_{\bar\gamma}\,\theta^{\bar\gamma}\wedge\eta_1-4i\pi_{\bar\gamma\bar\sigma}\, \mathcal C^{\bar\sigma}\,\theta^{\bar\gamma}\wedge(\eta_2+i\eta_3)
 +4i\pi_{\gamma\sigma}\, \mathcal C^{\sigma}\,\theta^{\gamma}\wedge(\eta_2-i\eta_3)\\
 +\ &\mathcal P\,\eta_1\wedge(\eta_2+i\eta_3)+\overline{\mathcal P}\,\eta_1\wedge(\eta_2-i\eta_3)+i\mathcal R\,(\eta_2+i\eta_3)\wedge(\eta_2-i\eta_3),
 \end{split}
 \end{equation}

\begin{equation}\label{dpsi_23}
 \begin{split}
 d\psi_2+i\,d\psi_3\ =\ \  \ &(\varphi_0-i\varphi_1)\wedge(\psi_2+i\psi_3)+i(\varphi_2+i\varphi_3)\wedge\psi_1+4\pi_{\gamma\delta}\phi^\gamma\wedge\phi^\delta+4i\pi^{\bar\sigma}_{\gamma}\,\mathcal M_{\bar\sigma\bar\delta}\,\theta^\gamma\wedge\theta^{\bar\delta}\\
 +&\ 4i \pi^{\bar\sigma}_{\gamma}\,\mathcal C_{\bar\sigma}\,\theta^\gamma\wedge\eta_1
 -4\mathcal H_{\bar\gamma}\,\theta^{\bar\gamma}\wedge\eta_1-4i\mathcal C_{\bar\gamma}\,\theta^{\bar\gamma}\wedge(\eta_2+i\eta_3)
 -4i\pi_{\gamma}^{\bar\sigma}\, \mathcal H_{\bar\sigma}\,\theta^{\gamma}\wedge(\eta_2-i\eta_3)\\
 -\ &i\mathcal R\,\eta_1\wedge(\eta_2+i\eta_3)+\overline{\mathcal Q}\,\eta_1\wedge(\eta_2-i\eta_3)-\overline{\mathcal P}\,(\eta_2+i\eta_3)\wedge(\eta_2-i\eta_3).
 \end{split}
 \end{equation}

 \end{prop}
 
In order to describe further relations and differential consequences between
the curvature components defined by
Proposition~\ref{curvature}, we introduce the following list of one-forms:

\begin{multline}\label{S*}
\mathcal S^\ast_{\alpha\beta\gamma\delta}\overset{def}{\ =\ }
d{\mathcal S}_{\alpha\beta\gamma\delta}- \tilde{\mathcal S}^\ast_{\alpha\beta\gamma\delta} 
\\
\tilde{\mathcal S}^\ast_{\alpha\beta\gamma\delta} =
\pi^{\tau\nu}\Gamma_{\nu\alpha}\,\mathcal S_{\tau\beta\gamma\delta}
 + \pi^{\tau\nu}\Gamma_{\nu\beta}\,\mathcal S_{\alpha\tau\gamma\delta} + 
\pi^{\tau\nu}\Gamma_{\nu\gamma}\,\mathcal S_{\alpha\beta\tau\delta}
 +\pi^{\tau\nu}\Gamma_{\nu\delta}\,\mathcal S_{\alpha\beta\gamma\tau}
 +\varphi_0\, \mathcal S_{\alpha\beta\gamma\delta}
 +2i\big(\pi_{\alpha\tau}\,\mathcal V_{\delta\beta\gamma}+
\\
\pi_{\beta\tau}\,\mathcal V_{\alpha\gamma\delta}+\pi_{\gamma\tau}\,\mathcal V_{\alpha\beta\delta}
 +\pi_{\delta\tau}\,\mathcal V_{\alpha\beta\gamma}\big)\theta^{\tau}
 +2i\Big(g_{\alpha\bar\tau}\,(\mathfrak j\mathcal V)_{\delta\beta\gamma}
 +g_{\beta\bar\tau}\,(\mathfrak j\mathcal V)_{\alpha\delta\gamma}+g_{\gamma\bar\tau}\,(\mathfrak j\mathcal V)_{\alpha\beta\delta}
 +g_{\delta\bar\tau}\,(\mathfrak j\mathcal V)_{\alpha\beta\gamma}\Big)\theta^{\bar\tau}
\end{multline} 
\begin{multline}\label{V*}
 \mathcal V^\ast_{\alpha\beta\gamma}\overset{def}{\ =\ }d{\mathcal V}_{\alpha\beta\gamma}
- \tilde{\mathcal V}^\ast_{\alpha\beta\gamma}
\\ 
\tilde{\mathcal V}^\ast_{\alpha\beta\gamma} = 
\pi^{\tau\nu}\Gamma_{\nu\alpha}\,\mathcal V_{\tau\beta\gamma}
 +\pi^{\tau\nu}\Gamma_{\nu\beta}\,\mathcal V_{\alpha\tau\gamma}
+ \pi^{\tau\nu}\Gamma_{\nu\gamma}\,\mathcal V_{\alpha\beta\tau}
+ i\pi^{\sigma}_{\bar\tau}\,\phi^{\bar\tau}\,\mathcal S_{\alpha\beta\gamma\sigma}
+ \frac{1}{2}\big(3\varphi_0+i\varphi_1\big)\, \mathcal V_{\alpha\beta\gamma} -
\\ 
\frac{1}{2}\big(\varphi_2-i\varphi_3\big)\, (\mathfrak j\mathcal V)_{\alpha\beta\gamma}
- 2\big(\pi_{\alpha\tau}\,\mathcal M_{\beta\gamma}+\pi_{\beta\tau}\,\mathcal M_{\alpha\gamma}+\pi_{\gamma\tau}\,\mathcal M_{\alpha\beta}\big)\theta^{\tau}
- 2\big(g_{\alpha\bar\tau}\, \mathcal L_{\beta\gamma}
 +g_{\beta\bar\tau}\,\mathcal L_{\alpha\gamma}+g_{\gamma\bar\tau}\,\mathcal L_{\alpha\beta}\big)\theta^{\bar\tau}
 \end{multline}
\begin{multline}\label{L*}
 \mathcal L^\ast_{\alpha\beta}\overset{def}{\ =\ }d\mathcal L_{\alpha\beta}- 
\tilde{\mathcal L}^\ast_{\alpha\beta}
\\
\tilde{\mathcal L}^\ast_{\alpha\beta} = 
\pi^{\tau\sigma}\Gamma_{\sigma\alpha}\,\mathcal L_{\tau\beta}
+\pi^{\tau\sigma}\Gamma_{\sigma\beta}\,\mathcal L_{\alpha\tau}
+ 2\varphi_0\,\mathcal L_{\alpha\beta}
+\frac{1}{2}\big(\varphi_2+i\varphi_3\big)\mathcal M_{\alpha\beta}
+\frac{1}{2}\big(\varphi_2-i\varphi_3\big)(\mathfrak j\mathcal M)_{\alpha\beta}
+ \phi^\sigma\,\mathcal V_{\alpha\beta\sigma}
+ \\
\pi^{\bar\mu}_\alpha\,\pi^{\bar\nu}_\beta\,\phi^{\bar\sigma}\,\mathcal V_{\bar\mu\bar\nu\bar\sigma}
+ 2i\big(\pi_{\alpha\tau}\mathcal C_\beta+\pi_{\beta\tau}\mathcal C_\alpha\big)\theta^\tau
+ 2i\big(g_{\alpha\bar\tau}\,\pi^{\bar\sigma}_{\beta}\,\mathcal C_{\bar\sigma}+g_{\beta\bar\tau}\,\pi^{\bar\sigma}_{\alpha}\,\mathcal C_{\bar\sigma}\big)\theta^{\bar\tau}
 \end{multline} 
\begin{multline}\label{M*}
 \mathcal M^\ast_{\alpha\beta}\overset{def}{\ =\ }d\mathcal M_{\alpha\beta}
- \tilde{\mathcal M}^\ast_{\alpha\beta}
\\ 
\tilde{\mathcal M}^\ast_{\alpha\beta} = 
\pi^{\tau\sigma}\Gamma_{\sigma\alpha}\,\mathcal M_{\tau\beta}
+\pi^{\tau\sigma}\Gamma_{\sigma\beta}\,\mathcal M_{\alpha\tau}
+\big(2\varphi_0+i\varphi_1\big)\,\mathcal M_{\alpha\beta}
-\big(\varphi_2-i\varphi_3\big)\mathcal L_{\alpha\beta}
-2\pi^{\sigma}_{\bar\tau}\,\phi^{\bar\tau}\,\mathcal V_{\alpha\beta\sigma}
 + \\
2\big(\pi_{\alpha\tau}\mathcal H_\beta+\pi_{\beta\tau}\mathcal H_\alpha\big)\theta^\tau
 +2i\big(g_{\alpha\bar\tau}\,\mathcal C_{\beta}+g_{\beta\bar\tau}\,\mathcal C_{\alpha}\big)\theta^{\bar\tau}
 \end{multline} 
\begin{multline}\label{C*}
\mathcal C^\ast_{\alpha}\overset{def}{\ =\ }d\mathcal C_\alpha-
\tilde{\mathcal C}^\ast_{\alpha}\\
\tilde{\mathcal C}^\ast_{\alpha} = 
\pi^{\tau\sigma}\Gamma_{\sigma\alpha}\,\mathcal C_\tau 
+ \frac{1}{2}\big(5\varphi_0+i\varphi_1\big)\,\mathcal C_\alpha 
- \pi^{\bar\sigma}_{\alpha}\big(\varphi_2
-i\varphi_3\big)\mathcal C_{\bar\sigma}
- 2i\pi^{\sigma}_{\bar\tau}\,\phi^{\bar\tau}\,\mathcal L_{\alpha\sigma} + 
i\phi^{\tau}\,\mathcal M_{\alpha\tau} 
+\frac{i}{2} \big(\varphi_2+
\\
i\varphi_3\big)\mathcal H_{\alpha}
-\frac{1}{2}\pi_{\alpha\tau}\,\mathcal P\,\theta^\tau
+ \frac{1}{2}g_{\alpha\bar\tau}\,\mathcal R\,\theta^{\bar\tau}
 \end{multline}
\begin{multline}\label{H*}
\mathcal H^\ast_{\alpha}\overset{def}{\ =\ }d\mathcal H_\alpha- 
\tilde{\mathcal H}^\ast_{\alpha}
\\
\tilde{\mathcal H}^\ast_{\alpha} = 
\pi^{\tau\sigma}\Gamma_{\sigma\alpha}\,\mathcal H_\tau 
 +\frac{1}{2}\big(5\varphi_0+3i\varphi_1\big)\,\mathcal H_\alpha
+\frac{3i}{2}\big(\varphi_2-i\varphi_3\big)\mathcal C_{\alpha}
-3\pi^{\sigma}_{\bar\tau}\,\phi^{\bar\tau}\,\mathcal M_{\alpha\sigma}
+\frac{1}{2}\pi_{\alpha\tau}\,\mathcal Q\,\theta^\tau+\frac{i}{2}g_{\alpha\bar\tau}\,\mathcal P\,\theta^{\bar\tau}
 \end{multline}
\begin{equation}\label{R*}
\mathcal R^\ast\overset{def}{\ =\ }d\mathcal R -\tilde{\mathcal
R}^\ast,\quad
\tilde{\mathcal
R}^\ast  = 3\varphi_0\,\mathcal R -\big(\varphi_2+i\varphi_3\big)\mathcal P
-\big(\varphi_2-i\varphi_3\big)\overline{\mathcal P}-8\phi^\tau\,\mathcal C_\tau+8\phi^{\bar\tau}\,\mathcal C_{\bar\tau}
 \end{equation}
 \begin{equation}\label{P*}
\mathcal P^\ast\overset{def}{\ =\ }d\mathcal P - \tilde{\mathcal
P}^\ast,\quad \tilde{\mathcal
P}^\ast =
 \big(3\varphi_0+i\varphi_1\big)\mathcal P - 
\frac{i}{2}\big(\varphi_2 -i\varphi_3\big){\mathcal Q}
+\frac{3}{2}\big(\varphi_2-i\varphi_3\big){\mathcal R}+4i\phi^\tau\mathcal H_{\tau}
-12\pi_{\bar\tau\bar\sigma}\phi^{\bar\tau}\,\mathcal C_{\bar\sigma}
 \end{equation}
  \begin{equation}\label{Q*}
\mathcal Q^\ast\overset{def}{\ =\ }d\mathcal Q - \tilde{\mathcal Q}^\ast
,\quad \tilde{\mathcal Q}^\ast =
 \big(3\varphi_0+2i\varphi_1\big)\mathcal Q - 
2i\big(\varphi_2-i\varphi_3\big){\mathcal P}+
16\pi_{\bar\tau\bar\sigma}\phi^{\bar\tau}\mathcal H^{\bar\sigma}
 \end{equation}

 Using \eqref{properties-curvature} we easily obtain that  one-forms $\mathcal S^\ast_{\alpha\beta\gamma\delta},\, \mathcal V^\ast_{\alpha\beta\gamma},\, \mathcal L^\ast_{\alpha\beta},\,  \mathcal M^\ast_{\alpha \beta}$ are totally symmetric in their indices and
 \begin{equation}
 \begin{cases}
 (\mathfrak j{\mathcal S}^\ast)_{\alpha\beta\gamma\delta}=\mathcal S^\ast_{\alpha\beta\gamma\delta}\\
 (\mathfrak j{\mathcal L}^\ast)_{\alpha\beta}=\mathcal L^\ast_{\alpha\beta}\\
 \overline {\mathcal R^\ast}=\mathcal R^\ast.
 \end{cases}
 \end{equation}

 \begin{prop}[Bianchi identities]\label{bianchi}
The following identities are satisfied:
 \begin{multline}\label{d2Gamma_ab}
 d^2\Gamma_{\alpha\beta}\  = \ \pi^{\sigma}_{\bar\delta}{\mathcal S}^\ast_{\alpha\beta\gamma\sigma}\wedge\theta^\gamma\wedge\theta^{\bar\delta}
+\ {\mathcal V}^\ast_{\alpha\beta\gamma}\wedge\theta^\gamma\wedge\eta_1 
 +\ \pi^{\bar\mu}_\alpha\,\pi^{\bar\nu}_\beta\,{\mathcal V}^\ast_{\bar\mu\bar\nu\bar\gamma}\wedge\theta^{\bar\gamma}\wedge\eta_1 \\
 -\ i\pi^{\sigma}_{\bar\gamma}\,{\mathcal V}^\ast_{\alpha\beta\sigma}\wedge\theta^{\bar\gamma}\wedge\big(\eta_2+i\eta_3\big) 
+\ i\pi^{\bar\mu}_\alpha\,\pi^{\bar\nu}_\beta\,\pi^{\bar\xi}_\gamma\,{\mathcal V}^\ast_{\bar\mu\bar\nu\bar\xi}\wedge\theta^{\gamma}\wedge\big(\eta_2-i\eta_3\big) \\
-i\mathcal L_{\alpha\beta}^\ast\wedge\big(\eta_2+i\eta_3\big)\wedge\big(\eta_2-i\eta_3\big) 
+\mathcal M^\ast_{\alpha\beta}\wedge\eta_1\wedge\big(\eta_2+i\eta_3\big) \\
+\pi^{\bar\mu}_\alpha\,\pi^{\bar\nu}_\beta {\mathcal M}^\ast_{\bar\mu\bar\nu}\wedge\eta_1\wedge\big(\eta_2-i\eta_3\big) \ =\ 0;
 \end{multline} 
 \begin{multline}\label{d2phi_a}
 d^2\phi_\alpha\  = \  
-\ i\pi^{\nu}_{\bar\gamma}\,{\mathcal V}^\ast_{\alpha\beta\nu}\wedge\theta^{\beta}\wedge\theta^{\bar\gamma} 
+\pi^{\bar\mu}_\alpha\,\mathcal L_{\bar\mu\bar\beta}^\ast\wedge\theta^{\bar\beta}\wedge\eta_1
 +\mathcal M^\ast_{\alpha\beta}\wedge\theta^{\beta}\wedge\eta_1\\
-i\pi^{\nu}_{\bar\beta}\,\mathcal M^\ast_{\alpha\nu}\wedge\theta^{\bar\beta}\wedge\big(\eta_2+i\eta_3\big) 
+i\mathcal L^\ast_{\alpha\beta}\wedge\theta^{\beta}\wedge\big(\eta_2-i\eta_3\big) 
-\mathcal C^\ast_\alpha
 \wedge\big(\eta_2+i\eta_3\big)\wedge\big(\eta_2-i\eta_3\big) \\
+i\pi_{\alpha}^{\bar\mu}\,\mathcal C^\ast_{\bar\mu}\wedge\eta_1\wedge\big(\eta_2-i\eta_3\big) 
 +\mathcal H^\ast_\alpha
 \wedge\eta_1\wedge\big(\eta_2+i\eta_3\big) \ =\ 0;
 \end{multline}
 \begin{multline}\label{d2psi_1}
 d^2\psi_1\  = \ 4\pi^\mu_{\bar\gamma}\, \mathcal L^\ast_{\beta\mu}\wedge\theta^{\beta}\wedge\theta^{\bar\gamma}+4\mathcal C^\ast_\beta\wedge\theta^\beta\wedge\eta_1 
 +4\mathcal C^\ast_{\bar\gamma}\wedge\theta^{\bar\gamma}\wedge\eta_1
 +4i\pi_{\beta}^{\bar\mu}\,\mathcal C^\ast_{\bar\mu}\wedge\theta^{\beta}\wedge\big(\eta_2-i\eta_3\big)\\ 
-4i\pi^{\mu}_{\bar\gamma}\,\mathcal C^\ast_{\mu}\wedge\theta^{\bar\gamma}\wedge\big(\eta_2+i\eta_3\big) 
+ \mathcal P^\ast\wedge\eta_1\wedge\big(\eta_2+i\eta_3\big) + \overline{\mathcal P^\ast}\wedge\eta_1\wedge\big(\eta_2-i\eta_3\big)\\
+i\mathcal R^\ast\wedge\big(\eta_2+i\eta_3\big) \wedge\big(\eta_2-i\eta_3\big) 
\ =\ 0;
 \end{multline}
  \begin{multline}\label{d2psi_2}
 d^2\big(\psi_2+i\psi_3\big)\  = \ 4i\pi^{\bar\mu}_{\beta}\, \mathcal M^\ast_{\bar\mu\bar\gamma}\wedge\theta^{\beta}\wedge\theta^{\bar\gamma}
 +4i\pi^{\bar\mu}_{\beta}\,\mathcal C^\ast_{\bar\mu}\wedge\theta^{\beta}\wedge\eta_1
  -4\mathcal H^\ast_{\bar\gamma}\wedge\theta^{\bar\gamma}\wedge\eta_1
  -4\mathcal C^\ast_{\bar\gamma}\wedge\theta^{\bar\gamma}\wedge\big(\eta_2+i\eta_3\big)\\
 -4i\pi^{\bar\mu}_{\beta}\,\mathcal H^\ast_{\bar\mu}\wedge\theta^{\beta}\wedge\big(\eta_2-i\eta_3\big)
 -i\mathcal R^\ast\wedge\eta_1\wedge\big(\eta_2+i\eta_3\big) + \overline{\mathcal Q^\ast}\wedge\eta_1\wedge\big(\eta_2-i\eta_3\big)
 \\ 
-\overline{\mathcal P^\ast}\wedge\big(\eta_2+i\eta_3\big) \wedge\big(\eta_2-i\eta_3\big) 
\ =\ 0.
 \end{multline}
 \end{prop} 
 
 \begin{prop}[The secondary derivatives]\label{secondary}\label{sec_der}
 On $P_1$, there exist unique, globally defined, complex valued funct	ions
 \begin{equation}
 \begin{split}
 \mathcal A_{\alpha\beta\gamma\delta\epsilon},\ & \mathcal B_{\alpha\beta\gamma\delta},\  \mathcal C_{\alpha\beta\gamma\delta} ,\ \mathcal  D_{\alpha\beta\gamma},\ \mathcal  E_{\alpha\beta\gamma},
 \ \mathcal  F_{\alpha\beta\gamma},\ \mathcal  G_{\alpha\beta},\ \mathcal  X_{\alpha\beta},\ \mathcal  Y_{\alpha\beta},\ \mathcal  Z_{\alpha\beta}, \\
  &\ (\mathcal N_1)_\alpha,\ (\mathcal N_2)_\alpha,\ (\mathcal N_3)_\alpha,\ (\mathcal N_4)_\alpha,\ (\mathcal N_5)_\alpha,\ \mathcal U_s,\  \mathcal W_s
\end{split}
\end{equation}
 so that:
 
 {\bf (I)} Each of the arrays  $\{\mathcal A_{\alpha\beta\gamma\delta\epsilon}\}$, $\{ \mathcal B_{\alpha\beta\gamma\delta}\}$, $\{ \mathcal C_{\alpha\beta\gamma\delta}\}$, $\{ \mathcal D_{\alpha\beta\gamma}\}$, $\{ \mathcal E_{\alpha\beta\gamma}\}$, $\{ \mathcal F_{\alpha\beta\gamma}\}$, $\{ \mathcal G_{\alpha\beta\gamma}\}$, $\{ \mathcal X_{\alpha\beta}\}$, $\{ \mathcal Y_{\alpha\beta}\}$, $\{ \mathcal Z_{\alpha\beta}\}$ is totally symmetric in its indices.
 
{\bf  (II)} We have  
 \begin{equation}
 \begin{split}
d\mathcal S_{\alpha\beta\gamma\delta}&\ =\ 
\tilde{\mathcal S}^\ast_{\alpha\beta\gamma\delta} 
+\mathcal A_{\alpha\beta\gamma\delta\epsilon}\,\theta^{\epsilon}-\pi^{\sigma}_{\bar\epsilon}(\mathfrak j\mathcal A)_{\alpha\beta\gamma\delta\sigma}\,\theta^{\bar\epsilon}
 +\Big(\mathcal B_{\alpha\beta\gamma\delta}+(\mathfrak j \mathcal B)_{\alpha\beta\gamma\delta}\Big)\eta_1 + i\mathcal C_{\alpha\beta\gamma\delta}\big(\eta_2+i\eta_3\big)\\
&\qquad\qquad\qquad\qquad\qquad\qquad\qquad\qquad\qquad\qquad\qquad\qquad\qquad - i(\mathfrak j\mathcal C)_{\alpha\beta\gamma\delta}\big(\eta_2-i\eta_3\big)\\
d\mathcal V_{\alpha\beta\gamma}&\ =\ 
\tilde{\mathcal V}^\ast_{\alpha\beta\gamma} + \mathcal C_{\alpha\beta\gamma\epsilon}\,\theta^{\epsilon}+\pi^{\sigma}_{\bar\epsilon}\,\mathcal B_{\alpha\beta\gamma\sigma}\,\theta^{\bar\epsilon}
 +\mathcal D_{\alpha\beta\gamma}\eta_1 + \mathcal E_{\alpha\beta\gamma}\big(\eta_2+i\eta_3\big)+\mathcal F_{\alpha\beta\gamma}\big(\eta_2-i\eta_3\big)\\
 d\mathcal L_{\alpha\beta}&\ =\ \tilde{\mathcal L}^\ast_{\alpha\beta}-
(\mathfrak j\mathcal F)_{\alpha\beta\epsilon}\,\theta^{\epsilon}\!-\!\pi^{\sigma}_{\bar\epsilon}\, \mathcal F_{\alpha\beta\sigma}\,\theta^{\bar\epsilon}
 +i\Big((\mathfrak j\mathcal Z)_{\alpha\beta}-\mathcal Z_{\alpha\beta}\Big)\eta_1 +i \mathcal G_{\alpha\beta}\big(\eta_2+i\eta_3\big)-i(\mathfrak j\mathcal G)_{\alpha\beta}\big(\eta_2-i\eta_3\big)\\
  d\mathcal M_{\alpha\beta}&\ =\ \tilde{\mathcal M}^\ast_{\alpha\beta}-\mathcal E_{\alpha\beta\epsilon}\,\theta^{\epsilon}+\pi^{\sigma}_{\bar\epsilon}\Big((\mathfrak j\mathcal F)_{\alpha\beta\sigma}-i\mathcal D_{\alpha\beta\sigma}\Big)\,\theta^{\bar\epsilon}
 +\mathcal X_{\alpha\beta}\eta_1 + \mathcal Y_{\alpha\beta}\big(\eta_2+i\eta_3\big)+\mathcal Z_{\alpha\beta}\big(\eta_2-i\eta_3\big)\\
  d\mathcal C_{\alpha}&\ =\ \tilde{\mathcal C}^\ast_{\alpha}+\mathcal G_{\alpha\epsilon}\,\theta^{\epsilon}-i\pi^{\sigma}_{\bar\epsilon}\mathcal Z_{\alpha\sigma}\,\theta^{\bar\epsilon}
 +(\mathcal N_1)_{\alpha}\eta_1 + (\mathcal N_2)_{\alpha}\big(\eta_2+i\eta_3\big)+(\mathcal N_3)_{\alpha}\big(\eta_2-i\eta_3\big)\\
  d\mathcal H_{\alpha}&\ =\ \tilde{\mathcal H}^\ast_{\alpha}-\mathcal Y_{\alpha\epsilon}\,\theta^{\epsilon}+i\pi^{\sigma}_{\bar\epsilon}\big(\mathcal G_{\alpha\sigma}-\mathcal X_{\alpha\sigma}\big)\theta^{\bar\epsilon}
 +(\mathcal N_4)_{\alpha}\eta_1 + (\mathcal N_5)_{\alpha}\big(\eta_2+i\eta_3\big)\\
 &\qquad\qquad\qquad\qquad\qquad\qquad\qquad\qquad\qquad\qquad\qquad+\Big((\mathcal N_1)_{\alpha}+i\pi^{\bar\sigma}_{\alpha}(\mathcal N_3)_{\bar\sigma}\Big)\big(\eta_2-i\eta_3\big)\\
  d\mathcal R&\ =\ \tilde{\mathcal R}^\ast + 4\pi^{\bar\sigma}_{\epsilon}(\mathcal N_3)_{\bar\sigma}\,\theta^{\epsilon}+4\pi^{\sigma}_{\bar\epsilon}(\mathcal N_3)_{\sigma}\,\theta^{\bar\epsilon}
 +i\big(\mathcal U_3-\overline{\mathcal U}_3\big)\eta_1 -i\big(\mathcal U_1+\mathcal W_3\big)\big(\eta_2+i\eta_3\big)\\
 &\qquad\qquad\qquad\qquad\qquad\qquad\qquad\qquad\qquad\qquad\qquad\qquad+i\big(\overline{\mathcal U}_1+\overline{\mathcal W}_3\big)\big(\eta_2-i\eta_3\big)\\
  d\mathcal P&\ =\ \tilde{\mathcal P}^\ast-4(\mathcal N_2)_{\epsilon}\,\theta^{\epsilon}-4\Big((\mathcal N_3)_{\bar\epsilon}+i\pi^{\sigma}_{\bar\epsilon}(\mathcal N_1)_{\sigma}\Big)\,\theta^{\bar\epsilon}
 +\mathcal U_1\eta_1 + \mathcal U_2\big(\eta_2+i\eta_3\big)+\mathcal U_3\big(\eta_2-i\eta_3\big)\\
  d\mathcal Q&\ =\ \tilde{\mathcal Q}^\ast+4(\mathcal N_5)_{\epsilon}\,\theta^{\epsilon}+4i\pi^{\sigma}_{\bar\epsilon}\Big((\mathcal N_2)_{\sigma}+(\mathcal N_4)_{\sigma}\Big)\,\theta^{\bar\epsilon}
 +\mathcal W_1\eta_1 + \mathcal W_2\big(\eta_2+i\eta_3\big)+\mathcal W_3\big(\eta_2-i\eta_3\big)
 \end{split}
 \end{equation}
  \end{prop}

 \begin{proof}[Proof of Propositions \ref{curvature}, \ref{bianchi} and \ref{secondary}] A differentiation of the first equation of  \eqref{str-eq-con}, after some calculations using \eqref{str-eq-con}, yields
 \begin{multline*}
 0\ =\ d^2\theta^\alpha\ = \ \pi^{\alpha\sigma}\Big(-d\Gamma_{\sigma\beta}-\pi^{\nu\tau}\Gamma_{\sigma\nu}\wedge\Gamma_{\tau\beta}+ 2\pi^{\bar\tau}_{\sigma}(\phi_\beta\wedge\theta_{\bar\tau}-\phi_{\bar\tau}\wedge\theta_\beta)+2\pi^{\bar\tau}_{\beta}(\phi_\sigma\wedge\theta_{\bar\tau}-\phi_{\bar\tau}\wedge\theta_\sigma)\Big)\wedge\theta^\beta \\
 +i\Big(-d\phi^{\alpha}-\pi^{\alpha\sigma}\Gamma_{\sigma\gamma}\wedge\phi^{\gamma}+\frac{1}{2}(\varphi_0-i\varphi_1)\wedge\phi^\alpha 
 -\frac{1}{2}\pi^{\alpha}_{\bar\gamma}(\varphi_2+i\varphi_3)\wedge\phi^{\bar\gamma}\\ +\frac{i}{2}\,\psi_1\wedge\theta^\alpha
  +\frac{1}{2}\,\pi^{\alpha}_{\bar\gamma}(\psi_2+i\psi_3)\wedge\theta^{\bar\gamma}\Big)\wedge\eta_1\\
  + \pi^{\alpha}_{\bar\beta}\Big(-d\phi^{\bar\beta}-\pi^{\bar\beta\bar\sigma}\Gamma_{\bar\sigma\bar\gamma}\wedge\phi^{\bar\gamma}+\frac{1}{2}(\varphi_0+i\varphi_1)\wedge\phi^{\bar\beta} 
 -\frac{1}{2}\pi^{\bar\beta}_{\gamma}(\varphi_2-i\varphi_3)\wedge\phi^{\gamma}\\ 
 -\frac{i}{2}\,\psi_1\wedge\theta^{\bar\beta}
  +\frac{1}{2}\,\pi^{\bar\beta}_{\gamma}(\psi_2-i\psi_3)\wedge\theta^{\gamma}\Big)\wedge(\eta_2+i\eta_3)
 \end{multline*}
 
 If we set
 \begin{equation}\label{set-X-Y}
 \begin{split}
 X_{\alpha\beta}=&d\Gamma_{\alpha\beta}+\pi^{\nu\tau}\Gamma_{\alpha\nu}\wedge\Gamma_{\tau\beta}- 2\pi^{\bar\tau}_{\alpha}(\phi_\beta\wedge\theta_{\bar\tau}+\phi_{\bar\tau}\wedge\theta_\beta)-2\pi^{\bar\tau}_{\beta}(\phi_\alpha\wedge\theta_{\bar\tau}-\phi_{\bar\tau}\wedge\theta_\alpha)\\
 Y^{\alpha}=&d\phi^{\alpha}+\pi^{\alpha\sigma}\Gamma_{\sigma\gamma}\wedge\phi^{\gamma}-\frac{1}{2}(\varphi_0-i\varphi_1)\wedge\phi^\alpha 
 +\frac{1}{2}\pi^{\alpha}_{\bar\gamma}(\varphi_2+i\varphi_3)\wedge\phi^{\bar\gamma}\\ 
 &-\frac{i}{2}\,\psi_1\wedge\theta^\alpha
  -\frac{1}{2}\,\pi^{\alpha}_{\bar\gamma}(\psi_2+i\psi_3)\wedge\theta^{\bar\gamma},
 \end{split}
 \end{equation}
 the above equation reads as
 \begin{equation}\label{simp-eq-d2theta}
 \pi^{\alpha\sigma}X_{\sigma\beta}\wedge\theta^\beta + iY^{\alpha}\wedge\eta_1+\pi^{\alpha}_{\bar\beta}Y^{\bar\beta}\wedge(\eta_2+i\eta_3)\ = \ 0.
 \end{equation}
 Consequently, there exist one-forms  $A_{\alpha\beta\gamma}$ so that 
 \begin{equation*}
 X_{\alpha\beta}\equiv A_{\alpha\beta\gamma}\wedge\theta^{\gamma}\mod{\{\eta_1,\eta_2+i\eta_3\}}.
 \end{equation*}

A small calculation using  \eqref{set-X-Y} shows that
\begin{equation}\label{prop-X_ab}
X_{\alpha\beta}=X_{\beta\alpha},\qquad (\mathfrak jX)_{\alpha\beta}=X_{\alpha\beta},
\end{equation}
and therefore,
\begin{equation*}
 A_{\alpha\beta\gamma}\wedge\theta^\gamma\ \equiv\ \pi^{\bar\sigma}_{\alpha}\,\pi^{\bar\tau}_{\beta}\,A_{\bar\sigma\bar\tau\bar\gamma}\wedge\theta^{\bar\gamma} \mod{\{\eta_1,\eta_2+i\eta_3,\eta_2-i\eta_3\}}.
 \end{equation*}
It follows that $A_{\alpha\beta\gamma}\equiv 0$ modulo $\{\theta^\gamma,\theta^{\bar\gamma},\eta_1,\eta_2+i\eta_3,\eta_2-i\eta_3\}$ and thus, there exist functions $A_{\alpha\beta\gamma\delta}, A_{\alpha\beta\gamma\bar\delta}$ so that
\begin{equation*}
X_{\alpha\beta}\equiv A_{\alpha\beta\gamma\delta}\,\theta^{\gamma}\wedge\theta^{\delta}+A_{\alpha\beta\gamma\bar\delta}\,\theta^{\gamma}\wedge\theta^{\bar\delta} \mod{\{\eta_1,\eta_2+i\eta_3,\eta_2-i\eta_3\}}
\end{equation*}
 and
 \begin{equation}
 \begin{cases}\label{A_abcd prop}
 A_{\alpha\beta\gamma\delta}=A_{\beta\alpha\gamma\delta}\\
 A_{\alpha\beta\gamma\delta}=-A_{\alpha\beta\delta\gamma}\\
 A_{\alpha\beta\gamma\bar\delta}=A_{\beta\alpha\gamma\bar\delta}.
 \end{cases}
 \end{equation}
 
 Substituting back into \eqref{simp-eq-d2theta} gives $A_{\alpha\beta\gamma\bar\delta}\,\theta^\gamma\wedge\theta^{\bar\delta}\wedge\theta^\beta=0$ and therefore, the array $\{A_{\alpha\beta\gamma\bar\delta}\}$ is totally symmetric in the indices $\alpha,\beta,\gamma$. We have also
 \begin{equation*}
 \begin{split}
 0=(\mathfrak jX)_{\alpha\beta}-X_{\alpha\beta}=\pi^{\bar\sigma}_{\alpha}\,\pi^{\bar\tau}_\beta\, A_{\bar\sigma\bar\tau\bar\gamma\bar\delta}\,\theta^{\bar\gamma}\wedge\theta^{\bar\delta}
 &+\pi^{\bar\sigma}_{\alpha}\,\pi^{\bar\tau}_\beta\, A_{\bar\sigma\bar\tau\bar\gamma\delta}\,\theta^{\bar\gamma}\wedge\theta^{\delta}\\&-A_{\alpha\beta\gamma\delta}\theta^{\gamma}\wedge\theta^{\delta}
 -A_{\alpha\beta\gamma\bar\delta}\theta^{\gamma}\wedge\theta^{\bar\delta}+ \dots
 \end{split}
 \end{equation*} 
 where the omitted terms are vanishing modulo $\{\eta_1,\eta_2+i\eta_3,\eta_2-i\eta_3\}$.
 Therefore,  $A_{\alpha\beta\gamma\delta}=0$ (in view of the second line of \eqref{A_abcd prop}) and 
 \begin{equation}\label{symmetry-S}
 A_{\alpha\beta\gamma\bar\delta}=-\pi^{\bar\sigma}_{\alpha}\,\pi^{\bar\tau}_\beta\, A_{\bar\sigma\bar\tau\bar\delta\gamma}.
 \end{equation}
 
 Let us define
 \begin{equation}
 \mathcal S_{\alpha\beta\gamma\delta}\overset{def}=-\pi^{\bar\sigma}_\delta \,A_{\alpha\beta\gamma\bar\sigma}.
 \end{equation}
 Then, since the array $\{A_{\alpha\beta\gamma\bar\delta}\}$ is totally symmetric in the indices $\alpha,\beta,\gamma$, equation \eqref{symmetry-S} implies that 
 the array $\{\mathcal S_{\alpha\beta\gamma\delta}\}$ is totally symmetric in all of its indices and $(\mathfrak j\mathcal S)_{\alpha\beta\gamma\delta}=\mathcal S_{\alpha\beta\gamma\delta}$.
 Furthermore, we have
 \begin{equation}\label{X_ab-shape}
 X_{\alpha\beta}\ \equiv\ \pi^{\sigma}_{\bar\delta}\,\mathcal S_{\alpha\beta\gamma\sigma}\,\theta^{\gamma}\wedge\theta^{\bar\delta}+(B_1)_{\alpha\beta}\wedge\eta_1 
 + (B_2)_{\alpha\beta}\wedge(\eta_2+i\eta_3)+(B_3)_{\alpha\beta}\wedge(\eta_2-i\eta_3)
 \end{equation}
 for some appropriate one-forms $(B_s)_{\alpha\beta}$. 
 
 On the other hand, \eqref{simp-eq-d2theta} implies that the two-forms $Y^\alpha$ and $\pi^{\alpha}_{\bar\beta}Y^{\bar\beta}$ are vanishing modulo $\{\theta^\gamma,\eta_1,\eta_2+i\eta_3\}$ and thus, the same holds true for $Y^{\bar\alpha}$. Therefore, there exist functions $C_{\bar\alpha\beta\bar\gamma}$ for which 
\begin{equation}\label{Y_a-hor}
Y^{\alpha}\ \equiv\ C^{\alpha}_{\dt\beta\bar\gamma}\,\theta^{\beta}\wedge\theta^{\bar\gamma} \mod{\eta_s}.
\end{equation}  
 
 We substitute \eqref{X_ab-shape} and \eqref{Y_a-hor} into \eqref{simp-eq-d2theta} to obtain, modulo $\{\eta_1\wedge\eta_2,\eta_2\wedge\eta_3,\eta_3\wedge\eta_1\}$,
 \begin{multline}\label{0=B1-B2-B3}
 \pi^{\alpha\sigma}(B_1)_{\sigma\beta}\wedge\eta_1\wedge\theta^\beta +\pi^{\alpha\sigma}(B_2)_{\sigma\beta}\wedge(\eta_2+i\eta_3)\wedge\theta^\beta +\pi^{\alpha\sigma}(B_3)_{\sigma\beta}\wedge(\eta_2-i\eta_3)\wedge\theta^\beta\\
 +iC^\alpha_{\dt\beta\bar\gamma}\theta^{\beta}\wedge\theta^{\bar\gamma}\wedge\eta_1+\pi^{\alpha}_{\bar\sigma}C^{\bar\sigma}_{\dt\bar\gamma\beta}\theta^{\bar\gamma}\wedge\theta^{\beta}\wedge(\eta_2+i\eta_3)\ \equiv\ 0.
 \end{multline}
 
 Consequently, 
 \begin{equation*}
 \begin{cases}
 (B_1)_{\alpha\beta}\equiv 0 \mod{\{\theta^\gamma,\theta^{\bar\gamma},\eta_s\}}\\
 (B_3)_{\alpha\beta}\equiv 0 \mod{\{\theta^\gamma,\eta_s\}}
 \end{cases}
 \end{equation*}
 
 By \eqref{prop-X_ab}, we have that $(B_2)_{\alpha\beta}=\pi^{\bar\sigma}_\alpha\,\pi^{\bar\tau}_\beta\, (B_3)_{\bar\sigma\tau}$ and therefore, 
  there exist functions $(B_1)_{\alpha\beta\gamma}$, $(B_1)_{\alpha\beta\bar\gamma}$ and $(B_3)_{\alpha\beta\gamma}$ so that
 \begin{equation*}
 \begin{cases}
 (B_1)_{\alpha\beta}\equiv (B_1)_{\alpha\beta\gamma}\,\theta^\gamma + (B_1)_{\alpha\beta\bar\gamma}\,\theta^{\bar\gamma} \\
 (B_2)_{\alpha\beta}\equiv \pi^{\bar\sigma}_\alpha\,\pi^{\bar\tau}_\beta (B_3)_{\bar\sigma\bar\tau\bar\gamma}\,\theta^{\bar\gamma}\qquad \mod{\eta_s}\\
 (B_3)_{\alpha\beta}\equiv (B_3)_{\alpha\beta\gamma}\,\theta^{\gamma}.
 \end{cases}
 \end{equation*}
 
 On the account of \eqref{prop-X_ab} and  \eqref{X_ab-shape}, we obtain that the arrays $\{(B_1)_{\alpha\beta\gamma}\}$, $\{(B_1)_{\alpha\beta\bar\gamma}\}$, $\{(B_3)_{\alpha\beta\gamma}\}$ are symmetric in $\alpha$, $\beta$ and satisfy 
 \begin{equation}
 (B_1)_{\alpha\beta\bar\gamma}=\pi^{\bar\sigma}_\alpha\,\pi^{\bar\tau}_\beta (B_1)_{\bar\sigma\bar\tau\bar\gamma}.
 \end{equation} 
 
 Substituting back into \eqref{0=B1-B2-B3} yields
 \begin{multline*}
 \pi^{\alpha\sigma}(B_1)_{\sigma\beta\gamma}\,\theta^{\beta}\wedge\theta^{\gamma}\wedge\eta_1
 +g^{\alpha\bar\sigma}\Big(-\pi^{\bar\tau}_\beta(B_1)_{\bar\sigma\bar\tau\bar\gamma}+iC_{\bar\sigma\beta\bar\gamma}\Big)\theta^\beta\wedge\theta^{\bar\gamma}\wedge\eta_1\\
 +\pi^{\alpha\sigma}\Big(\pi^{\bar\tau}_\sigma\,\pi^{\bar\mu}_\beta\,(B_3)_{\bar\tau\bar\mu\bar\gamma}+C_{\sigma\bar\gamma\beta}\Big)\,\theta^\beta\wedge\theta^{\bar\gamma}\wedge(\eta_2+i\eta_3)
 +\pi^{\alpha\sigma}\,(B_3)_{\sigma\beta\gamma}\,\theta^\beta\wedge\theta^\gamma\wedge(\eta_2-i\eta_3)\ =\ 0.
 \end{multline*}
 It follows that the arrays  $\{(B_1)_{\alpha\beta\gamma}\}$ and $\{(B_3)_{\alpha\beta\gamma}\}$ are totally symmetric in their indices, and we have the identities 
 \begin{equation*}
-\pi^{\bar\tau}_\beta(B_1)_{\bar\sigma\bar\tau\bar\gamma}+iC_{\bar\sigma\beta\bar\gamma}=0,\qquad
\pi^{\bar\tau}_\sigma\,\pi^{\bar\mu}_\beta\,(B_3)_{\bar\tau\bar\mu\bar\gamma}+C_{\sigma\bar\gamma\beta}=0.
\end{equation*}
Hence,
\begin{equation*}
C_{\bar\alpha\beta\bar\gamma}=-i\pi^{\bar\sigma}_{\beta}\,(B_1)_{\bar\alpha\bar\sigma\bar\gamma}
\end{equation*}
and also	
\begin{equation*}
(B_3)_{\alpha\beta\gamma}=-\pi^{\bar\sigma}_\alpha\,\pi^{\bar\tau}_{\beta}C_{\bar\sigma\gamma\bar\tau}=
i\pi^{\bar\sigma}_\alpha\,\pi^{\bar\tau}_{\beta}\,\pi^{\bar\mu}_{\gamma}(B_1)_{\bar\sigma\bar\tau\bar\mu}.
\end{equation*}

Setting 
\begin{equation*}
\mathcal V_{\alpha\beta\gamma}\overset{def}{=}(B_1)_{\alpha\beta\gamma},
\end{equation*}
we obtain 
 \begin{equation*}
 \begin{cases}
 (B_1)_{\alpha\beta}\ \equiv\ \mathcal V_{\alpha\beta\gamma}\,\theta^\gamma + \pi^{\bar\sigma}_{\alpha}\,\pi^{\bar\tau}_\beta\mathcal V_{\bar\sigma\bar\tau\bar\gamma}\,\theta^{\bar\gamma}\\
 (B_2)_{\alpha\beta}\ \equiv\ -i \pi^{\sigma}_{\bar\gamma}\mathcal V_{\alpha\beta\sigma}\,\theta^{\bar\gamma}\\
 (B_3)_{\alpha\beta}\ \equiv\  i (\mathfrak j\mathcal V)_{\alpha\beta\gamma}\,\theta^{\gamma}\\
 Y^{\alpha}\ \equiv\ -i\pi^{\bar\sigma}_\beta\,\mathcal V^{\alpha}_{\dt\bar\sigma\bar\gamma}\,\theta^{\beta}\wedge\theta^{\bar\gamma}.
 \end{cases}
 \quad \mod{\eta_s}
 \end{equation*}
Thus, for some appropriate functions $(A_s)_{\alpha\beta}$, we have 
\begin{equation}
 \begin{split}\label{eq-X_ab}
  X_{\alpha\beta}=&\pi^{\sigma}_{\bar\delta}\,\mathcal S_{\alpha\beta\gamma\sigma}\,\theta^{\gamma}\wedge\theta^{\bar\delta}+\Big(\mathcal V_{\alpha\beta\gamma}\,\theta^\gamma
 +\pi^{\bar\sigma}_{\alpha}\,\pi^{\bar\tau}_{\beta}\,\mathcal V_{\bar\sigma\bar\tau\bar\gamma}\,\theta^{\bar\gamma}\Big)\wedge\eta_1\\
 &-i\pi^{\sigma}_{\bar\gamma}\,\mathcal V_{\alpha\beta\sigma}\,\theta^{\bar\gamma}\wedge(\eta_2+i\eta_3)+i(\mathfrak j\mathcal V)_{\alpha\beta\gamma}\,\theta^{\gamma}\wedge(\eta_2-i\eta_3)\\
 &+(A_1)_{\alpha\beta}\,\eta_1\wedge(\eta_2+i\eta_3)+(A_2)_{\alpha\beta}\,\eta_1\wedge(\eta_2-i\eta_3)+(A_3)_{\alpha\beta}(\eta_2+i\eta_3)\wedge(\eta_2-i\eta_3).
 \end{split} 
 \end{equation}
Clearly, by \eqref{prop-X_ab}, $\{(A_s)_{\alpha\beta}\}$ are symmetric in the $\alpha,\beta$ and satisfy
\begin{equation}\label{prop-A_sab}
(A_2)_{\alpha\beta}=\pi^{\bar\sigma}_\alpha\,\pi^{\bar\tau}_\beta\,(A_1)_{\bar\sigma\bar\tau},\qquad (A_3)_{\alpha\beta}=-\pi^{\bar\sigma}_\alpha\,\pi^{\bar\tau}_\beta\,(A_3)_{\bar\sigma\bar\tau}.
\end{equation}

 Using one more time the argument that both $Y^\alpha$ and $Y^{\bar\alpha}$ are vanishing modulo $\{\theta^\gamma,\eta_1,\eta_2+i\eta_3\}$, we deduce that there exist functions $C_\alpha$,  $D_{\alpha\beta}$, $F_{\alpha\bar\beta}$, and one-forms $E_\alpha$ so that 
 \begin{equation}\label{eq-Y_a}
 \begin{split}
 Y^{\alpha}=- i\pi^{\bar\sigma}_\beta\,\mathcal V^{\alpha}_{\dt\bar\sigma\bar\gamma}\,\theta^{\beta}\wedge\theta^\gamma+D^{\alpha}_{\bar\beta}\,\theta^{\bar\beta}\wedge(\eta_2+i\eta_3)
 &+F^\alpha_\beta\,\theta^\beta\wedge(\eta_2-i\eta_3) \\&+ C^\alpha(\eta_2+i\eta_3)\wedge(\eta_2-i\eta_3)+E^\alpha\wedge\eta_1.
 \end{split}
 \end{equation}
 
 Substituting \eqref{eq-X_ab} and \eqref{eq-Y_a} back into \eqref{simp-eq-d2theta} gives
 \begin{multline}\label{subs-X_ab-Y_a}
 \pi^{\alpha\sigma}(A_1)_{\sigma\beta}\,\theta^\beta\wedge\eta_1\wedge(\eta_2+i\eta_3)
 +\pi^{\alpha\sigma}(A_2)_{\sigma\beta}\,\theta^\beta\wedge\eta_1\wedge(\eta_2-i\eta_3)\\
 +\pi^{\alpha\sigma}(A_3)_{\sigma\beta}\,\theta^\beta\wedge(\eta_2+i\eta_3)\wedge(\eta_2-i\eta_3)-iF^{\alpha}_{\beta}\,\theta^{\beta}\wedge\eta_1\wedge(\eta_2-i\eta_3)\\
 -iD^{\alpha}_{\bar\beta}\,\theta^{\bar\beta}\wedge\eta_1\wedge(\eta_2+i\eta_3)-\pi^{\alpha}_{\bar\sigma}\,D^{\bar\sigma}_\beta\,\theta^\beta\wedge(\eta_2+i\eta_3)\wedge(\eta_2-i\eta_3)\\
 +iC^\alpha \eta_1\wedge(\eta_2+i\eta_3)\wedge(\eta_2-i\eta_3)+\pi^{\alpha}_{\bar\beta}\,E^{\bar\beta}\wedge\eta_1\wedge(\eta_2+i\eta_3)\ =\ 0.
 \end{multline}
 
 By considering the coefficients of $\theta^\beta\wedge\eta_1\wedge(\eta_2-i\eta_3)$ and $\theta^\beta\wedge(\eta_2+i\eta_3)\wedge(\eta_2-i\eta_3)$, we obtain
 \begin{equation}\label{F_ab-D_ab}
 \begin{cases}
 F^{\alpha}_\beta=-i\pi^{\alpha\sigma}(A_2)_{\sigma\beta}\overset{\eqref{prop-A_sab}}{=} i\pi^{\bar\sigma}_\beta\,(A_1)^{\alpha}_{\bar\sigma}\\
 D_{\alpha\beta}=-(A_3)_{\alpha\beta}
 \end{cases}
 \end{equation}
 and  thus, \eqref{subs-X_ab-Y_a} simplifies to
 \begin{multline}\label{subs-X_ab-Y_a-2}
 \pi^{\alpha\sigma}(A_1)_{\sigma\beta}\,\theta^\beta\wedge\eta_1\wedge(\eta_2+i\eta_3)
 -iD^{\alpha}_{\bar\beta}\,\theta^{\bar\beta}\wedge\eta_1\wedge(\eta_2+i\eta_3)\\
 +iC^\alpha \eta_1\wedge(\eta_2+i\eta_3)\wedge(\eta_2-i\eta_3)+\pi^{\alpha}_{\bar\beta}\,E^{\bar\beta}\wedge\eta_1\wedge(\eta_2+i\eta_3)\ =\ 0.
 \end{multline}
 
 By \eqref{subs-X_ab-Y_a-2}, we obtain  
 \begin{equation}\label{E_a-shape}
 E^\alpha=E^\alpha_\beta\,\theta^\beta +E^\alpha_{\bar\beta}\,\theta^{\bar\beta}+(E_1)^\alpha\eta_1+(E_2)^\alpha(\eta_2+i\eta_3)+(E_3)^\alpha(\eta_2-i\eta_3),
 \end{equation}
 where $E^\alpha_\beta$, $E^\alpha_{\bar\beta}$ and $(E_s)^\alpha$ are some appropriate coefficients.
 Furthermore, by substituting \eqref{E_a-shape} into \eqref{subs-X_ab-Y_a-2}, we get
 \begin{equation}\label{E's-eq}
 \begin{cases}
 E_{\alpha\beta}=(A_1)_{\alpha\beta}\\
 E^{\alpha}_{\beta}=-i\pi^{\alpha\sigma}\,D_{\sigma\beta}\overset{\eqref{F_ab-D_ab}}{=}i\pi^{\alpha\sigma}(A_3)_{\sigma\beta}\\
 (E_2)^\alpha=-i\pi^{\alpha}_{\bar\sigma}\,C^{\bar\sigma}.
 \end{cases}
 \end{equation}
 
 Setting
 \begin{equation}\label{def-LMCH}
 \begin{cases}
\mathcal L_{\alpha\beta} \overset{def}{=}i(A_3)_{\alpha\beta}\\
 \mathcal M_{\alpha\beta}\overset{def}{=}(A_1)_{\alpha\beta}\\
 \mathcal C^\alpha \overset{def}{=} C^\alpha\\
\mathcal H^\alpha \overset{def}{=} -(E_3)^\alpha,
 \end{cases}
 \end{equation}
we obtain, by \eqref{set-X-Y}, \eqref{eq-X_ab}, \eqref{eq-Y_a}, \eqref{F_ab-D_ab},  \eqref{F_ab-D_ab}, \eqref{E's-eq} and \eqref{def-LMCH}, 
the relations \eqref{dGamma_ab} and \eqref{dphi_a}. 

We proceed by differentiating \eqref{dGamma_ab} and \eqref{dphi_a} one more time. After some rather long but straightforward calculations, we obtain  
\begin{multline}\label{d2Gamma_ab1}
 0\ =\ d^2\Gamma_{\alpha\beta}\  = \ \pi^{\sigma}_{\bar\delta}\bigg[d{\mathcal S}_{\alpha\beta\gamma\sigma}-\pi^{\tau\nu}\Gamma_{\nu\alpha}\,\mathcal S_{\tau\beta\gamma\sigma}
 -\pi^{\tau\nu}\Gamma_{\nu\beta}\,\mathcal S_{\alpha\tau\gamma\sigma}-\pi^{\tau\nu}\Gamma_{\nu\gamma}\,\mathcal S_{\alpha\beta\tau\sigma}
 -\pi^{\tau\nu}\Gamma_{\nu\sigma}\,\mathcal S_{\alpha\beta\gamma\tau}\\
 -\varphi_0\, \mathcal S_{\alpha\beta\gamma\sigma}
 -2i\big(\pi_{\alpha\tau}\,\mathcal V_{\sigma\beta\gamma}+\pi_{\beta\tau}\,\mathcal V_{\alpha\sigma\gamma}+\pi_{\gamma\tau}\,\mathcal V_{\alpha\beta\sigma}
 +\pi_{\sigma\tau}\,\mathcal V_{\alpha\beta\gamma}\big)\theta^{\tau}
 -2i\Big(g_{\alpha\bar\tau}\,(\mathfrak j\mathcal V)_{\sigma\beta\gamma}\\
 +g_{\beta\bar\tau}\,(\mathfrak j\mathcal V)_{\alpha\sigma\gamma}+g_{\gamma\bar\tau}\,(\mathfrak j\mathcal V)_{\alpha\beta\sigma}
 +g_{\sigma\bar\tau}\,(\mathfrak j\mathcal V)_{\alpha\beta\gamma}\Big)\theta^{\bar\tau}\bigg]\wedge\theta^\gamma\wedge\theta^{\bar\delta} 
 \end{multline}
\begin{multline*}
 \qquad\qquad+\ \bigg[d{\mathcal V}_{\alpha\beta\gamma}-\pi^{\tau\nu}\Gamma_{\nu\alpha}\,\mathcal V_{\tau\beta\gamma}
 -\pi^{\tau\nu}\Gamma_{\nu\beta}\,\mathcal V_{\alpha\tau\gamma}-\pi^{\tau\nu}\Gamma_{\nu\gamma}\,\mathcal V_{\alpha\beta\tau}
 +i\pi^{\sigma}_{\bar\tau}\,\phi^{\bar\tau}\,\mathcal S_{\alpha\beta\gamma\sigma}-\frac{1}{2}\big(3\varphi_0+i\varphi_1\big)\, \mathcal V_{\alpha\beta\gamma}\\
 +\frac{1}{2}\big(\varphi_2-i\varphi_3\big)\, (\mathfrak j\mathcal V)_{\alpha\beta\gamma}
 +2\big(\pi_{\alpha\tau}\,\mathcal M_{\beta\gamma}+\pi_{\beta\tau}\,\mathcal M_{\alpha\gamma}+\pi_{\gamma\tau}\,\mathcal M_{\alpha\beta}\big)\theta^{\tau}\\
 +2\big(g_{\alpha\bar\tau}\, \mathcal L_{\beta\gamma}
 +g_{\beta\bar\tau}\,\mathcal L_{\alpha\gamma}+g_{\gamma\bar\tau}\,\mathcal L_{\alpha\beta}\big)\theta^{\bar\tau}\bigg]\wedge\theta^\gamma\wedge\eta_1 
 \end{multline*}
\begin{multline*}
 \qquad\qquad+\ \pi^{\bar\mu}_\alpha\,\pi^{\bar\nu}_\beta\,\bigg[d{\mathcal V}_{\bar\mu\bar\nu\bar\gamma}-\pi^{\bar\tau\bar\sigma}\Gamma_{\bar\sigma\bar\mu}\,\mathcal V_{\bar\tau\bar\nu\bar\gamma}
 -\pi^{\bar\tau\bar\sigma}\Gamma_{\bar\sigma\nu}\,\mathcal V_{\bar\mu\bar\tau\bar\gamma}-\pi^{\bar\tau\bar\sigma}\Gamma_{\bar\sigma\bar\gamma}\,\mathcal V_{\bar\mu\bar\nu\bar\tau}
 -i\pi^{\bar\sigma}_{\tau}\,\phi^{\tau}\,\mathcal S_{\bar\mu\bar\nu\bar\gamma\bar\sigma}-\frac{1}{2}\big(3\varphi_0-i\varphi_1\big)\, \mathcal V_{\bar\mu\bar\nu\bar\gamma}\\
 +\frac{1}{2}\big(\varphi_2+i\varphi_3\big)\, (\mathfrak j\mathcal V)_{\bar\mu\bar\nu\bar\gamma}
 +2\big(\pi_{\bar\mu\bar\tau}\,\mathcal M_{\bar\nu\bar\gamma}+\pi_{\bar\nu\bar\tau}\,\mathcal M_{\bar\mu\bar\gamma}+\pi_{\bar\gamma\bar\tau}\,\mathcal M_{\bar\mu\bar\nu}\big)\theta^{\bar\tau}\\
 -2\big(g_{\tau\bar\mu}\, \mathcal L_{\bar\nu\bar\gamma}
 +g_{\tau\bar\nu}\,\mathcal L_{\bar\mu\bar\gamma}+g_{\tau\bar\gamma}\,\mathcal L_{\bar\mu\bar\nu}\big)\theta^{\tau}\bigg]\wedge\theta^{\bar\gamma}\wedge\eta_1 
 \end{multline*}
 \begin{multline*}
 \qquad\qquad-\ i\pi^{\sigma}_{\bar\gamma}\,\bigg[d{\mathcal V}_{\alpha\beta\sigma}-\pi^{\tau\nu}\Gamma_{\nu\alpha}\,\mathcal V_{\tau\beta\sigma}
 -\pi^{\tau\nu}\Gamma_{\nu\beta}\,\mathcal V_{\alpha\tau\sigma}-\pi^{\tau\nu}\Gamma_{\nu\sigma}\,\mathcal V_{\alpha\beta\tau}
 +i\pi^{\nu}_{\bar\tau}\,\phi^{\bar\tau}\,\mathcal S_{\alpha\beta\sigma\nu}-\frac{1}{2}\big(3\varphi_0+i\varphi_1\big)\, \mathcal V_{\alpha\beta\sigma}\\
 +\frac{1}{2}\big(\varphi_2-i\varphi_3\big)\, (\mathfrak j\mathcal V)_{\alpha\beta\sigma}
 +2\big(\pi_{\alpha\tau}\,\mathcal M_{\beta\sigma}+\pi_{\beta\tau}\,\mathcal M_{\alpha\sigma}+\pi_{\sigma\tau}\,\mathcal M_{\alpha\beta}\big)\theta^{\tau}\\
 +2\big(g_{\alpha\bar\tau}\, \mathcal L_{\beta\sigma}
 +g_{\beta\bar\tau}\,\mathcal L_{\alpha\sigma}+g_{\sigma\bar\tau}\,\mathcal L_{\alpha\beta}\big)\theta^{\bar\tau}\bigg]\wedge\theta^{\bar\gamma}\wedge\big(\eta_2+i\eta_3\big) 
 \end{multline*}
\begin{multline*}
 \qquad\qquad+\ i\pi^{\bar\mu}_\alpha\,\pi^{\bar\nu}_\beta\,\pi^{\bar\xi}_\gamma\,\bigg[d{\mathcal V}_{\bar\mu\bar\nu\bar\xi}-\pi^{\bar\tau\bar\sigma}\Gamma_{\bar\sigma\bar\mu}\,\mathcal V_{\bar\tau\bar\nu\bar\xi}
 -\pi^{\bar\tau\bar\sigma}\Gamma_{\bar\sigma\nu}\,\mathcal V_{\bar\mu\bar\tau\bar\xi}-\pi^{\bar\tau\bar\sigma}\Gamma_{\bar\sigma\bar\xi}\,\mathcal V_{\bar\mu\bar\nu\bar\tau}
 -i\pi^{\bar\sigma}_{\tau}\,\phi^{\tau}\,\mathcal S_{\bar\mu\bar\nu\bar\xi\bar\sigma}-\frac{1}{2}\big(3\varphi_0-i\varphi_1\big)\, \mathcal V_{\bar\mu\bar\nu\bar\xi}\\
 +\frac{1}{2}\big(\varphi_2+i\varphi_3\big)\, (\mathfrak j\mathcal V)_{\bar\mu\bar\nu\bar\xi}
 +2\big(\pi_{\bar\mu\bar\tau}\,\mathcal M_{\bar\nu\bar\xi}+\pi_{\bar\nu\bar\tau}\,\mathcal M_{\bar\mu\bar\xi}+\pi_{\bar\xi\bar\tau}\,\mathcal M_{\bar\mu\bar\nu}\big)\theta^{\bar\tau}\\
 -2\big(g_{\tau\bar\mu}\, \mathcal L_{\bar\nu\bar\xi}
 +g_{\tau\bar\nu}\,\mathcal L_{\bar\mu\bar\xi}+g_{\tau\bar\xi}\,\mathcal L_{\bar\mu\bar\nu}\big)\theta^{\tau}\bigg]\wedge\theta^{\gamma}\wedge\big(\eta_2-i\eta_3\big) 
 \end{multline*} 
 \begin{multline*}
 \qquad\qquad-i\bigg[d\mathcal L_{\alpha\beta}-\pi^{\tau\sigma}\Gamma_{\sigma\alpha}\,\mathcal L_{\tau\beta}-\pi^{\tau\sigma}\Gamma_{\sigma\beta}\,\mathcal L_{\alpha\tau}
-2\varphi_0\,\mathcal L_{\alpha\beta}
-\frac{1}{2}\big(\varphi_2+i\varphi_3\big)\mathcal M_{\alpha\beta}
 -\frac{1}{2}\big(\varphi_2-i\varphi_3\big)(\mathfrak j\mathcal M)_{\alpha\beta}\\
  -\phi^\sigma\,\mathcal V_{\alpha\beta\sigma}-\pi^{\bar\mu}_\alpha\,\pi^{\bar\nu}_\beta\,\phi^{\bar\sigma}\,\mathcal V_{\bar\mu\bar\nu\bar\sigma}
  -2i\big(\pi_{\alpha\tau}\mathcal C_\beta+\pi_{\beta\tau}\mathcal C_\alpha\big)\theta^\tau\\
 -2i\big(g_{\alpha\bar\tau}\,\pi^{\bar\sigma}_{\beta}\,\mathcal C_{\bar\sigma}+g_{\beta\bar\tau}\,\pi^{\bar\sigma}_{\alpha}\,\mathcal C_{\bar\sigma}\big)\theta^{\bar\tau}\bigg]\wedge\big(\eta_2+i\eta_3\big)\wedge\big(\eta_2-i\eta_3\big) 
 \end{multline*} 
  \begin{multline*}
 \qquad\qquad+\bigg[d\mathcal M_{\alpha\beta}-\pi^{\tau\sigma}\Gamma_{\sigma\alpha}\,\mathcal M_{\tau\beta}-\pi^{\tau\sigma}\Gamma_{\sigma\beta}\,\mathcal M_{\alpha\tau}
 -\big(2\varphi_0+i\varphi_1\big)\,\mathcal M_{\alpha\beta}+\big(\varphi_2-i\varphi_3\big)\mathcal L_{\alpha\beta}
 +2\pi^{\sigma}_{\bar\tau}\,\phi^{\bar\tau}\,\mathcal V_{\alpha\beta\sigma}\\
 +2\big(\pi_{\alpha\tau}\mathcal H_\beta+\pi_{\beta\tau}\mathcal H_\alpha\big)\theta^\tau
 -2i\big(g_{\alpha\bar\tau}\,\mathcal C_{\beta}+g_{\beta\bar\tau}\,\mathcal C_{\alpha}\big)\theta^{\bar\tau}\bigg]\wedge\eta_1\wedge\big(\eta_2+i\eta_3\big) 
 \end{multline*} 
 \begin{multline*}
 \qquad\qquad+\pi^{\bar\mu}_\alpha\,\pi^{\bar\nu}_\beta\bigg[d\mathcal M_{\bar\mu\bar\nu}-\pi^{\bar\tau\bar\sigma}\Gamma_{\bar\sigma\bar\mu}\,\mathcal M_{\bar\tau\bar\nu}-\pi^{\bar\tau\bar\sigma}\Gamma_{\bar\sigma\bar\nu}\,\mathcal M_{\bar\mu\bar\tau}
 -\big(2\varphi_0-i\varphi_1\big)\,\mathcal M_{\bar\mu\bar\nu}+\big(\varphi_2+i\varphi_3\big)\mathcal L_{\bar\mu\bar\nu}
 +2\pi^{\bar\sigma}_\tau\,\phi^\tau\,\mathcal V_{\bar\mu\bar\nu\bar\sigma}\\
 +2\big(\pi_{\bar\mu\bar\tau}\mathcal H_{\bar\nu}+\pi_{\bar\nu\bar\tau}\mathcal H_{\bar\mu}\big)\theta^{\bar\tau}
 +2i\big(g_{\tau\bar\mu}\,\mathcal C_{\bar\nu}+g_{\tau\bar\nu}\,\mathcal C_{\bar\mu}\big)\theta^{\tau}\bigg]\wedge\eta_1\wedge\big(\eta_2-i\eta_3\big) 
 \end{multline*} 
and

\begin{multline}\label{d2phi_a1}
 0\ =\ d^2\phi^\alpha\  = \ \frac{i}{2}\bigg[d{\psi_1}-\varphi_0\wedge\psi_1+\varphi_2\wedge\psi_3-\varphi_3\wedge\psi_2+4i\phi_\beta\wedge\phi^\beta
 -4\pi^\sigma_{\bar\gamma}\,\mathcal L_{\beta\sigma}\,\theta^\beta\wedge\theta^{\bar\gamma}\\
  -4\big(\mathcal C_\beta\,\theta^\beta+\mathcal C_{\bar\beta}\,\theta^{\bar\beta}\big)\wedge\eta_1 +4i\pi^{\sigma}_{\bar\beta}\,\mathcal C_{\sigma}\,\theta^{\bar\beta}\wedge\big(\eta_2+i\eta_3\big)
  -4i\pi^{\bar\sigma}_\beta\,\mathcal C_{\bar\sigma}\,\theta^\beta\wedge\big(\eta_2-i\eta_3\big)
  \bigg]\wedge\theta^\alpha 
 \end{multline}
\begin{multline*}
 \qquad\qquad+\ \frac{1}{2}\pi^{\alpha}_{\bar\delta}\bigg[d\big(\psi_2+i\psi_3\big)-\big(\varphi_0-i\varphi_1\big)\wedge\big(\psi_2+i\psi_3\big)-i\big(\varphi_2+i\varphi_3\big)\wedge\psi_1
 -4\pi_{\beta\gamma}\,\phi^\beta\wedge\phi^\gamma
 -4i\pi^{\bar\sigma}_{\beta}\,\mathcal M_{\bar\sigma\bar\gamma}\,\theta^\beta\wedge\theta^{\bar\gamma}\\
  +4\big(- i\pi^{\bar\sigma}_\beta\,\mathcal C_{\bar\sigma}\,\theta^\beta + \mathcal H_{\bar\beta}\,\theta^{\bar\beta}\big)\wedge\eta_1
   +4\mathcal C_{\bar\beta}\,\theta^{\bar\beta}\wedge\big(\eta_2+i\eta_3\big)
  +4i\pi^{\bar\sigma}_\beta\,\mathcal H_{\bar\sigma}\,\theta^\beta\wedge\big(\eta_2-i\eta_3\big)   \bigg]\wedge\theta^{\bar\delta} 
 \end{multline*}
 \begin{multline*}
\qquad\qquad-\ ig^{\alpha\bar\mu}\,\pi^{\bar\nu}_\beta\,\bigg[d{\mathcal V}_{\bar\mu\bar\nu\bar\gamma}-\pi^{\bar\tau\bar\sigma}\Gamma_{\bar\sigma\bar\mu}\,\mathcal V_{\bar\tau\bar\nu\bar\gamma}
 -\pi^{\bar\tau\bar\sigma}\Gamma_{\bar\sigma\nu}\,\mathcal V_{\bar\mu\bar\tau\bar\gamma}-\pi^{\bar\tau\bar\sigma}\Gamma_{\bar\sigma\bar\gamma}\,\mathcal V_{\bar\mu\bar\nu\bar\tau}
 -i\pi^{\bar\sigma}_{\tau}\,\phi^{\tau}\,\mathcal S_{\bar\mu\bar\nu\bar\gamma\bar\sigma}-\frac{1}{2}\big(3\varphi_0-i\varphi_1\big)\, \mathcal V_{\bar\mu\bar\nu\bar\gamma}\\
 +\frac{1}{2}\big(\varphi_2+i\varphi_3\big)\, (\mathfrak j\mathcal V)_{\bar\mu\bar\nu\bar\gamma}
 +2\big(\pi_{\bar\mu\bar\tau}\,\mathcal M_{\bar\nu\bar\gamma}+\pi_{\bar\nu\bar\tau}\,\mathcal M_{\bar\mu\bar\gamma}+\pi_{\bar\gamma\bar\tau}\,\mathcal M_{\bar\mu\bar\nu}\big)\theta^{\bar\tau}\\
 -2\big(g_{\tau\bar\mu}\, \mathcal L_{\bar\nu\bar\gamma}
 +g_{\tau\bar\nu}\,\mathcal L_{\bar\mu\bar\gamma}+g_{\tau\bar\gamma}\,\mathcal L_{\bar\mu\bar\nu}\big)\theta^{\tau}\bigg]\wedge\theta^{\beta}\wedge\theta^{\bar\gamma} 
 \end{multline*}
\begin{multline*}
 \qquad\qquad+\pi^{\alpha\mu}\,\bigg[d\mathcal L_{\mu\beta}-\pi^{\tau\sigma}\Gamma_{\sigma\mu}\,\mathcal L_{\tau\beta}-\pi^{\tau\sigma}\Gamma_{\sigma\beta}\,\mathcal L_{\mu\tau}
-2\varphi_0\,\mathcal L_{\mu\beta}
-\frac{1}{2}\mathcal M_{\mu\beta}\big(\varphi_2+i\varphi_3\big)
 -\frac{1}{2}(\mathfrak j\mathcal M)_{\mu\beta}\big(\varphi_2-i\varphi_3\big)\\
  -\phi^\sigma\,\mathcal V_{\mu\beta\sigma}-\pi^{\bar\tau}_\mu\,\pi^{\bar\nu}_\beta\,\phi^{\bar\sigma}\,\mathcal V_{\bar\tau\bar\nu\bar\sigma}
  -2i\big(\pi_{\mu\tau}\mathcal C_\beta+\pi_{\beta\tau}\mathcal C_\mu\big)\theta^\tau\\
 -2i\big(g_{\mu\bar\tau}\,\pi^{\bar\sigma}_{\beta}\,\mathcal C_{\bar\sigma}+g_{\beta\bar\tau}\,\pi^{\bar\sigma}_{\mu}\,\mathcal C_{\bar\sigma}\big)\theta^{\bar\tau}\bigg]\wedge\theta^{\beta}\wedge\eta_1
 \end{multline*} 
 \begin{multline*}
 \qquad\qquad+\bigg[d\mathcal M^\alpha_{\bar\beta}+\pi^{\alpha\sigma}\Gamma_{\sigma\tau}\,\mathcal M^\tau_{\bar\beta}-\pi^{\bar\tau\bar\sigma}\Gamma_{\bar\sigma\bar\beta}\,\mathcal M^\alpha_{\bar\tau}
 -\big(2\varphi_0-i\varphi_1\big)\,\mathcal M^\alpha_{\bar\beta}+\big(\varphi_2+i\varphi_3\big)\mathcal L^\alpha_{\bar\beta}
 +2\pi^{\bar\sigma}_{\tau}\,\phi^{\tau}\,\mathcal V^\alpha_{\dt\bar\beta\bar\sigma}\\
 +2\big(\pi^{\alpha}_{\bar\tau}\mathcal H_{\bar\beta}+\pi_{\bar\beta\bar\tau}\mathcal H^\alpha\big)\theta^{\bar\tau}
 +2i\big(\delta^\alpha_{\tau}\,\mathcal C_{\bar\beta}+g_{\tau\bar\beta}\,\mathcal C^{\alpha}\big)\theta^{\tau}\bigg]\wedge\theta^{\bar\beta}\wedge\eta_1
 \end{multline*} 
  \begin{multline*}
 \qquad\qquad+i\pi^{\bar\nu}_\beta\,\bigg[d\mathcal M^\alpha_{\bar\nu}+\pi^{\alpha\sigma}\Gamma_{\sigma\tau}\,\mathcal M^\tau_{\bar\nu}-\pi^{\bar\tau\bar\sigma}\Gamma_{\bar\sigma\bar\nu}\,\mathcal M^\alpha_{\bar\tau}
 -\big(2\varphi_0-i\varphi_1\big)\,\mathcal M^\alpha_{\bar\nu}+\big(\varphi_2+i\varphi_3\big)\mathcal L^\alpha_{\bar\nu}
 +2\pi^{\bar\sigma}_{\tau}\,\phi^{\tau}\,\mathcal V^\alpha_{\dt\bar\nu\bar\sigma}\\
 +2\big(\pi^{\alpha}_{\bar\tau}\mathcal H_{\bar\nu}+\pi_{\bar\nu\bar\tau}\mathcal H^\alpha\big)\theta^{\bar\tau}
 +2i\big(\delta^\alpha_{\tau}\,\mathcal C_{\bar\nu}+g_{\tau\bar\nu}\,\mathcal C^{\alpha}\big)\theta^{\tau}\bigg]\wedge\theta^\beta\wedge\big(\eta_2-i\eta_3\big) 
 \end{multline*}
\begin{multline*}
 \qquad\qquad-i\bigg[d\mathcal L^\alpha_{\bar\beta}+\pi^{\alpha\sigma}\Gamma_{\sigma\tau}\,\mathcal L^\tau_{\bar\beta}-\pi^{\bar\tau\bar\sigma}\Gamma_{\bar\sigma\bar\beta}\,\mathcal L^\alpha_{\bar\tau}
-2\varphi_0\,\mathcal L^\alpha_{\bar\beta}
-\frac{1}{2}\big(\varphi_2-i\varphi_3\big)\mathcal M^\alpha_{\bar\beta}
 -\frac{1}{2}\big(\varphi_2+i\varphi_3\big)(\mathfrak j\mathcal M)^\alpha_{\bar\beta}\\
  -\phi^{\bar\sigma}\,\mathcal V^\alpha_{\dt\bar\beta\bar\sigma}-\pi^{\alpha\mu}\,\pi^{\nu}_{\bar\beta}\,\phi^{\sigma}\,\mathcal V_{\mu\nu\sigma}
  +2i\big(-\pi^\alpha_{\bar\tau}\mathcal C_{\bar\beta}+\pi_{\bar\beta\bar\tau}\mathcal C^\alpha\big)\theta^{\bar\tau}\\
 +2i\big(\delta^\alpha_{\tau}\,\pi^{\sigma}_{\bar\beta}\,\mathcal C_{\sigma}+g_{\tau\bar\beta}\,\pi^{\alpha\sigma}\,\mathcal C_{\sigma}\big)\theta^{\tau}\bigg]\wedge\theta^{\bar\beta}\wedge\big(\eta_2+i\eta_3\big) 
 \end{multline*} 
\begin{multline*}
 \qquad\qquad+\bigg[d\mathcal C^\alpha+\pi^{\alpha\sigma}\Gamma_{\sigma\tau}\,\mathcal C^\tau 
 -\frac{1}{2}\big(5\varphi_0-i\varphi_1\big)\,\mathcal C^\alpha-\pi^\alpha_{\bar\sigma}\big(\varphi_2+i\varphi_3\big)\mathcal C^{\bar\sigma}
+ 2i\pi^{\alpha\sigma}\,\phi^\tau\,\mathcal L_{\sigma\tau} + i\phi^{\bar\tau}\,\mathcal M^\alpha_{\bar\tau} \\
+\frac{i}{2} \big(\varphi_2-i\varphi_3\big)\mathcal H^{\alpha}
 \bigg]\wedge\big(\eta_2+i\eta_3\big)\wedge\big(\eta_2-i\eta_3\big) 
 \end{multline*}
 \begin{multline*}
 \qquad\qquad+i\pi^{\alpha}_{\bar\mu}\,\bigg[d\mathcal C^{\bar\mu}+\pi^{\bar\mu\bar\sigma}\Gamma_{\bar\sigma\bar\tau}\,\mathcal C^{\bar\tau} 
 -\frac{1}{2}\big(5\varphi_0+i\varphi_1\big)\,\mathcal C^{\bar\mu}-\pi^{\bar\mu}_{\sigma}\big(\varphi_2-i\varphi_3\big)\mathcal C^{\sigma}
- 2i\pi^{\bar\mu\bar\sigma}\,\phi^{\bar\tau}\,\mathcal L_{\bar\sigma\bar\tau} - i\phi^{\tau}\,\mathcal M^{\bar\mu}_{\tau} \\
-\frac{i}{2} \big(\varphi_2+i\varphi_3\big)\mathcal H^{\bar\mu}
 \bigg]\wedge\eta_1\wedge\big(\eta_2+i\eta_3\big) 
 \end{multline*}
 \begin{multline*}
 \qquad\qquad+\bigg[d\mathcal H^\alpha+\pi^{\alpha\sigma}\Gamma_{\sigma\tau}\,\mathcal H^\tau 
 -\frac{1}{2}\big(5\varphi_0-3i\varphi_1\big)\,\mathcal H^\alpha+\frac{3i}{2}\big(\varphi_2+i\varphi_3\big)\mathcal C^{\alpha}
-3\pi^{\alpha\sigma}\,\phi^\tau\,(\mathfrak j\mathcal M)_{\sigma\tau}
 \bigg]\wedge\eta_1\wedge\big(\eta_2-i\eta_3\big) 
 \end{multline*}
 
Using  \eqref{S*} - \eqref{H*}, by \eqref{d2Gamma_ab1} we obtain the first of the Bianchi identities \eqref{d2Gamma_ab}, whereas  \eqref{d2phi_a1} reads as 
\begin{multline}\label{d2phi_a2}
 0\ =\ d^2\phi_\alpha\  = \ -\frac{i}{2}g_{\alpha\bar\beta}\,\Psi\wedge\theta^{\bar\beta} 
-\frac{1}{2}\pi_{\alpha\beta}\overline{\Phi}\wedge\theta^{\beta} 
 -\ i\pi^{\nu}_{\bar\gamma}\,{\mathcal V}^\ast_{\alpha\beta\nu}\wedge\theta^{\beta}\wedge\theta^{\bar\gamma} 
+\pi^{\bar\mu}_\alpha\,\mathcal L_{\bar\mu\bar\beta}^\ast\wedge\theta^{\bar\beta}\wedge\eta_1\\
 +\mathcal M^\ast_{\alpha\beta}\wedge\theta^{\beta}\wedge\eta_1
-i\pi^{\nu}_{\bar\beta}\,\mathcal M^\ast_{\alpha\nu}\wedge\theta^{\bar\beta}\wedge\big(\eta_2+i\eta_3\big) 
+i\mathcal L^\ast_{\alpha\beta}\wedge\theta^{\beta}\wedge\big(\eta_2-i\eta_3\big) \\
-\mathcal C'_\alpha
 \wedge\big(\eta_2+i\eta_3\big)\wedge\big(\eta_2-i\eta_3\big) 
+i\pi_{\alpha}^{\bar\mu}\,\mathcal C'_{\bar\mu}\wedge\eta_1\wedge\big(\eta_2-i\eta_3\big) 
 +\mathcal H'_\alpha
 \wedge\eta_1\wedge\big(\eta_2+i\eta_3\big) \ =\ 0,
 \end{multline}
 with $\Psi$, $\Phi$ being  two-forms defined by
\begin{multline}\label{Psi}
\Psi\overset{def}{\ =\ }d{\psi_1}-\varphi_0\wedge\psi_1+\varphi_2\wedge\psi_3-\varphi_3\wedge\psi_2+4i\phi_\beta\wedge\phi^\beta
 -4\pi^\sigma_{\bar\gamma}\,\mathcal L_{\beta\sigma}\,\theta^\beta\wedge\theta^{\bar\gamma}
  -4\big(\mathcal C_\beta\,\theta^\beta+\mathcal C_{\bar\beta}\,\theta^{\bar\beta}\big)\wedge\eta_1\\
   +4i\pi^{\sigma}_{\bar\beta}\,\mathcal C_{\sigma}\,\theta^{\bar\beta}\wedge\big(\eta_2+i\eta_3\big)
  -4i\pi^{\bar\sigma}_\beta\,\mathcal C_{\bar\sigma}\,\theta^\beta\wedge\big(\eta_2-i\eta_3\big),
\end{multline}
 \begin{multline}\label{Phi}
\Phi\overset{def}{\ =\ } d\big(\psi_2+i\psi_3\big)-\big(\varphi_0-i\varphi_1\big)\wedge\big(\psi_2+i\psi_3\big)-i\big(\varphi_2+i\varphi_3\big)\wedge\psi_1
 -4\pi_{\beta\gamma}\,\phi^\beta\wedge\phi^\gamma\\
 -4i\pi^{\bar\sigma}_{\beta}\,\mathcal M_{\bar\sigma\bar\gamma}\,\theta^\beta\wedge\theta^{\bar\gamma}
  +4\big(- i\pi^{\bar\sigma}_\beta\,\mathcal C_{\bar\sigma}\,\theta^\beta + \mathcal H_{\bar\beta}\,\theta^{\bar\beta}\big)\wedge\eta_1
   +4\mathcal C_{\bar\beta}\,\theta^{\bar\beta}\wedge\big(\eta_2+i\eta_3\big)\\
  +4i\pi^{\bar\sigma}_\beta\,\mathcal H_{\bar\sigma}\,\theta^\beta\wedge\big(\eta_2-i\eta_3\big),
 \end{multline} 
 and $\mathcal C'_\alpha$, $\mathcal H'_\alpha$ being one-forms given by
  \begin{multline}\label{C'}
\mathcal C'_{\alpha}\overset{def}{\ =\ }d\mathcal C_\alpha-\pi^{\tau\sigma}\Gamma_{\sigma\alpha}\,\mathcal C_\tau 
 -\frac{1}{2}\big(5\varphi_0+i\varphi_1\big)\,\mathcal C_\alpha+\pi^{\bar\sigma}_{\alpha}\big(\varphi_2-i\varphi_3\big)\mathcal C_{\bar\sigma}
+ 2i\pi^{\sigma}_{\bar\tau}\,\phi^{\bar\tau}\,\mathcal L_{\alpha\sigma} - i\phi^{\tau}\,\mathcal M_{\alpha\tau} \\
-\frac{i}{2} \big(\varphi_2+i\varphi_3\big)\mathcal H_{\alpha},
 \end{multline}
 \begin{multline}\label{H'}
\mathcal H'_{\alpha}\overset{def}{\ =\ }d\mathcal H_\alpha-\pi^{\tau\sigma}\Gamma_{\sigma\alpha}\,\mathcal H_\tau 
 -\frac{1}{2}\big(5\varphi_0+3i\varphi_1\big)\,\mathcal H_\alpha-\frac{3i}{2}\big(\varphi_2-i\varphi_3\big)\mathcal C_{\alpha}
+3\pi^{\sigma}_{\bar\tau}\,\phi^{\bar\tau}\,\mathcal M_{\alpha\sigma}.\\
 \end{multline}
 
 An immediate consequence of \eqref{d2Gamma_ab} and \eqref{d2phi_a2} is that
 \begin{equation}\label{S*_to_H*}
 \begin{cases}
 \mathcal S^\ast_{\alpha\beta\gamma\delta}=\mathcal S_{\alpha\beta\gamma\delta,\epsilon}\,\theta^{\epsilon}+ \mathcal S_{\alpha\beta\gamma\delta,\bar\epsilon}\,\theta^{\bar\epsilon}
 +(\mathcal S_1)_{\alpha\beta\gamma\delta}\,\eta_1 + (\mathcal S_2)_{\alpha\beta\gamma\delta}\big(\eta_2+i\eta_3\big) +
(\mathcal S_3)_{\alpha\beta\gamma\delta}\big(\eta_2-i\eta_3\big)\\
 \mathcal V^\ast_{\alpha\beta\gamma}=\mathcal V_{\alpha\beta\gamma,\epsilon}\,\theta^{\epsilon}+\mathcal V_{\alpha\beta\gamma,\bar\epsilon}\,\theta^{\bar\epsilon}
 +(\mathcal V_1)_{\alpha\beta\gamma}\,\eta_1 + (\mathcal V_2)_{\alpha\beta\gamma}\big(\eta_2+i\eta_3\big)+(\mathcal V_3)_{\alpha\beta\gamma}\big(\eta_2-i\eta_3\big)\\
 \mathcal L^\ast_{\alpha\beta}=\mathcal L_{\alpha\beta,\epsilon}\,\theta^{\epsilon}+\mathcal L_{\alpha\beta,\bar\epsilon}\,\theta^{\bar\epsilon}
 +(\mathcal L_1)_{\alpha\beta}\,\eta_1 + (\mathcal L_2)_{\alpha\beta}\big(\eta_2+i\eta_3\big)+(\mathcal L_3)_{\alpha\beta}\big(\eta_2-i\eta_3\big)\\
  \mathcal M^\ast_{\alpha\beta}=\ \mathcal M_{\alpha\beta,\epsilon}\,\theta^{\epsilon}+\mathcal M_{\alpha\beta,\bar\epsilon}\,\theta^{\bar\epsilon}
 +(\mathcal M_1)_{\alpha\beta}\,\eta_1 + (\mathcal M_2)_{\alpha\beta}\big(\eta_2+i\eta_3\big)+(\mathcal M_3)_{\alpha\beta}\big(\eta_2-i\eta_3\big)\\
  \mathcal C'_{\alpha}=\mathcal C_{\alpha,\epsilon}\,\theta^{\epsilon}+\mathcal C_{\alpha,\bar\epsilon}\,\theta^{\bar\epsilon}
 +(\mathcal C_1)_{\alpha}\,\eta_1 + (\mathcal C_2)_{\alpha}\big(\eta_2+i\eta_3\big)+(\mathcal C_3)_{\alpha}\big(\eta_2-i\eta_3\big)\\
  \mathcal H'_{\alpha}=\mathcal H_{\alpha,\epsilon}\,\theta^{\epsilon}+\mathcal H_{\alpha,\bar\epsilon}\theta^{\bar\epsilon}
 +(\mathcal H_1)_{\alpha}\,\eta_1 + (\mathcal H_2)_{\alpha}\big(\eta_2+i\eta_3\big)+(\mathcal H_3)_{\alpha}\big(\eta_2-i\eta_3\big).
\end{cases}
 \end{equation}
for some appropriate coefficients $\mathcal S_{\alpha\beta\gamma\delta,\epsilon}$,  $\mathcal S_{\alpha\beta\gamma\delta,\bar\epsilon}$,  $(\mathcal S_s)_{\alpha\beta\gamma\delta}$,
$\mathcal V_{\alpha\beta\gamma,\epsilon}$,  $\mathcal V_{\alpha\beta\gamma,\bar\epsilon}$,  $(\mathcal V_s)_{\alpha\beta\gamma}$, $\mathcal L_{\alpha\beta,\epsilon}$,  $\mathcal L_{\alpha\beta,\bar\epsilon}$,  $(\mathcal L_s)_{\alpha\beta}$, $\mathcal M_{\alpha\beta,\epsilon}$,  $\mathcal M_{\alpha\beta,\bar\epsilon}$,  $(\mathcal M_s)_{\alpha\beta}$, $\mathcal C_{\alpha,\epsilon}$,  $\mathcal C_{\alpha,\bar\epsilon}$,  $(\mathcal C_s)_{\alpha}$, $\mathcal H_{\alpha,\epsilon}$,  $\mathcal H_{\alpha,\bar\epsilon}$,  $(\mathcal H_s)_{\alpha}$. 

Substituting \eqref{S*_to_H*} back into \eqref{d2Gamma_ab}, we consider only the terms involving $\theta^{\epsilon}\wedge\theta^\gamma\wedge\theta^{\bar\delta}$  and  $\theta^{\bar\epsilon}\wedge\theta^\gamma\wedge\theta^{\bar\delta}$. Then, 
\begin{equation*}
\pi^{\sigma}_{\bar\delta}\Big(\mathcal S_{\alpha\beta\gamma\sigma,\epsilon}\,\theta^\epsilon+\mathcal S_{\alpha\beta\gamma\sigma,\bar\epsilon}\,\theta^{\bar\epsilon}\Big)\wedge\theta^\gamma\wedge\theta^{\bar\delta}\ =\ 0
\end{equation*}
and it follows that the array $\{\mathcal S_{\alpha\beta\gamma\sigma,\epsilon}\}$ must be totally symmetric. By the first line of \eqref{properties-curvature} (which we have already proved), 
\begin{equation*}
\mathcal S^\ast_{\alpha\beta\gamma\delta}=\pi^{\bar\mu}_\alpha\,\pi^{\bar\nu}_\beta\,\pi^{\bar\sigma}_\gamma\,\pi^{\bar\tau}_\delta\,\mathcal S^\ast_{\bar\mu\bar\nu\bar\sigma\bar\tau}
\end{equation*}
and therefore,
\begin{equation*}
\mathcal S_{\alpha\beta\gamma\delta,\bar\epsilon}=\pi^{\bar\mu}_\alpha\,\pi^{\bar\nu}_\beta\,\pi^{\bar\sigma}_\gamma\,\pi^{\bar\tau}_\delta\,\mathcal S_{\bar\mu\bar\nu\bar\sigma\bar\tau,\bar\epsilon}.
\end{equation*}
Hence, defining
\begin{equation}
\mathcal A_{\alpha\beta\gamma\delta\epsilon}\overset{def}{=}\mathcal S_{\alpha\beta\gamma\delta,\epsilon},
\end{equation}
we obtain
\begin{equation}
\mathcal S_{\alpha\beta\gamma\delta,\epsilon}=-\pi^{\sigma}_{\bar\epsilon}(\mathfrak j\mathcal A)_{\alpha\beta\gamma\delta\sigma}.
\end{equation}

The vanishing of the coefficients of $\theta^{\gamma}\wedge\theta^{\bar\delta}\wedge\eta_1$, $\theta^{\gamma}\wedge\theta^{\bar\delta}\wedge(\eta_2+i\eta_3)$ and $\theta^{\gamma}\wedge\theta^{\bar\delta}\wedge(\eta_2-i\eta_3)$ (after substituting  \eqref{S*_to_H*} into \eqref{d2Gamma_ab}) yields
\begin{equation*}
\begin{split}
&\pi^\sigma_{\bar\delta}(\mathcal S_1)_{\alpha\beta\gamma\sigma}-\mathcal V_{\alpha\beta\gamma,\bar\delta}+\pi^{\bar\mu}_\alpha\,\pi^{\bar\nu}_\beta\,\mathcal V_{\bar\mu\bar\nu\bar\delta,\gamma}\ =\ 0\\
&\pi^\sigma_{\bar\delta}(\mathcal S_2)_{\alpha\beta\gamma\sigma}-i\pi^{\sigma}_{\bar\delta}\mathcal V_{\alpha\beta\sigma,\gamma}\ =\ 0\\
&\pi^\sigma_{\bar\delta}(\mathcal S_3)_{\alpha\beta\gamma\sigma}-i\pi^{\bar\mu}_\alpha\,\pi^{\bar\nu}_\beta\,\pi^{\bar\tau}_\gamma\,\mathcal V_{\bar\mu\bar\nu\bar\tau,\bar\delta}\ =\ 0.
\end{split}
\end{equation*}
Therefore, if we define
\begin{equation*}
\begin{cases}
\mathcal B_{\alpha\beta\gamma\delta}\overset{def}{=} -\pi^{\bar\sigma}_\delta\,\mathcal V_{\alpha\beta\gamma,\bar\sigma}\\
\mathcal C_{\alpha\beta\gamma\delta}\overset{def}{=}-i(\mathcal S_2)_{\alpha\beta\gamma\delta},
\end{cases}
\end{equation*}
we obtain that the arrays $\{\mathcal B_{\alpha\beta\gamma\delta}\}$ and $\{\mathcal C_{\alpha\beta\gamma\delta}\}$ are totally symmetric and 
\begin{gather}
(\mathcal S_1)_{\alpha\beta\gamma\delta}=B_{\alpha\beta\gamma\delta}+(\mathfrak j \mathcal B)_{\alpha\beta\gamma\delta},\qquad
(\mathcal S_2)_{\alpha\beta\gamma\delta}=i\mathcal C_{\alpha\beta\gamma\delta},\qquad
(\mathcal S_3)_{\alpha\beta\gamma\delta}=-i(\mathfrak j \mathcal C)_{\alpha\beta\gamma\delta},\\\nonumber
\mathcal V_{\alpha\beta\gamma,\epsilon}=\mathcal C_{\alpha\beta\gamma\epsilon},\qquad \mathcal V_{\alpha\beta\gamma,\bar\epsilon}=\pi^{\delta}_{\bar\epsilon}\,\mathcal B_{\alpha\beta\gamma\delta}.
\end{gather}

Similarly, the coefficients of $\theta^{\gamma}\wedge\eta_1\wedge(\eta_2+i\eta_3)$,  $\theta^{\bar\gamma}\wedge\eta_1\wedge(\eta_2+i\eta_3)$ and $\theta^{\gamma}\wedge(\eta_2+i\eta_3)\wedge(\eta_2-i\eta_3)$ give the equations
\begin{equation*}
\begin{split}
&(\mathcal V_2)_{\alpha\beta\gamma}+\mathcal M_{\alpha\beta,\gamma}\ =\ 0\\
&i\pi^\sigma_{\bar\gamma}(\mathcal V_1)_{\alpha\beta\sigma}+\mathcal M_{\alpha\beta,\bar\gamma}
+\pi^{\bar\mu}_\alpha\,\pi^{\bar\nu}_\beta(\mathcal V_3)_{\bar\mu\bar\nu\bar\gamma}\ =\ 0\\
&i\pi^{\bar\mu}_\alpha\,\pi^{\bar\nu}_\beta\,\pi^{\bar\tau}_\gamma(\mathcal V_3)_{\bar\mu\bar\nu\bar\tau}+i\mathcal L_{\alpha\beta,\gamma}\ =\ 0.
\end{split}
\end{equation*}
Defining 
\begin{equation*}
\begin{cases}
\mathcal D_{\alpha\beta\gamma}\overset{def}{=} (\mathcal V_1)_{\alpha\beta\gamma}\\
\mathcal E_{\alpha\beta\gamma}\overset{def}{=} (\mathcal V_2)_{\alpha\beta\gamma}\\
\mathcal F_{\alpha\beta\gamma}\overset{def}{=} (\mathcal V_3)_{\alpha\beta\gamma},
\end{cases}
\end{equation*}
we deduce that
\begin{gather}
\mathcal L_{\alpha\beta,\gamma}=-(\mathfrak j\mathcal F)_{\alpha\beta\gamma},\qquad \mathcal L_{\alpha\beta,\bar\gamma}=-\pi^{\sigma}_{\bar\gamma}\,\mathcal F_{\alpha\beta\sigma},
\\\nonumber
\mathcal M_{\alpha\beta,\gamma}=-\mathcal E_{\alpha\beta\gamma},\qquad \mathcal M_{\alpha\beta,\bar\gamma}=-i\pi^{\sigma}_{\bar\gamma}\,\mathcal D_{\alpha\beta\sigma}
-\pi^{\bar\mu}_{\alpha}\,\pi^{\bar\nu}_{\beta}\,\mathcal F_{\bar\mu\bar\nu\bar\gamma}.
\end{gather}

Another consequence of \eqref{d2phi_a2} is that the two-forms $\Psi$ and $\Phi$
must be contained in $\Lambda^2\{\theta^\alpha,\theta^{\bar\alpha},\eta_s\}$ and therefore, there should exist functions  $(X_s)_{\alpha\beta}=-(X_s)_{\beta\alpha}$ and $(Y_s)_{\alpha\bar\beta}=-(Y_s)_{\bar\beta\alpha}$ so that
 \begin{equation}\label{shape-dpsi_s}
 \begin{split}
 d\psi_1-\varphi_0\wedge\psi_1+&\varphi_2\wedge\psi_3-\varphi_3\wedge\psi_2+4i\phi_\beta\wedge\phi^\beta+4\pi^{\sigma}_{\bar\beta}\,\mathcal L_{\alpha\sigma}\,\theta^\alpha\wedge\theta^{\bar\beta} \equiv\ \\
 &\equiv(X_1)_{\alpha\beta}\,\theta^{\alpha}\wedge\theta^{\beta}+(X_1)_{\bar\alpha\bar\beta}\,\theta^{\bar\alpha}\wedge\theta^{\bar\beta}
 +(Y_1)_{\alpha\bar\beta}\,\theta^{\alpha}\wedge\theta^{\bar\beta}\\
d\psi_2-\varphi_0\wedge\psi_2+&\varphi_3\wedge\psi_1-\varphi_1\wedge\psi_3+2\pi_{\alpha\beta}\,\phi^{\alpha}\wedge\phi^{\beta}\\
&
+2\pi_{\bar\alpha\bar\beta}\,\phi^{\bar\alpha}\wedge\phi^{\bar\beta}+2i\pi^{\sigma}_{\bar\beta}\Big(\mathcal M_{\alpha\sigma}-(\mathfrak j\mathcal M)_{\alpha\sigma}\Big)\theta^\alpha\wedge\theta^{\bar\beta}\\ 
 &\equiv (X_2)_{\alpha\beta}\,\theta^{\alpha}\wedge\theta^{\beta}+(X_2)_{\bar\alpha\bar\beta}\,\theta^{\bar\alpha}\wedge\theta^{\bar\beta}+(Y_2)_{\alpha\bar\beta}\,\theta^{\alpha}\wedge\theta^{\bar\beta}\\
d\psi_3-\varphi_0\wedge\psi_3+&\varphi_1\wedge\psi_2-\varphi_2\wedge\psi_1-2i\pi_{\alpha\beta}\,\phi^{\alpha}\wedge\phi^{\beta}\\
&
+2i\pi_{\bar\alpha\bar\beta}\,\phi^{\bar\alpha}\wedge\phi^{\bar\beta}-2\pi^{\sigma}_{\bar\beta}\Big(\mathcal M_{\alpha\sigma}+(\mathfrak j\mathcal M)_{\alpha\sigma}\Big)\theta^\alpha\wedge\theta^{\bar\beta}\\ 
 &\equiv (X_3)_{\alpha\beta}\,\theta^{\alpha}\wedge\theta^{\beta}+(X_3)_{\bar\alpha\bar\beta}\,\theta^{\bar\alpha}\wedge\theta^{\bar\beta}+(Y_3)_{\alpha\bar\beta}\,\theta^{\alpha}\wedge\theta^{\bar\beta}
\end{split}\qquad\mod {\eta_s}
 \end{equation}
 Substituting \eqref{shape-dpsi_s} into \eqref{d2phi_a2}, we obtain that
 \begin{multline}
 0\ =\ \frac{i}{2}\Big[(X_1)_{\beta\gamma}\,\theta^{\beta}\wedge\theta^{\gamma}+(X_1)_{\bar\beta\bar\gamma}\,\theta^{\bar\beta}\wedge\theta^{\bar\gamma}
 +(Y_1)_{\beta\bar\gamma}\,\theta^{\beta}\wedge\theta^{\bar\gamma}\Big]\wedge\theta^\alpha\\
 +\frac{1}{2}\pi^{\alpha}_{\bar\delta}\bigg[ \Big((X_2)_{\beta\gamma}+i(X_3)_{\beta\gamma}\Big)\theta^{\beta}\wedge\theta^{\gamma}
 +\Big((X_2)_{\bar\beta\bar\gamma}+i(X_3)_{\bar\beta\bar\gamma}\Big)\,\theta^{\bar\beta}\wedge\theta^{\bar\gamma}\\
 +\Big((Y_2)_{\beta\bar\gamma}+i(Y_3)_{\beta\bar\gamma}\Big)\theta^{\beta}\wedge\theta^{\bar\gamma}\bigg]\wedge\theta^{\bar\delta},
 \end{multline}
 which yields the system of equations
 \begin{equation}\label{eq_for_X_s_ab}
 \begin{split}
 &(X_1)_{\alpha\beta}=0,\\
 &(X_2)_{\alpha\beta}-i(X_3)_{\alpha\beta}=0,\\
 &i\Big(g_{\beta\bar\alpha} (Y_1)_{\gamma\bar\delta}-g_{\gamma\bar\alpha} (Y_1)_{\beta\bar\delta}\Big)\ =\ 2\pi_{\bar\alpha\bar\delta}\Big((X_2)_{\beta\gamma}+i(X_3)_{\beta\gamma}\Big),\\
 &\Big(\pi_{\bar\alpha\bar\delta}\,(Y_2)_{\beta\bar\gamma}-\pi_{\bar\alpha\bar\gamma}\,(Y_2)_{\beta\bar\delta}+i\pi_{\bar\alpha\bar\delta}\,(Y_3)_{\beta\bar\gamma}-i\pi_{\bar\alpha\bar\gamma}\,(Y_3)_{\beta\bar\delta}\Big)=2ig_{\bar\alpha\beta}(X_1)_{\bar\gamma\bar\delta}.
 \end{split}
 \end{equation}
 
 If we multiply the third line of \eqref{eq_for_X_s_ab} by $g^{\beta\bar\alpha}$ and take the sum in $\bar\alpha$ and $\beta$ we obtain that
 \begin{equation*}
 (Y_1)_{\gamma\bar\delta}=-\frac{2i}{4n-1}\pi^{\sigma}_{\bar\delta}\Big((X_2)_{\gamma\sigma}+i(X_3)_{\gamma\sigma}\Big).
 \end{equation*}
 Substituting back into \eqref{eq_for_X_s_ab} gives
 \begin{equation*}
 \frac{2}{4n-1}\Big[g_{\beta\bar\alpha}\,\pi^{\sigma}_{\bar\delta}\Big((X_2)_{\gamma\sigma}+i(X_3)_{\gamma\sigma}\Big)
 -g_{\gamma\bar\alpha}\,\pi^{\sigma}_{\bar\delta}\Big((X_2)_{\beta\sigma}+i(X_3)_{\beta\sigma}\Big)\Big]\ =\ 2\pi_{\bar\alpha\bar\delta}\Big((X_2)_{\beta\gamma}+i(X_3)_{\beta\gamma}\Big)
 \end{equation*}
 Now, we multiply the latter by $\pi^{\bar\alpha\bar\delta}$ and take the sum in $\bar\alpha$ and $\bar\delta$ to arrive at
 \begin{equation*}
-\frac{4}{4n-1}\Big((X_2)_{\beta\gamma}+i(X_3)_{\beta\gamma}\Big)=8n\Big((X_2)_{\beta\gamma}+i(X_3)_{\beta\gamma}\Big).
 \end{equation*}
 This together with the second line of \eqref{eq_for_X_s_ab} implies that $(X_2)_{\alpha\beta}=(X_3)_{\alpha\beta}=0$. Proceeding similarly with the forth line of \eqref{eq_for_X_s_ab}, we conclude that $(X_s)_{\alpha\beta}=0$ and $(Y_s)_{\alpha\bar\beta}=0$.
 
 By considering the coefficients of $\theta^\beta\wedge\theta^\gamma\wedge\eta_1$, $\theta^\beta\wedge\theta^{\bar\gamma}\wedge\eta_1$, $\theta^\beta\wedge\theta^\gamma\wedge(\eta_2+i\eta_3)$,  $\theta^\beta\wedge\theta^{\bar\gamma}\wedge(\eta_2+i\eta_3)$,  $\theta^\beta\wedge\theta^\gamma\wedge(\eta_2-i\eta_3)$ and  $\theta^\beta\wedge\theta^{\bar\gamma}\wedge(\eta_2-i\eta_3)$ in \eqref{d2phi_a2}, we easily obtain that 
 \begin{equation}\label{phi_psi_UVW}
 \begin{split}
\Psi
  &=U\,\eta_1\wedge\big(\eta_2+i\eta_3\big)+\overline{U}\,\eta_1\wedge\big(\eta_2-i\eta_3\big)+iW\,\big(\eta_2+i\eta_3\big)\wedge\big(\eta_2-i\eta_3\big)\\
\Phi
&=V_1\,\eta_1\wedge\big(\eta_2+i\eta_3\big)+V_2\,\eta_1\wedge\big(\eta_2-i\eta_3\big)+V_3\,\big(\eta_2+i\eta_3\big)\wedge\big(\eta_2-i\eta_3\big)
 \end{split}
 \end{equation}
for some appropriate coefficients $U$, $V_s$ and $W=\overline {W}$. Substituting back into \eqref{d2phi_a2} and considering the coefficients of  $\theta^{\bar\beta}\wedge\eta_1\wedge(\eta_2+i\eta_3)$, $\theta^{\bar\beta}\wedge\eta_1\wedge(\eta_2+i\eta_3)$, $\theta^{\beta}\wedge\eta_1\wedge(\eta_2-i\eta_3)$, $\theta^{\bar\beta}\wedge\eta_1\wedge(\eta_2-i\eta_3)$,  $\theta^{\beta}\wedge(\eta_2+i\eta_3)\wedge(\eta_2-i\eta_3)$ and $\theta^{\bar\beta}\wedge(\eta_2+i\eta_3)\wedge(\eta_2-i\eta_3)$ separately, we obtain the equations

\begin{equation}\label{eq_for_UV}
\begin{split}
&\mathcal H_{\alpha,\beta} +(\mathcal M_2)_{\alpha\beta}-\frac{1}{2}\pi_{\alpha\beta}\overline{V_2}\ =\ 0\\
&\mathcal H_{\alpha,\bar\beta}-\frac{i}{2}g_{\alpha\bar\beta}\,U+\pi^{\bar\mu}_\alpha (\mathcal L_3)_{\bar\mu\bar\beta}+i\pi^\nu_{\bar\beta}(\mathcal M_1)_{\alpha\nu}\ =\ 0\\
&\mathcal C_{\alpha,\bar\beta}-i\pi^{\bar\sigma}_\alpha\,(\mathcal M_3)_{\bar\sigma\bar\beta}
+\pi^{\bar\sigma}_\alpha\,(\mathcal L_1)_{\bar\sigma\bar\beta}-\frac{i}{2}g_{\alpha\bar\beta}V_1\ =\ 0\\
&\mathcal C_{\alpha\beta}+\frac{1}{2}\pi_{\alpha\beta}\, U + i(\mathcal L_2)_{\alpha\beta}\ =\ 0\\
&\mathcal C_{\alpha\beta}-\frac{1}{2}\pi_{\alpha\beta}\, \overline{V_3} + i(\mathcal L_2)_{\alpha\beta}\ =\ 0\\
&\mathcal C_{\alpha,\bar\beta}+i\pi^{\sigma}_{\bar\beta}\,(\mathcal M_3)_{\alpha\sigma}
-\frac{1}{2}g_{\alpha\bar\beta}W\ =\ 0.
\end{split}
\end{equation} 

By the third and the sixth lines of \eqref{eq_for_UV}, we have
\begin{equation*}
\frac{1}{2}\pi_{\alpha\beta}(W+i\overline V_1)=(\mathcal L_1)_{\alpha\beta}+i(\mathcal M_3)_{\alpha\beta}-i\pi^{\bar\mu}_\alpha\,\pi^{\bar\nu}_\beta(\mathcal M_3)_{\bar\mu\bar\nu}
\end{equation*}  
and hence 
\begin{equation*}
V_1=-iW,\qquad (\mathcal L_1)_{\alpha\beta}=-i(\mathcal M_3)_{\alpha\beta}+i\pi^{\bar\mu}_\alpha\,\pi^{\bar\nu}_\beta(\mathcal M_3)_{\bar\mu\bar\nu}.
\end{equation*}    

Whereas, the forth and the fifth lines of \eqref{eq_for_UV} yield
\begin{equation*}
V_3=-\overline{U}.
\end{equation*} 

Therefore, if we define
\begin{equation*}
\begin{cases}
\mathcal P\overset{def}{=}U\\
\mathcal Q\overset{def}{=}\overline{V_2}\\
\mathcal R\overset{def}{=}W,
\end{cases}
\end{equation*}
we get the structure equations \eqref{dpsi_1} and \eqref{dpsi_23}, which completes the proof of Proposition~\ref{curvature}.

Furthermore, defining 
\begin{equation*}
\begin{cases}
\mathcal G_{\alpha\beta}\overset{def}{=}-i(\mathcal L_2)_{\alpha\beta}\\
\mathcal X_{\alpha\beta}\overset{def}{=}(\mathcal M_1)_{\alpha\beta}\\
\mathcal Y_{\alpha\beta}\overset{def}{=}(\mathcal M_2)_{\alpha\beta}\\
\mathcal Z_{\alpha\beta}\overset{def}{=}(\mathcal M_3)_{\alpha\beta},
\end{cases}
\end{equation*}
we obtain that
\begin{equation}
\begin{gathered}
(\mathcal L_1)_{\alpha\beta}=i\Big((\mathfrak j\mathcal Z_{\alpha\beta})-\mathcal Z_{\alpha\beta}\Big),\qquad  (\mathcal L_2)_{\alpha\beta}=i\mathcal G_{\alpha\beta},
\qquad (\mathcal L_3)_{\alpha\beta}=-i(\mathfrak j\mathcal G)_{\alpha\beta},\\
\mathcal C_{\alpha,\beta}=\mathcal G_{\alpha\beta}-\frac{1}{2}\pi_{\alpha\beta}\,\mathcal P,\qquad
\mathcal C_{\alpha,\bar\beta}=-i\pi^{\sigma}_{\bar\beta}\,\mathcal Z_{\alpha\sigma}+\frac{1}{2}g_{\alpha\bar\beta}\,\mathcal R,\\
\mathcal H_{\alpha,\beta}=-\mathcal Y_{\alpha\beta}+\frac{1}{2}\pi_{\alpha\beta}\,\mathcal Q,\qquad
\mathcal H_{\alpha,\bar\beta}=i\pi^{\sigma}_{\bar\beta}\Big(\mathcal G_{\alpha\sigma}-\mathcal X_{\alpha\sigma}\Big)+\frac{i}{2}g_{\alpha\bar\beta}\,\mathcal P.
\end{gathered}
\end{equation}

Now, substituting \eqref{phi_psi_UVW} into \eqref{d2phi_a2} and using the above relations, we get the second of the Bianchi identities \eqref{d2phi_a}.

We proceed by differentiating both sides of the equations \eqref{dpsi_1} and \eqref{dpsi_23}. After some straightforward calculations we arrive at the third \eqref{d2psi_1} and the forth \eqref{d2psi_2} of the Bianchi identities, which completes the proof of Proposition~\ref{bianchi}.

One immediate consequence of  \eqref{d2psi_1} and \eqref{d2psi_2} is that the one forms $\mathcal P^\ast$, $\mathcal Q^\ast$ and $\mathcal R^\ast$ must be vanishing modulo $\{\theta^\alpha,\theta^{\bar\alpha},\eta_s\}$. Let 
\begin{equation}\label{P*_to_R*}
 \begin{cases}
 \mathcal P^\ast = X_\epsilon\,\theta^{\epsilon}+  Y_{\bar\epsilon}\,\theta^{\bar\epsilon}
 +\mathcal P_1\,\eta_1 + \mathcal P_2\big(\eta_2+i\eta_3\big) +
\mathcal P_3\big(\eta_2-i\eta_3\big)\\
 \mathcal Q^\ast = Z_\epsilon\,\theta^{\epsilon}+W_{\bar\epsilon}\,\theta^{\bar\epsilon}
 +\mathcal Q_1\,\eta_1 + \mathcal Q_2\big(\eta_2+i\eta_3\big) +
\mathcal Q_3\big(\eta_2-i\eta_3\big)\\
 \mathcal R^\ast = U_\epsilon\,\theta^{\epsilon}+U_{\bar\epsilon}\,\theta^{\bar\epsilon}
 +\mathcal R_1\,\eta_1 + \mathcal R_2\big(\eta_2+i\eta_3\big) +
\overline{\mathcal R_2}\big(\eta_2-i\eta_3\big),
\end{cases}
 \end{equation}
where $X_\epsilon$, $Y_\epsilon$, $Z_\epsilon$, $W_\epsilon$, $U_\epsilon$, $\mathcal P_s$, $\mathcal Q_s$ and $\mathcal R_s$ are some appropriate functions. Substituting \eqref{P*_to_R*} and \eqref{S*_to_H*} into \eqref{d2psi_1}, \eqref{d2psi_2}  and considering the coefficients of $\theta^\alpha\wedge\eta_1\wedge(\eta_2+i\eta_3)$, $\theta^{\bar\alpha}\wedge\eta_1\wedge(\eta_2+i\eta_3)$, $\theta^\alpha\wedge\eta_1\wedge(\eta_2-i\eta_3)$, $\theta^{\bar\alpha}\wedge\eta_1\wedge(\eta_2-i\eta_3)$, $\theta^\alpha\wedge(\eta_2+i\eta_3)\wedge(\eta_2-i\eta_3)$, $\theta^{\bar\alpha}\wedge(\eta_2+i\eta_3)\wedge(\eta_2-i\eta_3)$  and $\eta_1\wedge(\eta_2+i\eta_3)\wedge(\eta_2-i\eta_3)$ separately, we obtain the equations
\begin{equation}\label{eq_for_N_s}
\begin{split}
&X_{\alpha} +4(\mathcal C_2)_{\alpha}\ =\ 0\\
&U_{\alpha} -4\pi^{\bar\sigma}_\alpha(\mathcal C_3)_{\bar\sigma}\ =\ 0\\
&Y_{\alpha}+4(\mathcal C_3)_\alpha -4i\pi^{\bar\sigma}_\alpha(\mathcal C_1)_{\bar\sigma}\ =\ 0\\
&U_{\alpha} +4i(\mathcal H_3)_{\alpha}-4i(\mathcal C_1)_{\alpha}\ =\ 0\\
&W_{\alpha} +4i\pi^{\bar\sigma}_\alpha\Big((\mathcal C_2)_{\bar\sigma}+(\mathcal H_1)_{\bar\sigma}\Big)\ =\ 0\\
&Z_{\alpha} -4(\mathcal H_2)_{\alpha}\ =\ 0\\
&Y_{\alpha} -4i\pi^{\bar\sigma}_\alpha(\mathcal H_3)_{\bar\sigma}\ =\ 0\\
&\mathcal R_1 -i\big(\mathcal P_3 -\overline{\mathcal P_3}\big)\ =\ 0\\
&\mathcal R_2 +i\big(\mathcal P_1 +\mathcal Q_3\big)\ =\ 0.
\end{split}
\end{equation} 

From these we easily deduce that if we define
\begin{equation*}
\begin{gathered}
\mathcal U_1\overset{def}{=}\mathcal P_1,\qquad
\mathcal U_2\overset{def}{=}\mathcal P_2,\qquad
\mathcal U_3\overset{def}{=}\mathcal P_3,\\
\mathcal W_1\overset{def}{=}\mathcal Q_1,\qquad
\mathcal W_2\overset{def}{=}\mathcal Q_2,\qquad
\mathcal W_3\overset{def}{=}\mathcal Q_3,\\
(\mathcal N_1)_\alpha\overset{def}{=}(\mathcal C_1)_\alpha,\qquad
(\mathcal N_2)_\alpha\overset{def}{=}(\mathcal C_2)_\alpha,\qquad
(\mathcal N_3)_\alpha\overset{def}{=}(\mathcal C_3)_\alpha,\\
(\mathcal N_4)_\alpha\overset{def}{=}(\mathcal H_1)_\alpha,\qquad
(\mathcal N_5)_\alpha\overset{def}{=}(\mathcal H_2)_\alpha,\\
\end{gathered}
\end{equation*}
then we have that
\begin{equation*}
\begin{gathered}
X_\alpha=-4(\mathcal N_2)_\alpha,\qquad
Y_\alpha=4\Big(i\pi^{\bar\sigma}_\alpha(\mathcal N_1)_\alpha-(\mathcal N_3)_\alpha\Big),\\
Z_\alpha=4(\mathcal N_5)_\alpha,\qquad
W_\alpha=-4i\pi^{\bar\sigma}_\alpha\Big((\mathcal N_2)_{\bar\sigma}-(\mathcal N_4)_{\bar\sigma}\Big),\\
U_\alpha=4\pi^{\bar\sigma}_\alpha(\mathcal N_3)_{\bar\sigma},\qquad
\mathcal R_1=i\big(\mathcal U_3-\overline{\mathcal U_3}\big),\qquad 
\mathcal R_2=-i\big(\mathcal U_1+\mathcal W_3\big),\\
(\mathcal H_3)_\alpha=(\mathcal N_1)_\alpha +i\pi^{\bar\sigma}_\alpha(\mathcal N_3)_{\bar\sigma}.
\end{gathered}
\end{equation*}
This completes the proof of Proposition~\ref{sec_der}.
\end{proof}

\section{The associated Cartan geometry}\label{sec5}

Our next goal is to check that the construction of the canonical coframe 
from Theorem \ref{Theorem_1} coincides with the general normalization used
for all parabolic geometries (and explained briefly in the appendix). 

First we compare the structure equations from the Proposition \ref{curvature}
with those of the homogeneous model $G\to G/P$. This verifies that our
coframe lives on the principal fibre bundle with the right structure group. 

Next, we express the Kostant's codifferential on the cochains explicitly,
and we obtain that indeed, the curvature components from Proposition 
\ref{curvature} are normalized in the canonical way.

\subsection{A few algebraic constructions}\label{algebra_pre}
Consider the standard action of the group $Sp(n+1,1)$ on $\mathbb R^{4n+8}$ defined by some (fixed) identification $\mathbb R^{4n+8}\cong \mathbb H^{n+2}$. Let $J_1,J_2,J_3$ be the induced invariant quaternionic structure on $\mathbb R^{4n+8}$ and let $\lc,\rc$ be the corresponding inner product of signature $\big(+(4n+4),-4\big)$. The complexification $\mathbb C^{4n+8}$ of  $\mathbb R^{4n+8}$ splits as a direct sum of $i$ and $-i$ eigenspaces with respect to the complex structure $J_1$,
\begin{equation*}
\mathbb C^{4n+8}=W\oplus\overline{W}.
\end{equation*}

Let us fix a basis $\{\mathfrak v_1,\mathfrak v_2,\mathfrak e_\alpha,\mathfrak w_1,\mathfrak w_2\}$ of $W$ for which
\begin{equation}\label{basis_W_first}
J_2(\mathfrak v_1)=\overline{\mathfrak v_2},\qquad J_2(\mathfrak e_\alpha)=\pi^{\bar\beta}_\alpha\mathfrak e_{\bar\beta},\qquad J_2(\mathfrak w_1)=\overline{\mathfrak w_2}
\end{equation}
and
\begin{equation}\label{basis_W_second}
\begin{aligned}
\lc \mathfrak v_1, \overline{\mathfrak v_1} \rc=&0&\qquad \lc \mathfrak v_1, \overline{\mathfrak v_2} \rc=&0&\qquad \lc \mathfrak v_1, \mathfrak e_{\bar\alpha} \rc=&0&\qquad \lc \mathfrak v_1, \overline{\mathfrak w_1} \rc=&1&\qquad \lc \mathfrak v_1, \overline{\mathfrak w_2} \rc=&0&\\
\lc \mathfrak v_2, \overline{\mathfrak v_1} \rc=&0&\qquad \lc \mathfrak v_2, \overline{\mathfrak v_2} \rc=&0&\qquad \lc \mathfrak v_2, \mathfrak e_{\bar\alpha} \rc=&0&\qquad \lc \mathfrak v_2, \overline{\mathfrak w_1} \rc=&0&\qquad \lc \mathfrak v_2, \overline{\mathfrak w_2} \rc=&1&\\
\lc \mathfrak e_\alpha, \overline{\mathfrak v_1} \rc=&0&\qquad \lc \mathfrak e_\alpha, \overline{\mathfrak v_2} \rc=&0&\qquad \lc \mathfrak e_\alpha, \mathfrak e_{\bar\beta} \rc=&g_{\alpha\bar\beta}&\qquad \lc \mathfrak e_\alpha, \overline{\mathfrak w_1} \rc=&0&\qquad \lc \mathfrak e_\alpha, \overline{\mathfrak w_2} \rc=&0&\\
\lc \mathfrak w_1, \overline{\mathfrak v_1} \rc=&1&\qquad \lc \mathfrak w_1, \overline{\mathfrak v_2} \rc=&0&\qquad \lc \mathfrak w_1, \mathfrak e_{\bar\alpha} \rc=&0&\qquad \lc \mathfrak w_1, \overline{\mathfrak w_1} \rc=&0&\qquad \lc \mathfrak w_1, \overline{\mathfrak w_2} \rc=&0&\\
\lc \mathfrak w_2, \overline{\mathfrak v_1} \rc=&0&\qquad \lc \mathfrak w_2, \overline{\mathfrak v_2} \rc=&1&\qquad \lc \mathfrak w_2, \mathfrak e_{\bar\alpha} \rc=&0&\qquad \lc \mathfrak w_2, \overline{\mathfrak w_1} \rc=&0&\qquad \lc \mathfrak w_2, \overline{\mathfrak w_2} \rc=&0.&
\end{aligned}
\end{equation}
The group $Sp(n+1,1)$ consists of all endomorphisms of $W$ that take $\{\mathfrak v_1,\mathfrak v_2,\mathfrak e_\alpha,\mathfrak w_1,\mathfrak w_2\}$ into a bases  with the same properties \eqref{basis_W_first}, \eqref{basis_W_second}. By differentiating these at the identity, we obtain the Lie algebra $\mathfrak g=sp(n+1,1)$ as the set of all matrices of the form
\begin{equation}\label{sp_n11}
\begin{pmatrix}
 -\frac{1}{2}(\varphi_0+i\varphi_1)&-\frac{1}{2}(\varphi_2-i\varphi_3)&2ig_{\beta\bar\sigma}\,\phi^{\bar\sigma}& i\psi_1&(\psi_2-i\psi_3)\\
\frac{1}{2}(\varphi_2+i\varphi_3)&-\frac{1}{2}(\varphi_0-i\varphi_1)&2i\pi_{\beta\sigma}\,\phi^{\sigma}& -(\psi_2+i\psi_3)&-i\psi_1\\
 i\theta^\alpha&-i\pi^\alpha_{\bar\sigma}\,\theta^{\bar\sigma}&\pi^{\alpha\sigma}\,\Gamma_{\sigma\beta}&2i\phi^\alpha&-2i\pi^\alpha_{\bar\sigma}\,\phi^{\bar\sigma}\\
  \frac{i}{2}\eta_1&\frac{1}{2}(\eta_2-i\eta_3)&ig_{\beta\bar\sigma}\,\theta^{\bar\sigma}&  \frac{1}{2}(\varphi_0-i\varphi_1)&-\frac{1}{2}(\varphi_2-i\varphi_3)\\
  -\frac{1}{2}(\eta_2+i\eta_3)&-\frac{i}{2}\eta_1&i\pi_{\beta\sigma}\,\theta^{\sigma}&\frac{1}{2}(\varphi_2+i\varphi_3)&\frac{1}{2}(\varphi_0+i\varphi_1)\\
\end{pmatrix},
\end{equation}
where $\eta_s,\varphi_s,\psi_s$ are real, and $\theta^\alpha,\phi^\alpha,\Gamma_{\alpha\beta}$ are complex so that
\begin{equation*}
\Gamma_{\alpha\beta}=\Gamma_{\beta\alpha},\qquad (\mathfrak j\Gamma)_{\alpha\beta}=\Gamma_{\alpha\beta}.
\end{equation*}

We may interpret $\eta_s,\theta^\alpha,\varphi_0,\varphi_s,\Gamma_{\alpha\beta},\phi^\alpha,\psi_s$ as left-invariant one-forms on the Lie group $Sp(n+1,1)$. We immediately derive (by using just matrix multiplication) the structure equations of the group:  
\begin{equation}\label{str-eq-con-sp}
\begin{gathered}
d\eta_1=-\varphi_0\wedge\eta_1-\varphi_2\wedge\eta_3+\varphi_3\wedge\eta_2+2i g_{\alpha\bar\beta}\,\theta^{\alpha}\wedge\theta^{\bar\beta}\\
d\eta_2=-\varphi_0\wedge\eta_2-\varphi_3\wedge\eta_1+\varphi_1\wedge\eta_3+\pi_{\alpha\beta}\,\theta^{\alpha}\wedge\theta^{\beta}+\pi_{\bar\alpha\bar\beta}\,\theta^{\bar\alpha}\wedge\theta^{\bar\beta}\\
d\eta_3=-\varphi_0\wedge\eta_3-\varphi_1\wedge\eta_2 + \varphi_2\wedge\eta_1-i\pi_{\alpha\beta}\,\theta^{\alpha}\wedge\theta^{\beta}+i\pi_{\bar\alpha\bar\beta}\,\theta^{\bar\alpha}\wedge\theta^{\bar\beta}\\
d\theta^\alpha=-i\phi^\alpha\wedge\eta_1-\pi^\alpha_{\bar\sigma}\phi^{\bar\sigma}\wedge(\eta_2+i\eta_3)-\pi^{\alpha\sigma}\Gamma_{\sigma\beta}\wedge\theta^\beta-\frac{1}{2}(\varphi_0+i\varphi_1)\wedge\theta^\alpha-\frac{1}{2}\pi^\alpha_{\bar\beta}(\varphi_2+i\varphi_3)\wedge\theta^{\bar\beta}\\
d\varphi_0=-\psi_1\wedge\eta_1-\psi_2\wedge\eta_2-\psi_3\wedge\eta_3-2\phi_\beta\wedge\theta^\beta-2\phi_{\bar\beta}\wedge\theta^{\bar\beta}\\
d\varphi_1=-\varphi_2\wedge\varphi_3-\psi_2\wedge\eta_3+\psi_3\wedge\eta_2+2i\phi_\beta\wedge\theta^\beta-2i\phi_{\bar\beta}\wedge\theta^{\bar\beta}\\
d\varphi_2=-\varphi_3\wedge\varphi_1-\psi_3\wedge\eta_1+\psi_1\wedge\eta_3-2\pi_{\sigma_\beta}\phi^\sigma\wedge\theta^\beta-2\pi_{\bar\sigma\bar\beta}\phi^{\bar\sigma}\wedge\theta^{\bar\beta}\\
d\varphi_3=-\varphi_1\wedge\varphi_2-\psi_1\wedge\eta_2+\psi_2\wedge\eta_1+2i\pi_{\sigma_\beta}\phi^\sigma\wedge\theta^\beta-2i\pi_{\bar\sigma\bar\beta}\phi^{\bar\sigma}\wedge\theta^{\bar\beta}\\
d\Gamma_{\alpha\beta}\ =\ -\pi^{\sigma\tau}\Gamma_{\alpha\sigma}\wedge\Gamma_{\tau\beta} + 2\pi^{\bar\sigma}_{\alpha}(\phi_\beta\wedge\theta_{\bar\sigma}-\phi_{\bar\sigma}\wedge\theta_\beta)+2\pi^{\bar\sigma}_{\beta}(\phi_\alpha\wedge\theta_{\bar\sigma}-\phi_{\bar\sigma}\wedge\theta_\alpha)\\
d\phi^{\alpha}=\frac{1}{2}(\varphi_0-i\varphi_1)\wedge\phi^\alpha 
 -\frac{1}{2}\pi^\alpha_{\bar\gamma}(\varphi_2+i\varphi_3)\wedge\phi^{\bar\gamma}
 -\pi^{\alpha\sigma}\, \Gamma_{\sigma\gamma}\wedge\phi^{\gamma}
 +\frac{i}{2}\,\psi_1\wedge\theta^\alpha+\frac{1}{2}\,\pi^\alpha_{\bar\gamma}(\psi_2+i\psi_3)\wedge\theta^{\bar\gamma}\\
 d\psi_1\ =\ \varphi_0\wedge\psi_1-\varphi_2\wedge\psi_3+\varphi_3\wedge\psi_2-4i \phi_\gamma\wedge\phi^\gamma\\
 d\psi_2+i\,d\psi_3\ =\ (\varphi_0-i\varphi_1)\wedge(\psi_2+i\psi_3)+i(\varphi_2+i\varphi_3)\wedge\psi_1+4\pi_{\gamma\delta}\phi^\gamma\wedge\phi^\delta.
\end{gathered}
\end{equation}
Notice that the equations \eqref{str-eq-con-sp} are formally identical with the corresponding structure equations \eqref{str-eq-deta}, \eqref{str-eq-con}, \eqref{dGamma_ab}, \eqref{dphi_a},   \eqref{dpsi_1}, \eqref{dpsi_23} for the global coframing on $P_1$ constructed in Theorem~\ref{Theorem_1}, if assuming that all curvature components vanish, 
\begin{equation}\label{FlatEq}
\mathcal S_{\alpha\beta\gamma\delta}=\mathcal V_{\alpha\beta\gamma}=\mathcal L_{\alpha\beta}=\mathcal M_{\alpha\beta}=\mathcal C_\alpha=\mathcal H_\alpha=\mathcal P=\mathcal Q=\mathcal R=0.
\end{equation} 

The Killing form of the Lie algebra $\mathfrak g=sp(n+1,1)$, which we will denote by $\mathbb B$, is defined as $\mathbb B(A,B)={trace}\big(C\mapsto [A,[B,C]]\big)$, $A,B\in \mathfrak g$. Using the above notation, we compute 
\begin{multline}\label{Killing_form}
\mathbb B(A,B)=-(4n+6)\Big(\eta_s(A)\psi_s(B)+\psi_s(A)\eta_s(B)\Big) +(2n+6)\varphi_0(A)\varphi_0(B)-(2n+4)\varphi_s(A)\varphi_s(B)\\
-4(2n+7)\Big(\theta_\alpha(A)\phi^\alpha(B)+\phi^\alpha(A)\theta_\alpha(B) +\theta_{\bar\alpha}(A)\phi^{\bar\alpha}(B)+\phi^{\bar\alpha}(A)\theta_{\bar\alpha}(B)\Big)-7\,\Gamma_{\alpha\beta}(A)\Gamma^{\alpha\beta}(B)
\end{multline}
Notice that the sum $\Gamma_{\alpha\beta}(A)\,\Gamma^{\alpha\beta}(B)$ produces always a real number, since we have
\begin{equation*}
\Gamma_{\alpha\beta}(A)\,\Gamma^{\alpha\beta}(B)=(\mathfrak j\Gamma)_{\alpha\beta}(A)\,(\mathfrak j\Gamma)^{\alpha\beta}(B)=\Big(\pi_\alpha^{\bar\sigma}\pi_\beta^{\bar\tau}\Gamma_{\bar\sigma\bar\tau}(A)\Big)\Big(\pi^\alpha_{\bar\mu}\pi^\beta_{\bar\nu}\Gamma^{\bar\mu\bar\nu}(B)\Big)=\Gamma_{\bar\alpha\bar\beta}(A)\,\Gamma^{\bar\alpha\bar\beta}(B).
\end{equation*}

Furthermore, the Lie algebra $\mathfrak g$ has a splitting (which is also a $|2|$-grading) 
\begin{equation*}
\mathfrak g=\mathfrak g_{-2}\oplus \mathfrak g_{-1}\oplus \underbrace{\mathbb R\oplus sp(1)\oplus sp(n)}_{\mathfrak g_0}\oplus \mathfrak g_{1} \oplus \mathfrak g_{2},
\end{equation*} 
that dualizes the splitting 
 \begin{equation*}
 \{\eta_s\},\ \{\theta^\alpha\}, \ \{\varphi_0\},\ \{\varphi_s\},\  \{\Gamma_{\alpha\beta}\},\ \{\phi^\alpha\},\ \{\psi_s\}
 \end{equation*}
of the left-invariant one-forms. Let 
\begin{equation}\label{frame_E_s}
\{E_s\in \mathfrak g_{-2}\},\{ Z_\alpha, Z_{\bar\alpha} \in \mathfrak g_{-1}\}
\end{equation}
 be a frame (of the complexification) of $\mathfrak g_{-}\overset{def}{=}\mathfrak g_{-2}\oplus \mathfrak g_{-1}$  dual to the coframe $\{\eta_s\}, \{\theta^\alpha,\theta^{\bar\alpha}\},$ i.e. such that
\begin{equation*}
\eta_s(E_t)=\delta_{st},\qquad \theta^\alpha(Z_\beta)=\delta^{\alpha}_\beta, \qquad{\theta^{\bar\alpha}}(Z_\beta)=0, \qquad Z_{\bar\beta}=\overline{Z_\beta},
\end{equation*}
 and let $\{\hat E_s\in \mathfrak g_{2}\},\{ \hat Z^\alpha, \hat Z^{\bar\alpha} \in \mathfrak g_{1}\}$ be the corresponding  frame 
of $\mathfrak g_{1}\oplus \mathfrak g_{2}$ dual to \eqref{frame_E_s} with respect to the Killing form $\mathbb B$, i.e. such that
\begin{equation*}
\mathbb B (E_s,\hat E_t)=\delta_{st},\qquad
\mathbb B (Z_\alpha, \hat Z^\beta)=\delta_{\alpha}^{\beta},\qquad \mathbb B (Z_\alpha, \hat Z^\beta)=0,\qquad \hat Z^{\bar \alpha} = \overline{\hat Z^\alpha}.
\end{equation*}
Then,  the map $\partial^\ast:\Lambda^2( \mathfrak g_-)^\ast\otimes \mathfrak g\longrightarrow  (\mathfrak g_-)^\ast\otimes \mathfrak g$, given, for any $A\in\mathfrak g_-$, by
\begin{multline}\label{def_d_ast}
(\partial^\ast K)(A)=2\big[\hat E_s,K(A,E_s)\big]+2\big[\hat Z^\alpha,K(A,Z_\alpha)\big]+2\big[\hat Z^{\bar\alpha},K(A,Z_{\bar\alpha})\big]\\
-K\big([\hat E_s,A]_-, E_s\big)-K\big([\hat Z^\alpha,A]_-, Z_\alpha\big)-K\big([\hat Z^{\bar\alpha}, A]_-, Z_{\bar\alpha}\big),
\end{multline}
where  $X_-$ denotes the projection of $X\in \mathfrak g$ onto $\mathfrak g_-$,
is known as the Kostant codifferential (cf. \cite{CS}, p. 261 or \cite{Yam}, p. 468).

\begin{lemma}\label{lemma-costnat-codif}
If $K\in\Lambda^2( \mathfrak g_-)^\ast\otimes \Big(\mathfrak sp(n)\oplus \mathfrak g_1\oplus \mathfrak g_2 \Big)\subset \Lambda^2( \mathfrak g_-)^\ast\otimes \mathfrak g$, then for any $A\in \mathfrak g_-$, we have
\begin{multline*}
(\partial^\ast K)(A)=\frac{1}{4(2n+7)}\Bigg(i\Big(K(Z^\alpha,Z_\alpha)-K(Z^{\bar\alpha},Z_{\bar\alpha})\Big)\eta_1(A)-\pi^{\alpha\beta}\,K(Z_\alpha, Z_\beta)\Big(\eta_2(A)+i\eta_3(A)\Big)\\
-
\pi^{\bar\alpha\bar\beta}\,K(Z_{\bar\alpha}, Z_{\bar\beta})\Big(\eta_2(A)-i\eta_3(A)\Big)\Bigg)
-2\pi^{\beta\sigma}\Gamma^{\bar\alpha}_\sigma \Big(K(A,Z_{\bar\alpha})\Big)\hat Z_{\beta}
-2\pi^{\bar\beta\bar\sigma}\Gamma^{\alpha}_{\bar\sigma}\Big(K(A,Z_{\alpha})\Big)\hat Z_{\bar\beta}\\
+\frac{i4(2n+3)}{2n+7}\Bigg(\phi^\alpha\Big(K(A,Z_\alpha)\Big)-\phi^{\bar\alpha}\Big(K(A,Z_{\bar\alpha})\Big)\Bigg)\hat E_1
+\frac{4(2n+3)}{2n+7}\,\pi^\alpha_{\bar\sigma}\, \phi^{\bar\sigma}\Big(K(A,Z_\alpha)\Big)\Big(\hat E_2+i\hat E_3\Big)\\
+\frac{4(2n+3)}{2n+7}\,\pi^{\bar\alpha}_{\sigma}\, \phi^{\sigma}\Big(K(A,Z_{\bar\alpha})\Big)\Big(\hat E_2-i\hat E_3\Big)
\end{multline*}
\end{lemma}
\begin{proof}
By \eqref{Killing_form}, it follows that 
\begin{equation*}
\psi_s(\hat E_t)=-\frac{\delta_{st}}{4n+6},\qquad \phi_\alpha(\hat Z^\beta)=-\frac{\delta_\alpha^\beta}{4(2n+7)},\qquad \phi_{\bar\alpha}(\hat Z^\beta)=0.
\end{equation*} 

Let $A\in \mathfrak g_{-}$ and $B\in sp(n)\oplus\mathfrak g_1\oplus\mathfrak g_{2}$  be any two matrices. Then, using the structure equations \eqref{str-eq-con-sp} of $\mathfrak g$, we compute that
\begin{equation}\label{com_with_B}
\begin{aligned}
\big[\hat E_s,B\big] \ =\  &-(4n+6)\psi_t\big([\hat E_s,B]\big)\hat E_t=(4n+6){d\psi_t(\hat E_s, B)}\hat E_t=0,\\
\big[\hat Z^\alpha,B\big]\ =\ &-4(2n+7)\,{\phi^\beta\big([\hat Z^\alpha,B]\big)}\hat Z_\beta -4(2n+7)\,{\phi^{\bar\beta}\big([\hat Z^\alpha,B]\big)}\hat Z_{\bar\beta}-(4n+6)\,{\psi_s\big([\hat Z^\alpha,B]\big)}\hat E_s\\
&=4(2n+7)\,{d\phi^\beta(\hat Z^\alpha,B)}\hat Z_\beta +4(2n+7)\,{d\phi^{\bar\beta}(\hat Z^\alpha,B)}\hat Z_{\bar\beta}+(4n+6)\,{d\psi_s(\hat Z^\alpha,B)}\hat E_s\\
&=-{g^{\alpha\bar\tau}\,\pi^{\bar\beta\bar\sigma}}\,\Gamma_{\bar\sigma\bar\tau} (B) \hat Z_{\bar\beta}
+\frac{i(4n+6)}{2n+7}\,\phi^\alpha(B)\hat E_1 +\frac{(4n+6)}{2n+7}\,\pi^\alpha_{\bar\sigma}\,\phi^{\bar\sigma}(B)\big(\hat E_2+i\hat E_3\big),\\
\big[\hat Z^{\bar\alpha},B\big]\ =\  &-{g^{\bar\alpha\tau}\,\pi^{\beta\sigma}}\,\Gamma_{\sigma\tau} (B) \hat Z_{\beta}
-\frac{i(4n+6)}{2n+7}\,\phi^{\bar\alpha}(B)\hat E_1 +\frac{(4n+6)}{2n+7}\,\pi^{\bar\alpha}_{\sigma}\,\phi^{\sigma}(B)\big(\hat E_2-i\hat E_3\big),
\end{aligned}
\end{equation}
and, similarly,
\begin{equation}\label{com_with_A}
\begin{aligned}
\big[\hat E_s,A\big]_- \ =\  &0,\\
\big[\hat Z^\alpha,A\big]_-\ =\ &{\theta^\beta\big([\hat Z^\alpha,A]\big)} Z_\beta +{\theta^{\bar\beta}\big([\hat Z^\alpha,A]\big)} Z_{\bar\beta}+{\eta_s\big([\hat Z^\alpha,B]\big)} E_s\\
&=-{d\theta^\beta(\hat Z^\alpha,A)} Z_\beta -{d\theta^{\bar\beta}(\hat Z^\alpha,A)} Z_{\bar\beta}-{d\eta_s(\hat Z^\alpha,A)} E_s\\
&=-\frac{\pi^{\alpha\beta}\big(\eta_2(A)+i\eta_3(A)\big)}{4(2n+7)}\, Z_{\beta}+\frac{ig^{\alpha\bar\beta}\eta_1(A)}{4(2n+7)}\, Z_{\bar\beta},\\
\big[\hat Z^{\bar\alpha}, A\big]_-\ =\ &-\frac{\pi^{\bar\alpha\bar\beta}\big(\eta_2(A)-i\eta_3(A)\big)}{4(2n+7)}\, Z_{\bar\beta}-\frac{ig^{\bar\alpha\beta}\eta_1(A)}{4(2n+7)}\, Z_{\beta}.
\end{aligned}
\end{equation}
Substituting \eqref{com_with_B} and \eqref{com_with_A} into \eqref{def_d_ast} gives the lemma.

\end{proof}

The subalgebra $\mathfrak p=\mathfrak g_0\oplus \mathfrak g_1\oplus \mathfrak g_2\subset \mathfrak g$ determines a parabolic subgroup  $\mathfrak P\subset Sp(n+1,1)$, which, alternatively, can be described as the stabilizer of the complex 2-plane $\text{span}\{\mathfrak v_1,\mathfrak v_2\}\subset W$  in $Sp(n+1,1)$. Explicitly, $\mathfrak P$ consists of all matrices of the form
\begin{equation}\label{parabolic_P}
\left(\begin{array}{c|c|c}
A & -A \left(\begin{array}{c}
                 U_{\beta\bar\sigma}\,r^{\bar\sigma} \\ 
                 U_{\beta\bar\tau}\,\pi^{\bar\tau}_{\sigma}\,r^\sigma\\
                 \end{array}\right)	 & A\left(\begin{array}{cc}
                                           -\frac{1}{2}r_\sigma\,r^\sigma+i\lambda_1 & -\lambda_2+i\lambda_3\\
                                           \lambda_2+i\lambda_3 & -\frac{1}{2}r_\sigma\,r^\sigma-i\lambda_1\\ 
                                           \end{array}\right)\\[4mm]
\hline
\begin{array}{cc}
       \vdots & \vdots\\
       0 &  0\\
       \vdots & \vdots\\
      \end{array} & U^\alpha_\beta & \begin{array}{cc}
                                         \vdots \qquad&\qquad \vdots\\
                                         r^\alpha \qquad&\qquad \pi^\alpha_{\bar\sigma}\,r^{\bar\sigma}\\
                                         \vdots \qquad&\qquad \vdots\\
                                            \end{array}\\  [4mm]                                        
\hline
\begin{array}{cc}
0 &0\\
0&0\\
\end{array} & \begin{array}{c}
                  \dots \ 0 \ \dots \\ 
                \dots\  0\ \dots\\
                 \end{array} & \frac{1}{det A}A\\
\end{array}\right),
\end{equation}
where
$$A\in CSp(1)=\left\{\left(\begin{array}{cc}
                             a_1 & -\overline{a_2}\\
                             a_2 & \overline {a_1}\\
                             \end{array}\right)\ :\ 
a_1,\ a_2\in \mathbb C,\ a_1^2+a_2^2\ne 0\right\},
$$
$U=(U^\alpha_\beta)\in Sp(n)$, $r=(r^\alpha)\in \mathbb C^{2n}$ and $\lambda_1,\lambda_2,\lambda_3$ are real numbers.

In general, a qc structure whose curvature components, given by Proposition~\ref{curvature}, satisfy \eqref{FlatEq} is called flat. Conversely, it follows from the theorem of Frobenius that every flat qc structure is locally equivalent to the qc structure of the homogeneous model $Sp(n+1,1)/\mathfrak P$ (i.e., the Sphere with its standard qc structure). It is a well known result from the general theory of parabolic geometries \cite{CS} that actually the vanishing of $\mathcal S_{\alpha\beta\gamma\delta}$ alone is enough to ensure the flatness of the qc structure, since the remaining curvature components can be expressed as components of the covariant derivatives of $\mathcal S_{\alpha\beta\gamma\delta}$; the precise formulas for this are given here by Proposition~\ref{sec_der}. 
The functions $\mathcal S_{\alpha\beta\gamma\delta}$ represent the so called harmonic curvature of the qc structure and it is well known \cite{CS} that this part of the curvature can be pushed down to produce a tensor field on the base manifold $M$. In general it is, however, a highly nontrivial task to obtain an explicit expression for this tensor field on the base manifold. For the qc case, Ivanov and Vassilev have obtained such an expression \cite{IV} in terms of the so called Biquard connection and its differential invariants.

\subsection{The normal Cartan connection}\label{normal_Cartan}
Let (M,H) be a quaternionic contact manifold and take $pr: P_1\rightarrow M$ to be the composition of the two principal bundle projections
\begin{equation*}
P_1\overset{\pi_1}{\rightarrow} P_0 \overset{\pi_0}{\rightarrow} M,
\end{equation*}
as constructed in Section~\ref{sec_eq_prob}. We can consider the global coframing
\begin{equation}\label{glob_cofr}
\{\eta_1,\eta_2,\eta_2,\theta^\alpha,\theta^{\bar\alpha},\varphi_0,\varphi_1,\varphi_2,\varphi_3\}\cup\{\Gamma_{\alpha\beta} : \alpha\le\beta\}\cup\{\phi^\alpha,\phi^{\bar\alpha},\psi_1,\psi_2,\psi_3\}
\end{equation}
of the (complexified) tangent bundle $TP_1$ (cf. Theorem~\ref{Theorem_1})
as a map $\omega:TP_1\rightarrow \mathfrak g$ ($\mathfrak g=sp(n+1,1)$), by declaring $\omega=\Omega(\eta_s,\theta^\alpha,\varphi_0,\varphi_s,\Gamma_{\alpha\beta},\phi^\alpha,\psi_s)$ to be the matrix given by formula \eqref{sp_n11}. 

Recall that the local sections of $P_1$ are precisely the local coframings $\eta_s,\theta^\alpha,\theta^{\bar\alpha},\varphi_0,\varphi_s$ on $TP_o$ for which the assertion of Lemma~\ref{Lex} is satisfied, and observe that the equations in the lemma coincide with the first three of the structure equations \eqref{str-eq-con-sp} for the corresponding left-invariant one-forms on $Sp(n+1,1).$ Since the adjoint action of $\mathfrak P$ (cf. \eqref{parabolic_P}) on $\mathfrak g$ preserves \eqref{str-eq-con-sp}, we can use it to define a natural action of $\mathfrak P$  on the manifold $P_1$ that will preserve the fibers of the projection $pr$. In fact, one can show that $pr: P_1\rightarrow M$ is a principle bundle with structure group $\mathfrak P/ \mathbb Z_2$. Moreover, the uniqueness part of Theorem~\ref{Theorem_1} ensures that the $\mathfrak g$-valued form $\omega$ on $P_1$ will be $\mathfrak P$ - equivariant and therefore it gives a Cartan connection on $pr: P_1\rightarrow M$.   

The curvature of the Cartan connection $\omega$ is a function $K\in C^{\infty}\Big(P_1,\Lambda^2(\mathfrak g_-)^\ast\otimes\big(sp(n)\oplus\mathfrak g_1\oplus \mathfrak g_2\big)\Big)$ which, by \eqref{dGamma_ab}, \eqref{dphi_a}, \eqref{dpsi_1}, \eqref{dpsi_23} and \eqref{str-eq-con-sp}, is given by

 \begin{equation*}
 \begin{split}
 \Gamma_{\alpha\beta}(K)\ = \ \ &\pi^{\sigma}_{\bar\delta}\,\mathcal S_{\alpha\beta\gamma\sigma}\,\theta^{\gamma}\wedge\theta^{\bar\delta}+\Big(\mathcal V_{\alpha\beta\gamma}\,\theta^\gamma
 +\pi^{\bar\sigma}_{\alpha}\,\pi^{\bar\tau}_{\beta}\,\mathcal V_{\bar\sigma\bar\tau\bar\gamma}\,\theta^{\bar\gamma}\Big)\wedge\eta_1\\
 &-i\pi^{\sigma}_{\bar\gamma}\,\mathcal V_{\alpha\beta\sigma}\,\theta^{\bar\gamma}\wedge(\eta_2+i\eta_3)+i(\mathfrak j\mathcal V)_{\alpha\beta\gamma}\,\theta^{\gamma}\wedge(\eta_2-i\eta_3)\\
 &-i\mathcal L_{\alpha\beta}\,(\eta_2+i\eta_3)\wedge(\eta_2-i\eta_3)+\mathcal M_{\alpha\beta}\,\eta_1\wedge(\eta_2+i\eta_3)\\
 +\ &(\mathfrak j M)_{\alpha\beta}\,\eta_1\wedge(\eta_2-i\eta_3),\\
\phi_{\alpha}(K)\ = -&i\pi_{\bar\delta}^{\sigma}\,\mathcal V_{\alpha\gamma\sigma}\,\theta^{\gamma}\wedge\theta^{\bar\delta}+\mathcal M_{\alpha\gamma}\,\theta^\gamma\wedge\eta_1+\pi^{\bar\sigma}_{\alpha}\,\mathcal L_{\bar\sigma\bar\gamma}\,\theta^{\bar\gamma}\wedge\eta_1\\
 +\ &i\mathcal L_{\alpha\gamma}\,\theta^{\gamma}\wedge(\eta_2-i\eta_3)-i\pi^{\sigma}_{\bar\gamma}\,\mathcal M_{\alpha\sigma}\,\theta^{\bar\gamma}\wedge(\eta_2+i\eta_3)
 -\mathcal C_{\alpha}(\eta_2+i\eta_3)\wedge(\eta_2-i\eta_3)\\
 +\ &\mathcal H_\alpha\,\eta_1\wedge(\eta_2+i\eta_3)+i\pi_{\alpha\sigma}\,\mathcal C^\sigma\,\eta_1\wedge(\eta_2-i\eta_3),\\
 \psi_1(K)\ =\ \ &4\pi^{\sigma}_{\bar\delta}\,\mathcal L_{\gamma\sigma}\,\theta^\gamma\wedge\theta^{\bar\delta}+4\mathcal C_\gamma\,\theta^\gamma\wedge\eta_1
 + 4\mathcal C_{\bar\gamma}\,\theta^{\bar\gamma}\wedge\eta_1-4i\pi_{\bar\gamma\bar\sigma}\, \mathcal C^{\bar\sigma}\,\theta^{\bar\gamma}\wedge(\eta_2+i\eta_3)\\
 +\ &4i\pi_{\gamma\sigma}\, \mathcal C^{\sigma}\,\theta^{\gamma}\wedge(\eta_2-i\eta_3)
 + \mathcal P\,\eta_1\wedge(\eta_2+i\eta_3)+\overline{\mathcal P}\,\eta_1\wedge(\eta_2-i\eta_3)\\
 +\ &i\mathcal R\,(\eta_2+i\eta_3)\wedge(\eta_2-i\eta_3),\\
 \psi_2(K)+i\,\psi_3(K)\ =\ \ &4i\pi^{\bar\sigma}_{\gamma}\,\mathcal M_{\bar\sigma\bar\delta}\,\theta^\gamma\wedge\theta^{\bar\delta}
 + 4i \pi^{\bar\sigma}_{\gamma}\,\mathcal C_{\bar\sigma}\,\theta^\gamma\wedge\eta_1
 -4\mathcal H_{\bar\gamma}\,\theta^{\bar\gamma}\wedge\eta_1-4i\mathcal C_{\bar\gamma}\,\theta^{\bar\gamma}\wedge(\eta_2+i\eta_3)\\
 -\ &4i\pi_{\gamma}^{\bar\sigma}\, \mathcal H_{\bar\sigma}\,\theta^{\gamma}\wedge(\eta_2-i\eta_3)
 - i\mathcal R\,\eta_1\wedge(\eta_2+i\eta_3)+\overline{\mathcal Q}\,\eta_1\wedge(\eta_2-i\eta_3)\\
 -\ &\overline{\mathcal P}\,(\eta_2+i\eta_3)\wedge(\eta_2-i\eta_3).
 \end{split}
 \end{equation*}

The properties of the curvature components (cf. Proposition~\ref{curvature})
\begin{equation*}
 \mathcal S_{\alpha\beta\gamma\delta},\  \mathcal V_{\alpha\beta\gamma},\ \mathcal  L_{\alpha\beta},\ \mathcal  M_{\alpha\beta}, \ \mathcal C_\alpha, \  \mathcal H_\alpha,\  \mathcal P, \  \mathcal Q,\  \mathcal R
\end{equation*}
imply that
\begin{gather*}
g^{\alpha\bar\beta}K(Z_\alpha, Z_{\bar\beta})=0,\qquad \pi^{\alpha\beta}K(Z_\alpha, Z_{\beta})=0,\qquad \Gamma_{\alpha\beta}\Big(K(Z^\beta, .)\Big)=0,\\
 \phi_{\alpha}\Big(K(Z^\alpha, .)\Big)=0, \qquad \pi^{\alpha\beta}\phi_{\alpha}\Big(K(Z_\beta, .)\Big)=0,
\end{gather*}
and therefore, by Lemma~\eqref{lemma-costnat-codif}, we have $\partial^\ast K =0.$ Thus $(P_1,\omega)$ coincides with the regular, normal Cartan geometry associated with the quaternionic contact manifold $(M,H)$. 

\section{Appendix}
This section serves as a brief collection of basic facts on Cartan
geometries. All the details and much more information can be found in the
book \cite{CS}. At the same time, we provide links to the general structure
theory to our computations.

\subsection{Cartan geometries}
Elie Cartan's generalized spaces (espace generalis\'e) are 
curved analogs of the homogeneous space $G/P$ for Lie groups $P\subset G$.
They are defined as right invariant absolute parallelism $\omega$ 
on a principal $P$--bundle $\mathcal
G$ reproducing the fundamental vector fields. Let us write 
$\mathfrak g$ and $\mathfrak p$ for the Lie algebras lof $G$ and $P$,
respectively.

A \emph{Cartan geometry $(\mathcal G,\omega)$
of type $G/P$} is a principal fiber bundle $\mathcal G$ with structure group
$P$, equipped with a smooth one-form $\omega\in\Omega^1(\mathcal G,\mathfrak g)$ 
satisfying
\begin{enumerate}
\item $\omega(\zeta_Z)(u)=Z$ for all $u\in\mathcal G$ and 
fundamental fields $\zeta_Z$,
$Z\in \mathfrak p$
\item $(r^p)^*\omega = \operatorname{Ad}(p^{-1})\o \omega$ for all $\in P$
\item $\omega|_{T_u\mathcal G}: T_u\mathcal G\to \mathfrak g$ is a 
linear isomorphism for
all $u\in \mathcal G$.
\end{enumerate} 
In particular, each $X\in \mathfrak g$ defines the {\em constant vector field}
$\omega^{-1}(X)$ defined by $\omega(\omega^{-1}(X)(u))=X$, $u\in \mathcal G$. 
The one forms with latter three properties are called {\em Cartan connections}. 

The homogeneous model $G\to G/P$ together 
with the Maurer Cartan form $\omega$ is an example of such geometry.

The morphisms between parabolic geometries $(\mathcal G,\omega)$ and
$(\mathcal G',\omega')$ are
principal fiber bundle morphisms $\phi$ which
preserve the Cartan connections, i.e. $\phi:\mathcal G\to \mathcal G'$ and
$\phi^*\omega'=\omega$. 

The structure equations $d\omega + \frac12 [\omega,\omega] = K$ 
define the horizontal smooth form 
$K\in\Omega^2(\mathcal G, \mathfrak g)$ called the {\em curvature} of the 
Cartan connection $\omega$.
The {\em curvature function} $\kappa:\mathcal G\to \wedge^2\mathfrak g_{-}^*
\otimes \mathfrak g$ is then defined by means of the
parallelism 
$$
\kappa(u)(X,Y)=K(\omega^{-1}(X)(u),\omega^{-1}(Y)(u))= 
[X,Y] - \omega([\omega^{-1}(X),\omega^{-1}(Y)])
.$$
In particular, the curvature function is valued in the cochains for the
second cohomology $H^2(\mathfrak g/\mathfrak p,\mathfrak g)$. 

The curvature vanishes if and only if the geometry is locally equivalent to its homogeneous model. 

\subsection{Parabolic geometries}
If we consider a semisimple Lie group $G$ and its parabolic subgroup $P$, we
call the Cartan geometries \emph{parabolic}.

It well known that the choice of the parabolic subalgebra $\mathfrak
p\subset \mathfrak g$ is equivalent to its grading 
$$
\mathfrak g = \mathfrak g_{-k}\oplus \dots\mathfrak g_{-1} \oplus \mathfrak
g_0 \oplus \mathfrak g_1 \oplus \dots \oplus \mathfrak g_k
$$
where $\mathfrak g_i = \mathfrak g_{-i}^*$ with respect to the Killing form.

Then there are two ways
how to split the curvature function $\kappa$ now. 
We may consider the target components $\kappa_i$
according to the values in $\mathfrak g_i$. The whole 
$\mathfrak g_-$--component 
$\kappa_-$ is called the {\em torsion} of the Cartan connection $\omega$.
The other possibility is to consider the homogeneity of the bilinear maps
$\kappa(u)$, i.e. 
$$
\kappa=\sum_{\ell=-k+2}^{3k}\kappa^{(\ell)},\quad \kappa^{(\ell)}_{|\mathfrak 
g_i\times\mathfrak
g_j}:\mathfrak 
g_i\times\mathfrak
g_j\to \mathfrak g_{i+j+\ell} 
.$$

Since we deal with semisimple algebras only, there is the
codifferential $\partial^*$ which is adjoint to the Lie algebra cohomology 
differential
$\partial$. Consequently, there is the Hodge theory
on the cochains which enables to deal very effectively with the curvatures.
In particular, we may use several restrictions on the values of the
curvature which turn out to be quite useful.

The parabolic geometry $(\mathcal G,\omega)$ with the curvature function
$\kappa$ is
called {\em flat} if $\kappa=0$, {\em torsion--free} if $\kappa_-=0$, 
{\em normal} if $\partial^*\circ\kappa=0$, and {\em regular} if
$\kappa^{(j)}=0$ for all $j\le 0$.

If all curvature $\kappa^{(j)}$ vanish for $j<0$, then the filtration
obtained from the grading of the Lie algebra and the absolute parallelism
$\omega$ is compatible with brackets of Lie vector fields, and if $k^{(0)}$
vanishes as well, then these brackets even coincide with the algebraic
bracket inherited by the absolute parallelism from $\mathfrak g_{\le0}$.  

\subsection{Regular filtrations on manifolds}
Let us fix the graded $\mathfrak g$ and $\mathfrak p$ as above.
Starting with a filtration $TM=T^{-k}M\supset \dots \supset T^{-1}M$,
assume that its
associated graded vector bundle $\operatorname{Gr}TM$, with its algebraic
bracket induced by the Lie bracket of vector fields, is pointwise isomorphic
to the negative part of the graded Lie algebra 
$\mathfrak g_{-k}\oplus\dots\oplus
\mathfrak g_{-1}$. If $\mathfrak g_0$ is smaller than the entire
$\mathfrak{gl}(\mathfrak g_{-1})$, then assume the frame bundle of
$\operatorname{Gr}TM$ has been reduced to the structure group $G_0$.
We call such filtrations {\em regular infinitesimal flag structures 
of type $\mathfrak g/\mathfrak p$}. 

\begin{thrm}
There is the bijective correspondence between the isomorphism classes of
regular normal parabolic geometries of type $G/P$ and the regular
infinitesimal flag structures of type $\mathfrak g/\frak p$ on $M$, 
except for one series of
one--graded, and one series of two--graded Lie algebras $\mathfrak g$ for which 
$H^1(\frak g_-,\frak g)$ is nonzero in homogeneous degree one.
\end{thrm} 

Although this general theorem is proved in a constructive way, cf.
\cite{CS} or \cite{Yam}, the explicit and effective construction is far
from trivial in the individual cases as soon as the grading is of length
$k\ge 2$. Thus the arguments leading to the theorem are rather serving as
guidelines for the explicit constructions.

The entire section 5 provides the links of the Cartan connection from the
above theorem to our construction in the paper. In particular, the matrix of
one-forms \eqref{sp_n11} is just the explicit expression of $\omega$,
the comparison of the structure equations for the Lie group $Sp(n+1,1)$ with
the structure equations for this $\omega$ shows the right equivariance and we have
computed there that the normality condition is satisfied too.

The uniqueness part from the latter theorem then implies that we have
constructed the canonical regular and normal Cartan connection $\omega$.

\subsection{The curvature} 
We are now in position to say more about the structure of the curvature. 
Each representation $\rho$ of the entire group $G$ on a vector space
$\mathbb V$ defines the natural bundle
$\mathcal G\times_\rho \mathbb V$ over
the manifolds $M$ with the regular infinitesimal flag structures. Moreover, the
unique extension $\tilde\omega:T\tilde{\mathcal G} \to \mathfrak g$ 
of the canonical Cartan connection $\omega$ to the extended  bundle $\tilde{
\mathcal G}=\mathcal G\times_P G$ provides the canonical covariant derivative
on all such natural bundles (as a principal connection on $\tilde {\mathcal G}$).
The adjoint representation of $G$ on $\mathfrak
g$ is the best example leading to the so called \emph{adjoint tractor bundle
$\mathcal A$} and the curvature $\kappa$ can be interpreted as a two-form on
the manifold $M$ with values in $\mathcal A$. 

The splitting of $\mathfrak g$ into irreducible $G_0$ components corresponds
to the splitting of the adjoint tractor bundle into components seen whenever
we reduce the structure group to $G_0$. In our case, we have $\mathfrak g_0= \mathfrak h\oplus \mathfrak{sp}(n)$ and 
there is the only harmonic component $\mathcal 
S_{\alpha\beta\gamma\delta}$ which corresponds to cochains
$\Lambda^2\mathfrak g_{-1}^*\otimes \mathfrak{sp}(n)$. This is the only
component of homogeneity two and all the potentially nonzero 
components of $\kappa$ as deduced
in the Proposition
\ref{curvature}, are listed in the table.

\medskip
\centerline{%
\begin{tabular}{|c|c|c|}
\hline
homogeneity & the cochains & object in Proposition \ref{curvature}
\\
\hline
2 & $\mathfrak g_{-1}\wedge \mathfrak g_{-1} \to \mathfrak{sp}(n)$ & 
$\mathcal
S_{\alpha\beta\gamma\delta}$
\\
\hline
3 & $\mathfrak g_{-2}\otimes \mathfrak g_{-1} \to \mathfrak{sp}(n)$ & 
$\mathcal
V_{\alpha\beta\gamma}$
\\
\hline
3 & $\mathfrak g_{-1}\wedge \mathfrak g_{-1} \to \mathfrak g_1$ & 
$\mathcal
V_{\alpha\beta\gamma}$
\\
\hline
4 & $\mathfrak g_{-2}\wedge \mathfrak g_{-2} \to \mathfrak{sp}(n)$ & 
$\mathcal
L_{\alpha\beta}, \mathcal M_{\alpha\beta}$
\\
\hline
4 & $\mathfrak g_{-2}\otimes \mathfrak g_{-1} \to \mathfrak g_1$ & 
$\mathcal
L_{\alpha\beta}, \mathcal M_{\alpha\beta}$
\\
\hline
4 & $\mathfrak g_{-1}\wedge \mathfrak g_{-1} \to \mathfrak g_2$ & 
$\mathcal
L_{\alpha\beta}, \mathcal M_{\alpha\beta}$
\\
\hline
5 & $\mathfrak g_{-2}\wedge \mathfrak g_{-2} \to \mathfrak g_1$ & 
$\mathcal
C_{\alpha}, \mathcal H_{\alpha}$
\\
\hline
5 & $\mathfrak g_{-2}\otimes \mathfrak g_{-1} \to \mathfrak g_2$ & 
$\mathcal
C_{\alpha}, \mathcal H_{\alpha}$
\\
\hline
6 & $\mathfrak g_{-2}\wedge \mathfrak g_{-2} \to \mathfrak g_2$ & 
$\mathcal P, \mathcal Q, \mathcal R$
\\
\hline
\end{tabular}
}
\medskip 

Notice that it is the $\partial^*\kappa=0$ normalization which enforces
several potentially different components to coincide.

Since the second Lie algebra 
cohomology $H^2(\mathfrak g/\mathfrak p,\mathfrak g)$ is completely
reducible with trivial action of $\mathfrak g_{>0}$, the harmonic curvature
components living in the kernel of both $\partial^*$ and $\partial$ are well
defined tensors on $M$. A great general result of the so called BGG calculus
reads that the entire curvature of a regular and normal Cartan connection is
obtained as the image of the harmonic part under a suitable natural linear
differential operator. 

Although we have not delivered this linear differential operator
explicitely, we came quite close in Proposition \ref{secondary}.
Indeed, notice that the expression for the differential $d\mathcal
S_{\alpha\beta\gamma\delta}$ contains known combination of the curvature
components and the new quantities identified in the proposition. A
straightforward check of the involved symmetries and relations reveals that
we can actually compute the quantities $\mathcal
V_{\alpha\beta\gamma}$ in terms of the differential of $\mathcal
S_{\alpha\beta\gamma\delta}$. Similarly, the next line allows to express
the component composed of $\mathcal
M_{\alpha\beta}$ and $\mathcal
L_{\alpha\beta}$ depending on second order derivatives of $\mathcal
S_{\alpha\beta\gamma\delta}$ and first derivatives of $\mathcal
V_{\alpha\beta\gamma}$. This goes on, until we finally express $\mathcal P$,
$\mathcal Q$, and $\mathcal R$ from the lines involving the differentials of
$\mathcal C_{\alpha}$ and $\mathcal H_{\alpha}$. The latter expression will
involve fourth derivatives of $\mathcal S_{\alpha\beta\gamma\delta}$, as
expected.


\end{document}